\documentclass[a4paper,8pt]{article}
\usepackage[T1]{fontenc}
\usepackage[utf8]{inputenc}
\usepackage[french,english]{babel}
\usepackage{amsmath}
\usepackage{amsthm}
\usepackage{amssymb}
\usepackage{makeidx}
\usepackage{graphicx}
\usepackage{enumitem}
 \usepackage[flushleft]{threeparttable}
\usepackage[hidelinks]{hyperref}
\usepackage[left=4cm,right=4cm,top=2cm,bottom=2cm]{geometry}
\usepackage[usenames,dvipsnames,svgnames,table]{xcolor}
\newtheorem{thm}{Theorem}[section]
\newtheorem{lem}[thm]{Lemma}
\newtheorem{prop}[thm]{Proposition}
\newtheorem{cor}[thm]{Corollary}
\newtheorem{res}[thm]{Result}
\theoremstyle{definition}
\newtheorem{defn}[thm]{Definition}
\theoremstyle{definition}
\newtheorem{rem}[thm]{Remark}	
\theoremstyle{definition}
\newtheorem{note}[thm]{Notation}
\theoremstyle{definition}
\newtheorem{exa}[thm]{Example}
\usepackage{tikz}
\usepackage{tikz-cd}
\usetikzlibrary{arrows,shapes,positioning}
\usetikzlibrary{decorations.markings}
\tikzstyle arrowstyle=[scale=1]
\usepackage{lipsum}
\usepackage{url}
\usepackage{imakeidx}
\makeindex[title=Index of definitions and notations,intoc]

\begin{document}
\title{Real algebraic curves on real del Pezzo surfaces}
\date{ }
\author{Matilde Manzaroli}

\maketitle
\begin{abstract}
The study of the topology of real algebraic varieties dates back to the work of Harnack, Klein and Hilbert in the 19th century; in particular, the isotopy type classification of real algebraic curves in real toric surfaces is a classical subject that has undergone considerable evolution. On the other hand, not much is known for more general ambient surfaces. We take a step forward in the study of topological type classification of real algebraic curves on non-toric surfaces focusing on real del Pezzo surfaces of degree $1$ and $2$ with multi-components real part. We use degeneration methods and real enumerative geometry in combination with variations of classical methods to give obstructions to the existence of topological type classes realised by real algebraic curves and to give constructions of real algebraic curves with prescribed topology.
\end{abstract}
\tableofcontents
\section{Introduction}
\label{sec: intro}
The study of topology of real algebraic varieties dates back to the work of Harnack, Klein and Hilbert in the 19-th century (\cite{Harn76}, \cite{Hilb02}, \cite{Klei73}). A \textit{real algebraic variety} $(X,\sigma)$ is a compact complex algebraic variety equipped with an anti-holomorphic involution $\sigma:X \rightarrow X$, called real structure. The real part $\mathbb{R}X$ of $(X, \sigma)$ is the set of points fixed by the involution $\sigma$. 
Hilbert proposed in the first part of his 16-th problem to classify the isotopy types of real algebraic curves of degree $6$ in $\mathbb{R}P^2$, respectively of real algebraic surfaces of degree $4$ in $\mathbb{R}P^3$. The classification on $\mathbb{R}P^2$ had been achieved by Gudkov (\cite{Gudk69}) at the end of the 60's and the other one by Kharlamov (\cite{Khar76}, \cite{Khar78}) around ten years later. At the moment, the classification of real algebraic surfaces in the real projective space is known only up to degree $4$, respectively that of real algebraic plane curves up to degree $7$ (\cite{Viro84a}, \cite{Viro84b}). \\
There are two main directions in the study of the topology of real algebraic varieties. The first is to give obstructions to the existence of topological type classes realised by real algebraic varieties. The second direction is to provide constructions of real algebraic varieties with prescribed topology. From the 70's, especially thanks to the work of Arnold and Rokhlin (\cite{Arno71}, \cite{Rokh72}, \cite{Rokh74}, \cite{Rokh78}, \cite{Rokh80}), many general obstructions had been discovered. On the other hand, the construction techniques had remained relatively elementary for a long time. In 1979, Viro provided a breakthrough in the construction direction by inventing the \textit{patchworking method} (\cite{Viro84a}, \cite{Viro84b}). Such method and its generalizations still remain the most powerful tools to construct real algebraic hypersurfaces with prescribed topology in real algebraic toric varieties with the standard real structure. Exploiting the patchworking technique, several construction results have been achieved on real algebraic toric varieties. On the other hand, only few works have been devoted to classification problems on real non-toric varieties (\cite{Mikh98}) or on toric varieties with the non-standard real structure (for example, \cite{GudShu80}, \cite{Mikh94}, \cite{DegZvo99}, \cite{Manz20}). 
\begin{defn}
\label{defn: ksphere} 
Let $(X,\sigma)$ be a real del Pezzo surface. We say that $X$ is a $k$-\textit{sphere real del Pezzo surface} if:
\begin{enumerate}[label=(\roman*)]
\item $X$ has degree $2$ and $\mathbb{R}X$ is homeomorphic to $\bigsqcup_{j=1}^{k} S^2 $, with $1 \leq k \leq 4$ ;
\item $X$ has degree $1$ and $\mathbb{R}X$ is homeomorphic to $\mathbb{R}P^{2} \sqcup \bigsqcup_{j=1}^{k}  S^2 $, with $0 \leq k \leq 4$.
\end{enumerate}
\end{defn}
We take a step forward in the classification of embedded topology of real algebraic curves on real non-toric surfaces studying the  topological type classifications of real algebraic curves on $k$-sphere real del Pezzo surfaces of degree $1$ and $2$.\\
In 1998, Mikhalkin (\cite{Mikh98}) was the first to face the issue of classifying real algebraic curves on real algebraic surfaces with non-connected real part. In particular, he studied the topology of the real part of transverse intersections of real quadric surfaces with real cubic surfaces in $\mathbb{C}P^3$ equipped with the standard real structure. To our knowledge, there are no other classifications on real algebraic surfaces with non-connected real part. 
\subsection{Classification and generalities}
We look at real curves in real $k$-sphere del Pezzo surfaces whose homology class is determined by an integer, called \textit{class}. 
\begin{defn}
\label{defn: curve_d}
Let $(X, \sigma)$ be a real del Pezzo surface and let $A\subset X$ be a real algebraic curve. Then, we say that $A$ has \textit{class} $d$ on $X$ if $A$ realizes $dc_1(X)$ in $H_2(X; \mathbb{Z})$, where $c_1(X)$ is the anti-canonical class of $X$.
\end{defn} 
\begin{defn}
\label{defn: homeo}
Let $\bigsqcup_{i=1}^{l} B_{i}$ and $\bigsqcup_{i=1}^{l} B_{i}'$ be two collection of $l$ circles embedded in the disjoint union $V$ of $k$ spheres (and a real projective plane). 
\begin{itemize}
\item We say that the pairs $(V , \bigsqcup_{i=1}^{l} B_{i})$ and $(V, \bigsqcup_{i=1}^{l} B_{i}')$ are \textit{homeomorphic} if there exists a homeomorphism $f:V \rightarrow V$ such that $f(\bigsqcup_{i=1}^{l} B_{i})=\bigsqcup_{i=1}^{l} B_{i}'$.
\item We call \textit{topological type} any arrangement realized by a pair $(V , \bigsqcup_{i=1}^{l} B_{i})$.
\item Fix a non-negative integer $d$. We say that a topological type is a \textit{topological type in class $d$} if
\begin{enumerate}[label=(\roman*)]
\item $V$ is the disjoint union of $k$ spheres and $l$ is bounded by $d(d-1)+2$;
\item $V$ is the disjoint union of $k$ spheres and a real projective plane, and $l$ is bounded by $\frac{d(d-1)}{2}+2$.
\end{enumerate}
\end{itemize}
\end{defn}

Since $k$-sphere real del Pezzo surfaces of degree $1$ or $2$ have moduli, the classification of topological types in class $d$ on one surface, up to homeomorphism, may differ from that on another surface in the same deformation family. 
In this paper, a topological type $\mathcal{S}$ is said to be \textit{realizable in class} $d$ if there exist a $k$-sphere real del Pezzo surface $X^{k}$ of degree $1$ (resp. $2$) and a non-singular real algebraic curve $A\subset X^{k}$ of class $d$ such that the pair $(\mathbb{R}X^{k}, \mathbb{R}A)$ realizes $\mathcal{S}$.\\
\par Let us present some general definitions and known results about real algebraic curves. Let $(X,\sigma)$ be any real algebraic non-singular compact curve. A very useful tool is Harnack-Klein's inequality (\cite{Harn76}, \cite{Klei73}), which bounds the number $l$ of connected components of $\mathbb{R}X$ by the genus $g$ of $X$ plus one. We say that $(X, \sigma)$ is a \textit{$M$-curve} or a \textit{maximal} curve, if $l=g+1$. If $X \setminus \mathbb{R}X$ is connected, we say that $X$ is of \textit{type II} or \textit{non-separating}, otherwise of \textit{type I} or \textit{separating} (\cite{Klei73}). Looking at the real part of the curve and its position with respect to its complexification gives us information about $l$ and viceversa. For example, we know that if $X$ is maximal, then $X$ is of type I. Or, if $X$ is of type I then $l$ has the parity of $g+1$. Moreover, if $X$ is of type I, the two  halves of $X\setminus \mathbb{R}X$ induce two opposite orientations on $\mathbb{R}X$ called \textit{complex orientations} of the curve (\cite{Rokh72}).\\
\par Harnack-Klein's inequality, in combination with the adjunction formula, gives obstructions on the number of connected components of any embedded real algebraic curve. For any non-negative integer $d$, such obstruction has been taken into account in Definition \ref{defn: homeo} when we define topological types in class $d$ in the disjoint union of $k$ spheres (and a real projective plane).\\

\par The anti-(bi)canonical system of a $k$-sphere real del Pezzo surface of degree $2$, respectively degree $1$ is a double cover of $\mathbb{C}P^2$ ramified along non-singular real quartic, respectively a double cover of a quadratic cone in $\mathbb{C}P^3$ ramified along a non-singular real cubic section and the vertex. Conversely, any such double cover yields a del Pezzo surface of degree $2$, respectively $1$. We mainly focus on the  topological type classification of real algebraic curves in $4$-sphere real del Pezzo surfaces of degree $2$ and of degree $1$. Later, we apply the classification tools developed for $4$-sphere del Pezzo surfaces to $k$-sphere del Pezzo surfaces, with $k < 4$. 
In combination with variations of classical classification methods, the main tools of construction of real curves rely on degeneration methods and those of obstruction of topological types rely on real enumerative geometry. The classifications in $k$-sphere real del Pezzo surfaces of degree $2$ are in Section \ref{sec: DP2} and the classifications in those of degree $1$ in Section \ref{sec: DP1}.
\subsection{$k$-sphere real degree $2$ del Pezzo surfaces}
\label{subsec: DP2} 
 Let us denote by $X^k$ a $k$-sphere del Pezzo surface of degree $2$, with $1 \leq k \leq 4$. \\
 In Sections \ref{subsec: def_DP2} - \ref{subsec: real_schemes_DP2} we give definitions, notations and state the main results about $X^k$. The proofs of the main statements are in Sections \ref{subsec: obstr_DP2}, \ref{subsec: class 1 and 2_DP2}, \ref{subsec: class_3_DP2}, \ref{subsec: symplectic} and \ref{subsec: final_constr_DP2}. Moreover, Section \ref{subsec: difficult_constr_DP2} is devoted to the construction tools used in Section \ref{subsec: final_constr_DP2}.
\par We say that a real algebraic curve in $X^k$ is \textit{symmetric} if it can be realized as lifting of a degree $d$ real plane curve via the anti-canonical map of $X^k$. Moreover, we call \textit{symmetric} a topological type in class $d$ in $\mathbb{R}X^k$ if it is realizable by a symmetric real algebraic curve of class $d$ in $X^k$; Definition \ref{defn: non-symm}. \\
We have a complete classification for class $1$ and $2$. 
\begin{res}[Proposition \ref{prop: d=1,d=2_DP2}]
\label{res: prop: d=1,d=2_DP2}
For any topological type $\mathcal{S}$ in class $d=1,2$, there exist some $X^k$ and a symmetric real algebraic curve of class $d$ in $X^k$ realizing $\mathcal{S}$.
\end{res}
Fixed a topological type $S$ in class $d=1,2$. One can construct a plane degree $d$ real curve $C$ and a plane real quartic $\tilde{Q}$ arranged in $\mathbb{R}P^{2}$ so that $\mathcal{S}$ is realized in $X^{k}$ by a real curve $A \subset X^{k}$ of class $d$, where $X^{k}$ is the double cover of $\mathbb{C}P^{2}$ ramified along $\tilde{Q}$ and $A$ is the lifting of $C$. \\
Later, we focus on real algebraic curves of class $d\geq 3$. First of all, Harnack-Klein's inequality does not give a complete set of restrictions anymore. Besides an application of Comessatti-Petrovsky inequality (\cite{Come28}, \cite{Petro33}, \cite{Petro38}) for real curves of even class $d$ (Proposition \ref{prop: petro}), most of all classical obstructions do not seem to apply; for example we did not find applications of the congruence results of \cite{Rokh72} and \cite{GuiMar77} to our setting.
\par In Section \ref{subsec: obstr_DP2}, in order to give new obstructions for topological types for every integer $d$, we use a variation of Bézout's type restrictions exploiting Welschinger invariants (\cite{Wels05}, \cite{Shus15}, \cite{IteKhaShu15}), from real enumerative geometry.
\par For $d=3$, dealing with real maximal curves only, we are able to obtain a partial classification on $4$-sphere real del Pezzo surfaces of degree $2$.
\begin{res}[Theorem \ref{thm: principal_DP2}]
\label{res: main_DP2}  
There are $74$ topological types with $8$ connected components in class $3$ which are not prohibited by Bézout's type restrictions. Moreover $48$ among the $74$ topological types are such that: for each of them there exist some $X^4$ and a real maximal curve of class $3$ in $X^4$ realizing it. In addition $19$ out of $48$ are realized by real symmetric curves.
\end{res}
There is one topological type in class $3$ that we can realize by a non-singular real symplectic curve (Proposition \ref{prop: symplectic_curve_surface}) in a real \textit{symplectic} degree $2$ del Pezzo surface with real part composed by $4$ spheres. We do not know yet whether this topological type is realizable algebraically. \\
\par Using the obstructions in Section \ref{subsec: obstr_DP2} and combining the construction methods previously adopted to the case of $X^4$, we obtain the following result on $X^k$ with $k < 4$.
\begin{res}[Proposition \ref{prop: principal_DP2}]
\label{res: main_prop_DP2} 
There are respectively $79$, $61$ and $28$ topological types with $8$ connected components in class $3$ which are not prohibited by Bézout's type restrictions. Moreover $49$, $38$ and $17$ respectively among the $79$, $61$ and $28$ topological types are such that: for each of them there exist some $X^k$ and real maximal curve of class $3$ in $X^k$ realizing it, respectively with $1 \leq k \leq 3$. In addition $6$ out of respectively $49$, $38$ and $2$ out of $17$ are realized by real symmetric curves.
\end{res}
Finally, we deal with non-symmetric topological types in class $d$ on $X^{4}$. We do not know yet if there are non-symmetric topological types realizable in class $3$ and $4$ on $X^{4}$. But, we have the following result.
\begin{res}[Proposition \ref{prop: non-symm_real_scheme}]
\label{res: non-symm}  
For any integer $d\geq 5$, there exists a topological type $\mathcal{S}$ in class $d$ (consisting of $2d+1$ connected components) such that there exists some $X^4$ and a non-symmetric real algebraic curve of class $d$ in $X^4$ realizing $\mathcal{S}$. Furthermore $\mathcal{S}$ is not realizable by any symmetric real algebraic curve of class $d$ in any $X^4$. 
\end{res} 
All symmetric topological types in class $3$ 
are realized using construction techniques similar to those used for class $1$ and $2$ (Result \ref{res: prop: d=1,d=2_DP2}). Such construction methods do not seem to be enough to realize all topological types in class $d \geq 3$. Indeed, we use degeneration methods to realize some topological types in class $3$ and some (non-symmetric) topological types in class $d \geq 5$; see Proposition \ref{prop: star_label_DP2} and the end of Section \ref{subsec: final_constr_DP2}. Let us give a rough idea of the construction technique. 
\par We start with a real degree $2$ del Pezzo surface $X_0$ with a real non-degenerate double point as only singularity and real part composed by $k$ two-dimensional connected components, with $1 \leq k \leq 3$. 
Then, we degenerate $X_0$ to the union of a real ruled surface $T$ and a $k$-sphere real del Pezzo surface $S$ intersecting transversely along a curve $E$ (Proposition \ref{prop: existence_DP2_family}).
We construct real algebraic curves $C_T, C_S$ of given topology and homology class separately on $T$ and on $S$; Section \ref{subsubsec: intermediate_construction_DP2}. Moreover $C_{T}$ and $C_{S}$ have to intersect $E$ in the same collection of points. Then, to end the construction, we use the version of patchworking developed by Shustin and Tyomkin (\cite{ShuTyo06I}, \cite{ShuTyo06II}) which allows us, under some transversality conditions, to "glue" such surfaces and curves to realize real algebraic curves on some $X^k$, respectively on some $X^{k+1}$, 
with topology prescribed by the topological type realized by the triplet $(\mathbb{R}T \cup \mathbb{R} S, \mathbb{R}E, \mathbb{R}C_T \cup \mathbb{R}C_S)$; see Theorem \ref{thm: weak_patch_DP2}. The patchworking technique presented in \cite{ShuTyo06I}, \cite{ShuTyo06II} has been recently exploited in \cite{BDIM18} to construct real algebraic curves whose real part consists of a finite number of points in $\mathbb{C}P^2$ and in the quadric ellipsoid.
\subsection{$k$-sphere real degree $1$ del Pezzo surfaces}
\label{subsec: DP1}
Let us denote by $Y^k$ a $k$-sphere del Pezzo surfaces of degree $1$, with $0 \leq k \leq 4$. \\
In Sections \ref{subsec: intro_DP1}, \ref{subsec: real_schemes_DP1}, \ref{subsec: positive_negative_DP1}, \ref{subsec: main_DP1} we give definitions, notations and state the main results about $k$-sphere real del Pezzo surfaces of degree $1$. The proofs are in Sections \ref{subsec: obstructions_DP1} and \ref{subsec: construction_DP1}.
The obstruction presented in Proposition \ref{prop: pseudolines}, called $\mathcal{J}$-\textit{obstruction},  gives a complete set of restrictions for topological types up to class $3$.
\begin{res}[Proposition \ref{prop: delpezzo_1}]
\label{res: delpezzo_1} 
For any topological type $\mathcal{S}$ in class $d=1,2,3$, which is not prohibited by the $\mathcal{J}$-obstruction, there exist some $Y^k$ and a real algebraic curve of class $d$ in $Y^k$ realizing $\mathcal{S}$.
\end{res}
Moreover, in the case of $Y^4$, one can define a notion of positivity of the spheres inducing a refined classification (Section \ref{subsec: positive_negative_DP1}).
Up to class $3$, we show that the refined and non-refined classifications are the same. Let us consider the disjoint union of a real projective plane and $4$ spheres. Label two spheres as positive $S^2_+$, the others as negative $S^2_-$ and define 
$$V^{+}:= S^2_+ \sqcup S^2_+ \text{ and } V^{-}:= S^2_- \sqcup S^2_-.$$ 
\begin{defn}
\label{defn: homeorefined}
Let $\bigsqcup_{i=1}^{l} B_{i}$ and $\bigsqcup_{i=1}^{l} B_{i}'$ be two collection of $l$ circles embedded in $V:= \mathbb{R}P^{2}\sqcup V^{+} \sqcup V^{-}$.
\begin{enumerate}
\item We say that the pairs $(V, \bigsqcup_{i=1}^{l}  B_{i})$ and $(V , \bigsqcup_{i=1}^{l} B_{i}')$ are \textit{refined homeomorphic} if they are homeomorphic via a homeomorphism $f:V \rightarrow V$ such that $f(V^{\pm})=V^{\pm}$.
\item A topological type in $V$ up to refined homeomorphism, is called a \textit{refined topological type}.
\end{enumerate}
\end{defn}

\begin{res}[Theorem \ref{thm: delpezzo_1}]
\label{res: theorem_delpezzo_1} 
For any refined topological type $\mathcal{S}$ in class $d=1,2,3$, which is not prohibited by the $\mathcal{J}$-obstruction, there exist some $Y^4$ and a real algebraic curve of class $d$ in $Y^4$ realizing $\mathcal{S}$.
\end{res}
Furthermore, Proposition \ref{prop: bezout_del_pezzo_deg_1} gives Bézout-type restrictions for class $d\geq 4$.
\section{Preliminaries}
\subsection{Encoding topological types and real schemes}
\label{subsec: cha2_codage_isotopie}
Let $X$ be a real algebraic surface equipped with a real structure $\sigma: X \rightarrow X$. Let $\sigma_* :H_2(X;\mathbb{Z})\rightarrow H_2(X;\mathbb{Z})$ be the group homomorphism induced by $\sigma$ and let $H_2^-(X;\mathbb{Z})$ be the $(-1)$-eigenspace of $\sigma_*$. In the following, for a fixed homology class $\alpha \in H_2^-(X; \mathbb{Z})$, we are interested in the classification of the topological types of the pair $(\mathbb{R}X, \mathbb{R}A)$ up to homeomorphism, where $A \subset X$ is a non-singular real algebraic curve realizing $\alpha$ in $H_2(X; \mathbb{Z})$. The real part of $A$ is homeomorphic to a union of circles 
embedded in $\mathbb{R}X$, and can be embedded in $\mathbb{R}X$ in different ways. For the purpose of this paper, we only need to explain how to encode the embedding of a given collection $\bigsqcup_{i=1,..,l}B_i$ of $l$ disjoint circles in $\mathbb{R}P^2$ and in $S^2$. An embedded circle realizing the trivial-class in $H_1(\mathbb{R}P^2; \mathbb{Z}/2\mathbb{Z})$ or $H_1(S^2; \mathbb{Z})$ is called \textit{oval}, otherwise is called \textit{pseudo-line}. \\
\par An oval in $\mathbb{R}P^{2}$ separates two disjoint non-homeomorphic connected components: 
the connected component homeomorphic to a disk is called \textit{interior} of the oval; the other one is called \textit{exterior} of the oval. For each pair of ovals, if one is in the interior of the other we speak about an \textit{injective pair}, otherwise a \textit{non-injective pair}. \\
On the other hand, an oval in $S^{2}$ bounds two disks; therefore, on $S^{2}$ interior and exterior of an oval are not well defined. It follows that the encoding on $S^{2}$ is not well defined either and it depends on the choice of a point $p$ on $S^{2} \setminus\bigsqcup_{i=1,..,l}B_i$. Let us take  $S^2$ deprived of $p$, which is homeomorphic to $\mathbb{R}^{2}$ and  let us call \textit{oval} any circle embedded in $\mathbb{R}^2$. Analogously to the case of $\mathbb{R}P^{2}$, in $\mathbb{R}^{2}$ one define interior and exterior of an oval and (non-)injective pairs for each pair of ovals. \\
\par We shall adopt the following notation to encode a given topological type realized by a pair $(V,\bigsqcup_{i=1,..,l}B_i)$, where $V$ is $\mathbb{R}^2$ or $\mathbb{R}P^2$. 
\begin{note}
An empty union of ovals is denoted by $0$. We say that a union of $l$ ovals realizes $l$ if there are no injective pairs. The symbol $\langle \mathcal{S}\rangle$ denotes the disjoint union of a non-empty collection of ovals realizing $\mathcal{S}$, and an oval forming an injective pair with each oval of the collection. \\
We use the following notation only in Proposition \ref{prop: non-symm_real_scheme}, Proposition \ref{prop: S_pencil_quartic_DP2} and in the proof of the Proposition \ref{prop: non-symm_real_scheme} in Section \ref{subsec: final_constr_DP2}. Let $h$ be a non-negative integer. The symbol $N(h,\mathcal{S})$ denotes 
\begin{itemize}[label=\textbullet]
\item $\mathcal{S}$, if $h=0$;
\item $\langle   N(h-1,\mathcal{S})  \rangle $, 
otherwise.
\end{itemize}
Finally, the disjoint union of any two collections of ovals, realizing respectively $\mathcal{S}'$ and $\mathcal{S}''$ in $V$, is denoted by $\mathcal{S}'  \sqcup  \mathcal{S}''$ if none of the ovals of one collection forms an injective pair with the ovals of the other one and they are both non-empty collections. A disjoint union of the form $\mathcal{S}' \sqcup 0$ is still denoted $\mathcal{S}'$. Moreover, a pseudo-line in $\mathbb{R}P^2$ is denoted by $\mathcal{J}$.
\end{note}
\begin{defn}
\label{defn: realizing_real_scheme}
\begin{itemize}
\item[]
\item We say that the pair $(S^2, \bigsqcup_{i=1,..,l} B_i)$ realizes $\mathcal{S}$ if there exists a point $p\in S^2 \setminus \bigsqcup_{i=1,..,l}B_i$ such that $(S^2 \setminus \{p\},  \bigsqcup_{i=1,..,l}B_i)$ realizes $\mathcal{S}$.
\item Let $(\mathbb{R}P^2, \bigsqcup_{i=1,..,l} B_i)$ and $(V, \bigsqcup_{i=1,..,l'} B_i')$ be pairs respectively realizing $\mathcal{S}$ and $\mathcal{S}'$, where $V$ is a disjoint union of $2$-spheres. We say that the pair $(\mathbb{R}P^2 \sqcup V, \bigsqcup_{i=1,..,l} B_i\sqcup \bigsqcup_{i=1,..,l'} B_i')$ realizes $\mathcal{S}   |   \mathcal{S}'$.
\item Let $(S^2, \bigsqcup_{i=1,..,l} B_i)$ and $(S^2, \bigsqcup_{i=1,..,l'} B_i')$ be pairs respectively realizing $\mathcal{S}$ and $\mathcal{S}'$. We say that the pair $(S^2 \sqcup S^2, \bigsqcup_{i=1,..,l} B_i\sqcup \bigsqcup_{i=1,..,l'} B_i')$ realizes $\mathcal{S}   :   \mathcal{S}'$.
\end{itemize}
\end{defn}
\begin{defn}
\label{defn: real_scheme}
Let $(X,\sigma)$ be a real algebraic surface. A topological type $\mathcal{S}$ in $\mathbb{R}X$, up to homeomorphism, is called \textit{real scheme}. Let $A \subset X$ be a real curve. We say that $A$ has real scheme $\mathcal{S}$ if the pair $(\mathbb{R}X, \mathbb{R}A)$ realizes $\mathcal{S}$, up to homeomorphism.
\end{defn}
Finally, we need some more definitions for particular collections of ovals.
\begin{defn}
\label{defn: nest_projective_plane}
A collection of $h$ ovals in $\mathbb{R}P^2$ is called a \textit{nest} of depth $h$ if any two ovals of the collection form an injective pair. Let $N_{1}$ and $N_{2}$ be two nests of depth $i_1$ and $i_2$ in $\mathbb{R}P^2$. We say that the nests are disjoint if each pair of ovals, composed by an oval of $N_1$ and an oval of $N_2$, is non-injective.
\end{defn}
\begin{defn}
\label{defn: nest_sphere}
A collection $N_h$ of $h$ ovals in $S^2$ is a \textit{nest} if each connected component of $S^2 \setminus N_h$ is either a disk or an annulus.\\
Let $N_{i_k}$ be $k$ nests of depth $i_k$ in $S^2$, with $k\geq 3$. We say that the nests are \textit{disjoint} if a disk of $S^2 \setminus N_{i_j}$ contains all other $k-1$ nests, for all $j\in \{1,..,k\}$.
\end{defn}
\subsection{Hirzebruch surfaces}
\label{subsec: cha2_Hirzebruch_surf}
A Hirzebruch surface is a compact complex surface which admits a holomorphic fibration over $\mathbb{C}P^1$ with fiber $\mathbb{C}P^1$ (\cite{Beau83}). Every Hirzebruch surface is biholomorphic to exactly one of the surfaces $\Sigma_n=\mathbb{P}(\mathcal{O}_{\mathbb{C}P^1}(n)\oplus \mathbb{C})$ for $n \geq 0$. The surface $\Sigma_n$ admits a natural fibration $$\pi_n: \Sigma_n \rightarrow \mathbb{C}P^1$$ with fiber $\mathbb{C}P^1=:F_n$. Denote by $B_n$, resp. $E_n$, the section $\mathbb{P}(\mathcal{O}_{\mathbb{C}P^1}(n)\oplus \{0\})$, resp. $\mathbb{P}(\{0\}\oplus \mathbb{C})$. The self-intersection of $B_n$ (resp. $E_n$ and $F_n$) is $n$ (resp. $-n$ and $0$). When $n \geq 1$, the exceptional divisor $E_n$ determines uniquely the Hirzebruch surface since it is the only irreducible and reduced algebraic curve in $\Sigma_n$ with negative self-intersection.\\
\par For example $\Sigma_0=\mathbb{C}P^1\times \mathbb{C}P^1$. The Hirzebruch surface $\Sigma_1$ is the complex projective plane blown-up at a point, and $\Sigma_2$ is the quadratic cone with equation $Q: X^2+ Y^2-Z^2=0$ blown-up at the node in $\mathbb{C}P^3$. The fibration of $\Sigma_2$ (resp. of $\Sigma_1$) is the extension of the projection from the blown-up point to a hyperplane section (resp. to a line) which does not pass through the blown-up point.\\
\par The group $H_2(\Sigma_n; \mathbb{Z})$ is isomorphic to $\mathbb{Z} \oplus \mathbb{Z}$ and is generated by the classes $[ B_n ]$ and $[ F_n ]$. An algebraic curve C in $\Sigma_n$ is said to be of bidegree $(a,b)$ if it realizes the homology class $a[ B_n ] + b[ F_n ]$ in $H_2(\Sigma_n; \mathbb{Z})$. Note that $[E_n]=[B_n]-n[F_n]$ in $H_2(\Sigma_n; \mathbb{Z})$. An algebraic curve of bidegree $(3,0)$ on $\Sigma_n$ is called a \textit{trigonal curve}.\\
\par We can obtain $\Sigma_{n+1}$ from $\Sigma_n$ via a birational transformation $\beta^p_n: \Sigma_n \dashrightarrow \Sigma_{n+1}$ which is the composition of a blow-up at a point $p\in E_n \subset \Sigma_n$ and a blow-down of the strict transform of the fiber $\pi^{-1}_n(\pi_n(p))$.\\
\par The surface $\Sigma_n$ is also the projective toric surface which corresponds to the polygon of vertices $(0,0), (0,1),$ $(1,1), (n+1,0)$. The Newton polygon of an algebraic curve $C$ of bidegree $(a,b)$ on $\Sigma_n$, lies inside the trapeze with vertices $(0,0), (0,a), (b,a), (an+b,0)$. The surface $\Sigma_n$ is canonically endowed by a real structure induced by the standard complex conjugation in $(\mathbb{C^*})^2$. 
For this real structure $\mathbb{R}\Sigma_n$ is a torus if $n$ is even and a Klein bottle if $n$ is odd. We will depict $\mathbb{R}\Sigma_n$ as a quadrangle whose opposite sides are identified in a suitable way. Moreover, the horizontal sides will represent $\mathbb{R}E_n$. Furthermore, let $C$ be any type I real algebraic curve in $\Sigma_n$, the depicted orientation on $\mathbb{R}C$ will denote a complex orientation of the curve.\\
The restriction of $\pi_n$ to $\mathbb{R}\Sigma_n$ defines an $S^1$-bundle over $S^1$ that we denote by $\mathcal{L}$. We are interested in the isotopy types with respect to $\mathcal{L}$ of real algebraic curves in $\mathbb{R}\Sigma_n$. 
\begin{defn}
\begin{itemize}
\item[]
\item Let $\eta$ be an arrangement of circles and points immersed in $\mathbb{R}\Sigma_n$ such that for any immersed point there exists a line of $\mathcal{L}$ intersecting the point with multiplicity $2$ and the multiplicity of intersection at each point of an immersed circle with the lines of $\mathcal{L}$ is at most $2$. Such an arrangement, up to homeomorphism, is called real scheme. We say that a real algebraic curve $C \in \mathbb{R}\Sigma_n$ has real scheme $\eta$ if the pair $(\mathbb{R}\Sigma_n, \mathbb{R}C)$ realizes $\eta$, up to homeomorphism.
\item Two real schemes in $\mathbb{R}\Sigma_n$ are \textit{$\mathcal{L}$-isotopic} if there exists an isotopy of $\mathbb{R}\Sigma_n$ which brings one arrangement to the other, each line of $\mathcal{L}$ to another line of $\mathcal{L}$ and whose restriction to $\mathbb{R}E_n$ is an isotopy of $\mathbb{R}E_n$. 
\item A real scheme in $\mathbb{R}\Sigma_n$ up to $\mathcal{L}$-isotopy of $\mathbb{R}\Sigma_n$ is called an \textit{$\mathcal{L}$-scheme}.
\item An $\mathcal{L}$-scheme $\eta$ is \textit{realizable} by a real algebraic curve of bidegree $(a,b)$ in $\Sigma_n$ if there exists such a curve 
whose real part is $\mathcal{L}$-isotopic to $\eta$. 
\item A \textit{trigonal} $\mathcal{L}$-scheme is an $\mathcal{L}$-scheme in $\mathbb{R}\Sigma_n$ which intersects each fiber in 1 or 3 real points counted with multiplicities and which does not intersect $\mathbb{R}E_n$.
\item A trigonal $\mathcal{L}$-scheme $\eta$ in $\mathbb{R}\Sigma_n$ is \textit{hyperbolic} if it intersects each fiber in $3$ real points counted with multiplicities.
\end{itemize}
\end{defn}
\subsection{Dessins d'enfants}
\label{subsec: cha2_dess_enf}
Orevkov in \cite{Orev03} has formulated the existence of real algebraic trigonal curves realizing a given trigonal $\mathcal{L}$-scheme in $\mathbb{R}\Sigma_n$ in terms of the existence of a graph on $\mathbb{C}P^1$. In the proof of Lemma \ref{lem: curves_sigma1_DP2} and of Proposition \ref{prop: (2,2)}, we use this construction technique.
\begin{defn}
\label{defn: complete_graph}
Let $n$ be a fixed positive integer. We say that a graph $\Gamma$ is a \textit{real trigonal graph of degree} $n$ if 
\begin{itemize}
\item it is a finite oriented connected graph embedded in $\mathbb{C}P^1$, invariant under the standard complex conjugation of $\mathbb{C}P^1$;
\item it is decorated with the following additional structure:
\begin{itemize}
\item every edge of $\Gamma$ is colored solid, bold or dotted;
\item every vertex of $\Gamma$ is $\bullet$, $\circ$, $\times$ (said \textit{essential vertices}) or monochrome
\end{itemize}
and satisfying the following conditions:
\begin{enumerate}
\item $\mathbb{P}_{\mathbb{R}}^{1}$ is a union of vertices and edges of $\Gamma$;
\item any vertex is incident to an even number of edges; moreover, any $\circ$-vertex (resp. $\bullet$-vertex) to a multiple of $4$ (resp. $6$) number of edges;
\item  for each type of essential vertices, the total sum of edges incident to the vertices of a same type is $12n$;
\item there are no monochrome cycles;
\item the orientations of the edges of $\Gamma$ form an orientation of $\partial (\mathbb{C}P^1\setminus \Gamma)$ which is compatible with an orientation of $\mathbb{C}P^1\setminus \Gamma$ (see Fig. \ref{fig: vertices_gamma});
\item all edges incidents to a monochrome vertex have the same color;
\item $\times$-vertices are incident to incoming solid edges and outgoing dotted edges;
\item $\circ$-vertices are incident to incoming dotted edges and outgoing bold edges;
\item $\bullet$-vertices are incident to incoming bold edges and outgoing solid  edges.
\end{enumerate}
\end{itemize}
\end{defn}
Let $n$ be a positive integer and let $c(x,y)=y^3+b_2(x)y+b_3(x)$ be a real polynomial, where $b_i(x)$ has degree $in$ in $x$. By a suitable change of coordinates in $\Sigma_n$, any trigonal curve $C$ in $\Sigma_n$ can be put into this form. Denote by $\Delta=-4b_2^3+27b_3^2$ the discriminant of $c(x,y)$ with respect to the variable $y$. The knowledge of the arrangement of the real roots of the real polynomials $\Delta=-4b_2^3+27b_3^2$, $27b_3^2$ and $-4b_2^3$ in $\mathbb{R}\Sigma_n$ allows to recover the trigonal $\mathcal{L}$-scheme realized by $C$ in $\mathbb{R}\Sigma_n$. Let $f: \mathbb{C}P^1 \rightarrow \mathbb{C}P^1$ be the homogenized discriminant, i.e. the rational function defined by $f:=\frac{\Delta}{27b_3^2}$. Orevkov's method allows to construct real polynomials $c(x,y)$ which have prescribed arrangements of the real roots and the construction is based on a consideration of the arrangement of the graph given by $f^{-1}(\mathbb{R}P^1)$ with the coloring and orientation induced by those of $\mathbb{R}P^1$ as depicted in Fig.\ref{fig: coloring_proj_line_locally}.1. In this section, we only give an algorithmic way to encode any trigonal $\mathcal{L}$-scheme in $\mathbb{R}\Sigma_n$ into a colored oriented graph on $\mathbb{R}P^1 \subset \mathbb{C}P^1$ just looking at the intersections of the fibers of $\mathcal{L}$ with $\eta$; for details see \cite{Bruga07},  \cite{Degt12}, \cite{Orev03}.
\begin{figure}[h!]
\begin{center}
\begin{picture}(100,40)
\put(-92,-13){Fig. \ref{fig: coloring_proj_line_locally}.1: Colored oriented $\mathbb{R}P^1$}
\put(-28,3){\small{$0$}}
\put(-65,3){\small{$\infty$}}
\put(3,3){\small{$1$}}
\put(82,-13){Fig. \ref{fig: coloring_proj_line_locally}.2: $\mathcal{L}_{I}$ in $\mathbb{R}\Sigma_n$}
\put(109,37){$\overbrace{\quad \quad \quad}$}
\put(118,47){$\mathcal{L}_{I}$}
\put(-138,-8){\includegraphics[width=1\textwidth]{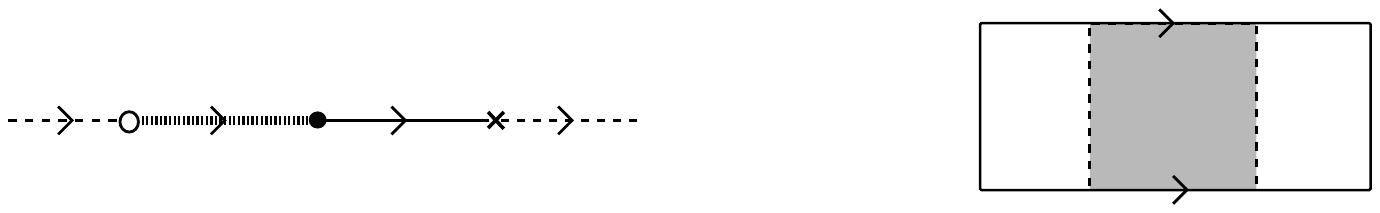}}
\end{picture}
\end{center}
\caption{ }
\label{fig: coloring_proj_line_locally}
\end{figure} 
\begin{defn}
\label{defn: completable_graph}
\begin{itemize}
Let $\eta$ be a trigonal $\mathcal{L}$-scheme 
in $\mathbb{R}\Sigma_n$.   
For any fixed interval of points $I:=\{(x,\overline{y}): \overline{y} \in \mathbb{R}, \quad x \in [a, a+b] \subset \mathbb{R}\} \subset \mathbb{R}\Sigma_n$, we denote with $\mathcal{L}_{I}$ the fibers of $\mathcal{L}$ containing the points of $I$; see Fig. \ref{fig: coloring_proj_line_locally}.2. Thanks to $\pi_{n|_{\mathbb{R}\Sigma_n}}$ we can encode $\eta$ into a colored oriented graph $\overline{\Gamma}$ on $\mathbb{R}P^1 \subset \mathbb{C}P^1$ as follows (in Fig. \ref{fig: real graph_translate} the dashed lines denote fibers of $\mathcal{L}$):
\begin{enumerate}
\item To each fiber of $\mathcal{L}$ intersecting $\eta$ in two points, one of which with multiplicity $2$, we associate a $\times$-vertex on $\mathbb{R}P^1$. 
\item Fixed an interval $I$, let $F_1,F_2$ be two fibers of $\mathcal{L}_{I}$ intersecting $\eta$ in two points, one of which with multiplicity $2$, such that $\eta$, up to $\mathcal{L}$-isotopy, is locally as depicted in Fig. \ref{fig: real graph_translate}.2 or \ref{fig: real graph_translate}.3. Let $F_3$ be another fiber between $F_1, F_2$. Then, we associate to $F_3$ a $\circ$-vertex on $\mathbb{R}P^1$. Moreover, if between $F_1$ and $F_2$ each other fiber intersects $\eta$ in only one real point (as in Fig. \ref{fig: real graph_translate}.2), then we associate to a fiber between $F_1$ and $F_3$ (resp. $F_3$ and $F_2$) a $\bullet$-vertex on $\mathbb{R}P^1$. Between $\bullet$ and $\circ$-vertices we put bold edges.
\item For all intervals $I$, except for the fibers of $\mathcal{L}_{I}$ to which we associate essential vertices and bold edges, we associate dotted (resp. solid) edges on $\mathbb{R}P^1$ to the fibers of $\mathcal{L}_{I}$ which intersect $\eta$ in three distinct real points (resp. only one real point).
\item The orientations of the edges incident to a vertex are in an alternating order. In particular, the orientations of the edges incident to an essential vertex are respectively as described in $5-7$ of Definition \ref{defn: complete_graph}.
\end{enumerate}
The graph $\overline{\Gamma}$, called \textit{real graph}, is considered up to isotopy of $\mathbb{R}P^1$, namely it is determined by the order of its colored vertices since the edges are determined by the color of their adjacent vertices.\\
We say that $\overline{\Gamma}$ is \textit{completable in degree} $n$ if there exists a complete real trigonal graph $\Gamma$ of degree $n$ such that $\Gamma \cap \mathbb{R}P^1=\overline{\Gamma}$. 
\end{itemize}
\end{defn}

\begin{figure}[h!]
\begin{center}
\begin{picture}(100,90)
\put(-87,-13){Fig. \ref{fig: real graph_translate}.1}
\put(-35,-13){Fig. \ref{fig: real graph_translate}.2}
\put(39,-13){Fig. \ref{fig: real graph_translate}.3}
\put(-43,90){\small{$F_1$}}
\put(-20,90){\small{$F_3$}}
\put(0,90){\small{$F_2$}}
\put(101,-13){Fig. \ref{fig: real graph_translate}.4}
\put(40,90){\small{$F_1$}}
\put(50,90){\small{$F_3$}}
\put(60,90){\small{$F_2$}}
\put(169,-13){Fig. \ref{fig: real graph_translate}.5}
\put(-120,-50){\includegraphics[width=1\textwidth]{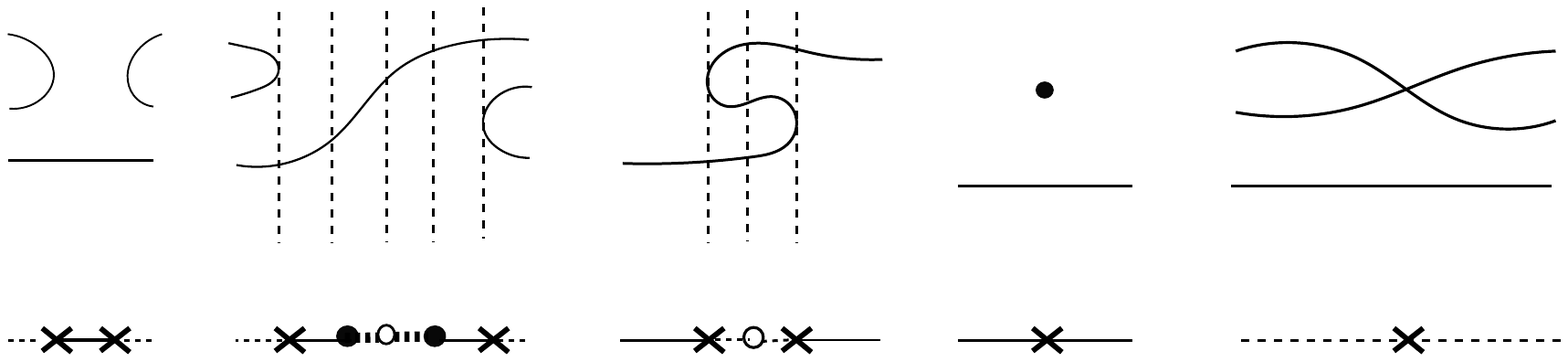}}
\end{picture}
\end{center}
\caption{Local topology of trigonal $\mathcal{L}$-schemes and their corresponding real graphs.}
\label{fig: real graph_translate}
\end{figure} 
\begin{thm}[\cite{Orev03}, \cite{Degt12}]
\label{thm: existence_trigonal}
A trigonal $\mathcal{L}$-scheme on $\mathbb{R}\Sigma_n$ is realizable by a real algebraic trigonal curve if and only if its real graph is completable in degree $n$. 
\end{thm}
Given a real graph $\overline{\Gamma}$, we  depict only the completion to a real trigonal graph $\Gamma$ on a hemisphere of $\mathbb{C}P^1$ since $\Gamma$ is symmetric with respect to the standard complex conjugation. Moreover, we can omit orientations in figures representing real trigonal graphs because each vertex is adjacent to an even number of edges oriented in an alternating order as, for example, depicted in Fig. \ref{fig: vertices_gamma} and such orientations are compatible with each others.\\
\begin{figure}[h!]
\begin{center}
\begin{tikzpicture} [scale=0.4]
\node at (0:0) {$\bullet$};
\draw  [densely dotted, line width=0.07cm] (0:0) -- (30:2);
\draw [decoration={markings,mark=at position 1 with {\arrow[scale=2.0]{<}}},
    postaction={decorate},
    shorten >=0.1pt] (30:0.8) -- (30:1);
\draw (0:0) -- (90:2);
\draw [decoration={markings,mark=at position 1 with {\arrow[scale=2.0]{>}}},
    postaction={decorate},
    shorten >=0.1pt] (90:0.8) -- (90:1);
\draw  [densely dotted, line width=0.07cm] (0:0) -- (150:2);
\draw [decoration={markings,mark=at position 1 with {\arrow[scale=2.0]{<}}},
    postaction={decorate},
    shorten >=0.1pt] (150:0.8) -- (150:1);
\draw (0:0) -- (210:2);
\draw [decoration={markings,mark=at position 1 with {\arrow[scale=2.0]{>}}},
    postaction={decorate},
    shorten >=0.1pt] (210:0.8) -- (210:1);
\draw  [densely dotted, line width=0.07cm] (0:0) -- (270:2);
\draw [decoration={markings,mark=at position 1 with {\arrow[scale=2.0]{<}}},
    postaction={decorate},
    shorten >=0.1pt] (270:0.8) -- (270:1);
\draw (0:0) -- (-30:2);
\draw [decoration={markings,mark=at position 1 with {\arrow[scale=2.0]{>}}},
    postaction={decorate},
    shorten >=0.1pt] (-30:0.8) -- (-30:1);
\node at (0:0) {$\bullet$};
\end{tikzpicture}
\begin{tikzpicture}[scale=0.4]
\draw  [densely dotted, line width=0.07cm] (0:0) -- (0:2);
\draw [decoration={markings,mark=at position 1 with {\arrow[scale=2.0]{>}}},
    postaction={decorate},
    shorten >=0.1pt] (0:1);
\draw [dashed] (0:0) -- (90:2);
\draw [decoration={markings,mark=at position 1 with {\arrow[scale=2.0]{<}}},
    postaction={decorate},
    shorten >=0.1pt] (90:1) -- (90:1.1);
\draw [densely dotted, line width=0.07cm] (0:0) -- (180:2);
\draw [decoration={markings,mark=at position 1 with {\arrow[scale=2.0]{<}}},
    postaction={decorate},
    shorten >=0.1pt] (180:1);
\draw [dashed] (0:0) -- (-90:2);
\draw [decoration={markings,mark=at position 1 with {\arrow[scale=2.0]{<}}},
    postaction={decorate},
    shorten >=0.1pt] (-90:1) -- (-90:1.1);
\draw [fill=white] (0,0) circle (1.5mm);
\end{tikzpicture}
\begin{tikzpicture}[scale=0.4]
\node at (0:0) {$\times$};
\draw (0:0) -- (0:2);
\draw [decoration={markings,mark=at position 1 with {\arrow[scale=2.0]{<}}},
    postaction={decorate},
    shorten >=0.1pt] (0:1);
\draw [dashed] (0,0) -- (0:-2);
\draw [decoration={markings,mark=at position 1 with {\arrow[scale=2.0]{<}}},
    postaction={decorate},
    shorten >=0.1pt] (0:-1);
\end{tikzpicture}
\end{center}
\caption{Colored vertices of a real trigonal graph.}
\label{fig: vertices_gamma}
\end{figure}

\section{Real curves on $k$-sphere real del Pezzo surfaces of degree 2}
\label{sec: DP2}

\subsection{Definitions}
\label{subsec: def_DP2}
Let $X$ be $\mathbb{C}P^2$ blown up at seven points in generic position; then, the surface $X$ is a del Pezzo surface of degree $2$ (see \cite{Russ02}, \cite[Chapter 8]{Dolg12}). The anti-canonical system $\phi$ of $X$ is a double ramified cover of $\mathbb{C}P^2$ and the branch locus of $\phi$ consists of an irreducible non-singular quartic $\overline{Q}$ defined by a homogeneous polynomial $f(x,y,z)$. By construction, the anti-canonical class $c_1(X)$ is the pull back via $\phi$ of the class of a line in $\mathbb{C}P^2$ (\cite{DegIteKha00}). Moreover, the surface $X$ is isomorphic to the real hypersurface in $\mathbb{C}P(1,1,1,2)$ defined by the weighted polynomial equation $f(x,y,z)=w^2$, with coordinates $x$, $y$, $z$ and $w$ respectively of weights $1$ and $2$. Conversely, any double cover of $\mathbb{C}P^2$ ramified along a non-singular algebraic quartic yields a del Pezzo surface of degree $2$. 
\par If one equips $X$ with a real structure $\sigma$, the quartic $\overline{Q}$ is real and $f(x,y,z)$ can be chosen with real coefficients and so that the real surface $(X, \sigma)$ is isomorphic to the real hypersurface in $\mathbb{C}P(1,1,1,2)$ of equation $f(x,y,z)=w^2$. 
It follows that the double cover $\phi$ projects $\mathbb{R}X$ into the region
$$\Pi_+:=\{[x:y:z]\in \mathbb{R}P^2: f(x,y,z)\geq 0\}.$$ 
Conversely, the double cover of $\mathbb{C}P^2$ ramified along a non-singular real quartic $\overline{Q} \subset \mathbb{C}P^2$ and a choice of a real polynomial equation $f(x,y,z)$ of $\overline{Q}$ yields a real del Pezzo surface $X$. \\
\begin{figure}[!h]
\begin{picture}(100,55)
\put(190,53){\textcolor{black}{$\phi$}}
\put(60,-6){\includegraphics[width=0.70\textwidth]{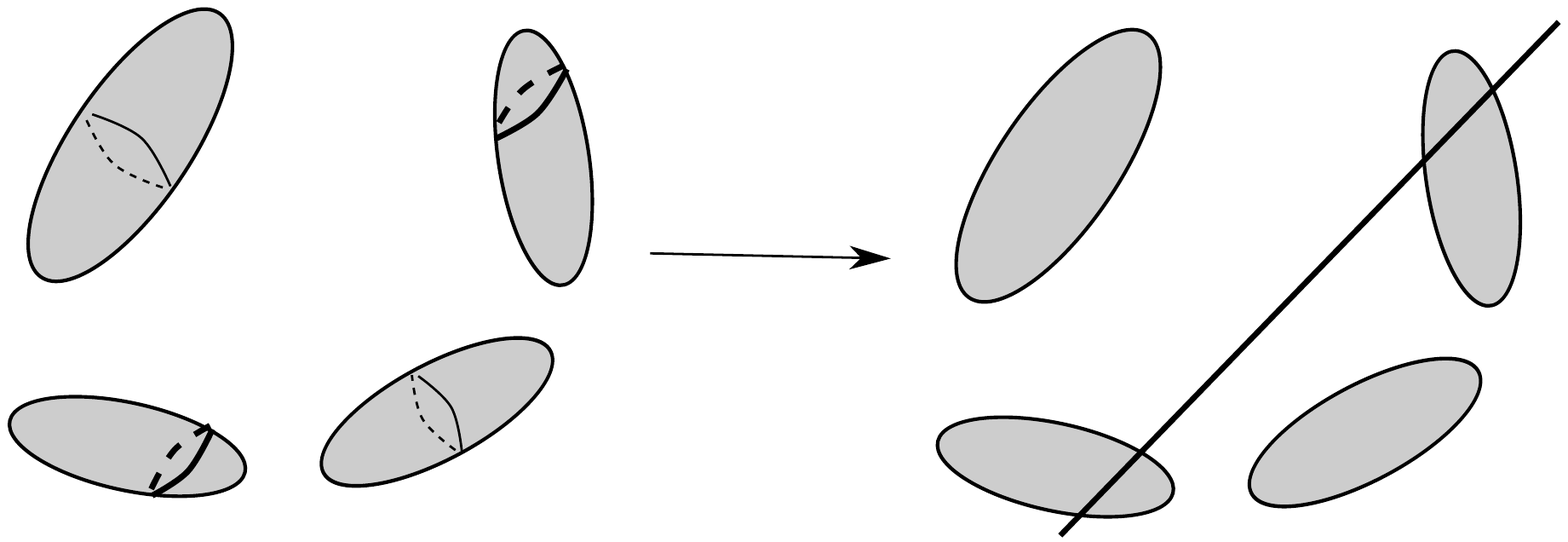}} 
\end{picture}
\caption{Example: $\phi: \mathbb{R}X \mapsto \Pi_{+}$.}
\label{fig: double_cover_DP2}
\end{figure}
\par The surface $X$ is a $\mathbb{R}$-minimal\footnote{\label{myfootnote}We say that a real algebraic variety $(X,\sigma)$ is $\mathbb{R}$-minimal, if every real degree $1$ holomorphic function  $f:X \rightarrow Y$ to a real algebraic surface $(Y, \tau)$ is a biholomorphism.} if and only if $X$ is a $4$-sphere or a $3$-sphere real del Pezzo surface (Definition \ref{defn: ksphere}). 
Moreover $X$ is a $k$-sphere real del Pezzo surface,
with $1 \leq k \leq 4$ if and only if $\overline{Q}$ is a non-singular real quartic realizing the real scheme $k$ in $\mathbb{R}P^2$ and $\Pi_+$ is orientable; see \cite{DegKha02} and, as example for $k=4$, look at Fig. \ref{fig: double_cover_DP2}, where $\Pi_+$ is in gray on the right.
\begin{note}
Let $X$ be a $k$-sphere real del Pezzo surface of degree $2$. We denote the connected components of $\mathbb{R}X$ with $ X_1,\dots,X_k$.
\end{note}
The lifting of a non-singular real algebraic curve $C\subset\mathbb{C}P^2$ of degree $d$ via $\phi$ is a real algebraic curve $A$ of class $d$ in $X$ (Definition \ref{defn: curve_d}). Moreover, from the topological type realized by the triplet $(\mathbb{R}P^2, \mathbb{R}\overline{Q}, \mathbb{R}C)$, one recovers the real scheme of the pair $(\mathbb{R}X, \mathbb{R}A)$. As example, assume that $d=1$, the quartic $\overline{Q}$ is maximal, $\Pi_+$ is orientable and $\mathbb{R}\overline{Q} \cup \mathbb{R}C$ is arranged in $\mathbb{R}P^2$ as depicted in Fig. \ref{fig: double_cover_DP2} on the right; then the pair $(\mathbb{R}X, \mathbb{R}A)$ has real scheme $1:1:0:0$; Fig. \ref{fig: double_cover_DP2} on the left.
\par If $X$ is a $4$-sphere real del Pezzo surface of degree $2$, any real algebraic curve has class $d$, where $d$ is some non-negative integer (\cite{Russ02}).\\
\par Combining Harnack-Klein's inequality and the adjunction formula, one obtains the following immediate result. 
\begin{prop}
\label{prop: H-K_DP2}
Let $A$ be a real algebraic curve of class $d$ in a $k$-sphere real del Pezzo surface $X$ of degree $2$, with $1 \leq k \leq 4$. Then, the number $l$ of ovals of $\mathbb{R}A$ is bounded as follows:
$$l \leq d(d-1)+2.$$
\end{prop}
\subsection{Real schemes}
\label{subsec: real_schemesDP2}
\begin{defn}
\label{defn: real_schemeDP2}
Let $X$ be a $k$-sphere real del Pezzo surface $X$ of degree $2$, with $0 \leq k \leq 4$. Let us denote $\mathcal{S}_{DP2}(X,k)$ the set of all real schemes in $X$. \\
Notice that $\mathcal{S}_{DP2}(X,k)$ does not depend on the choice of $X$. Therefore, from now on, we omit $X$ and write $\mathcal{S}_{DP2}(k).$
\end{defn}

Let us enrich the notion of real scheme with some extra conditions deriving from Proposition \ref{prop: H-K_DP2}.

\begin{defn}
Let $\mathcal{S}$ be in $\mathcal{S}_{DP2}(k)$. We say that $\mathcal{S}$ is \textit{in class} $d$ or we write $\mathcal{S} \in \mathcal{S}_{DP2}(k,d)$ if the number of ovals of $\mathcal{S}$ does not exceed $d(d-1)+2$.
\end{defn}

\begin{defn}
We say that $\mathcal{S} \in \mathcal{S}_{DP2}(k,d)$ is \textit{realizable in} $X^{k}$ \textit{and in class} $d$, if there exist a $k$-sphere real del Pezzo surface $X^{k}$ of degree $2$ and a real algebraic curve $A \subset X^{k}$ of class $d$, such that the pair $(\mathbb{R}X^{k}, \mathbb{R}A)$ realizes $\mathcal{S}$. 
\end{defn}
To refine the classifications of real schemes in $\mathcal{S}_{DP2}(k,d)$, one can additionally ask about the realizability of a given real scheme by symmetric real curves.
\begin{defn}
\label{defn: non-symm}
\begin{itemize}
\item[]
\item Let $A$ be a class $d$ algebraic curve in a $k$-sphere real degree $2$ del Pezzo surface $X$, with $1 \leq k \leq 4$. Let $\phi: X \rightarrow \mathbb{C}P^2$ be the anti-canonical map of $X$. 
We say that $A$ is \textit{symmetric} if it is the lifting via $\phi$ of a plane algebraic curve of degree $d$. Otherwise, we say that $A$ is \textit{non-symmetric}.
\item We say that $\mathcal{S} \in \mathcal{S}_{DP2}(k,d)$ is \textit{symmetric in class} $d$ if it is realizable in $X^k$ and in class $d$ by a symmetric real algebraic curve of class $d$. Otherwise, we say that $\mathcal{S}$ is \textit{non-symmetric in class} $d$.
\end{itemize} 
\end{defn}
Let us lighten the real scheme notation introduced in Section \ref{subsec: cha2_codage_isotopie}.
\begin{note}
\label{note: simplify_notation_DP2}
Let $\mathcal{S}:=\mathcal{S}_1:\dots: \mathcal{S}_4$ be a topological type in the disjoint union of $4$ spheres. Let $X^{k}$ be a $k$-sphere real del Pezzo surface of degree $2$, with $1 \leq k \leq 4$.  If $\mathcal{S}$ has at least $4-k$ trivial entries, we say that $\mathcal{S}$ is a real scheme in $\mathbb{R}X^{k}$.
\end{note}
\subsection{Main results}
\label{subsec: real_schemes_DP2} 
The real scheme classification is complete in class $1$ and $2$. The proof of the following statement is in Section \ref{subsec: class 1 and 2_DP2}.
\begin{prop}[Class 1 and 2]
\label{prop: d=1,d=2_DP2}
Any real scheme $\mathcal{S} \in \mathcal{S}_{DP2}(k,d)$ where $d=1,2$ and $1 \leq k \leq 4$, is realizable in $X^{k}$ and in class $d$. Moreover $\mathcal{S}$ is symmetric in class $d$.
\end{prop}
For real schemes in class $d \geq 3$, Proposition \ref{prop: H-K_DP2} does not provide a complete system of restrictions anymore. In Section \ref{subsec: obstr_DP2}, we show how to use Welschinger-type invariants to obtain Bézout-type obstructions for any class $d$ and, for $d$ even, we present an application of Comessatti-Petrovsky inequality.
\par Sections \ref{subsec: class 1 and 2_DP2} - \ref{subsec: final_constr_DP2}
are devoted to construction of real algebraic curves; several construction techniques are combined, including \textit{dessins d'enfants} and recent developments of Viros's patchworking method. 
\par We mainly focus on classifications on $4$-sphere del Pezzo surfaces of degree $2$: we give a partial classification of real schemes in class $3$ with $8$ ovals (Theorem \ref{thm: principal_DP2}), and we realize some non-symmetric real schemes for each class $d \geq 5$ with $2d+1$ ovals (Proposition \ref{prop: non-symm_real_scheme}). In addition, we use the classification tools adopted on $4$-sphere del Pezzo surfaces of degree $2$ in the case of $k$-sphere real degree $2$ del Pezzo surfaces, with $k < 4$ (Proposition \ref{prop: principal_DP2}). 
\begin{thm}[$k=4$ and Class 3]
\label{thm: principal_DP2}
There are 74 real schemes in $\mathcal{S}_{DP2}(4,3)$ with $8$ ovals, which are non-prohibited by Proposition \ref{prop: bezout_inv_welschinger} and Lemmas \ref{lem: caso_particolare_4_sfere_k=2}, \ref{lem: caso_particolare_DP2_obstr} . Among those $48$ are realizable in $X^{4}$ and in class $3$. Moreover $19$ out the $48$ are symmetric in class $3$.
\end{thm}
\begin{prop}[$k=3,2,1$ and Class 3]
\label{prop: principal_DP2}
There are respectively $79$, $61$ and $28$ real schemes in $\mathcal{S}_{DP2}(k,3)$ with 8 ovals, which are non-prohibited by Proposition \ref{prop: bezout_inv_welschinger} for $k=3,2$ and Lemma \ref{lem: welschinger_1_sfere_class_2} for $k=1$. Among those respectively $49$, $38$ and $17$ are realizable in $X^{k}$ and in class $3$. Moreover $6$ and $2$ are symmetric in class $3$, respectively for $k=3,2$ and $k=1$. 
\end{prop}
The realization of the real schemes in $\mathcal{S}_{DP2}(k,3)$ of Theorem \ref{thm: principal_DP2} and Proposition \ref{prop: principal_DP2} are in the proof of Proposition \ref{prop: circ_label_DP2} (symmetric ones) and of Proposition \ref{prop: star_label_DP2}.\\
On the last page of the paper, it is presented:
\begin{itemize}
\item a summary of the realized (symmetric) real schemes in $X^{k}$ and in class $3$, with $1 \leq k \leq 4$ (Table \ref{tabella=realized3}).
\item a list of real schemes in class $3$ which are still unrealized on $X^{k}$, with $1 \leq k \leq 3$; but, which can not be realized by class $3$ real curves on $4$-sphere real del Pezzo surface of degree $2$ (Table \ref{tabellaknot4}).
\end{itemize}
The real scheme $2    \sqcup  \langle1\rangle    \sqcup  \langle1\rangle    \sqcup  \langle1\rangle:0:0:0$ 
is symplectically realizable in $X^{4}$ and in class $3$ (proof in Section \ref{subsec: symplectic}), but we do not know yet whether it is realizable algebraically. 
\begin{prop}
\label{prop: symplectic_curve_surface}
There exist a $4$-sphere real symplectic degree $2$ del Pezzo surface $X$ and a non-singular real symplectic curve of class $3$ in $X$ realizing the real scheme $2    \sqcup  \langle1\rangle    \sqcup  \langle1\rangle    \sqcup  \langle1\rangle:0:0:0$.  
\end{prop}
For each class $d \geq 5$, we realize some (non-symmetric) real schemes in class $d$. The proof of the following statement is at the end of Section \ref{subsec: final_constr_DP2}.
\begin{table}[ht]
 \centering
\begin{threeparttable}
 \caption{ \label{tabella=non-symm} \textbf{(Non-symmetric) real schemes $\mathcal{S}_{1}:\mathcal{S}_{2}:N(h_3,0): N(h_4,0)$ realized in $X^{4}$ and in class $d \geq 5$}}
\begin{tabular}{ | l | l | l |}
\hline
&\textbf{$\mathcal{S}_{1}:\mathcal{S}_{2}$} & \textbf{Extra conditions} \\
\hline
&&$k_{1}=0$, $h_{2}\not = 1$\\
$(1)$ &$1    \sqcup   N(h_1,1) : 1    \sqcup   N(k_2,1    \sqcup  \langle 1 \rangle)    \sqcup   N(h_2,0)$ & $h_{1} \equiv 1 \mod 2$\\
& & If $h_{2}=2$, $k_{2}\equiv 1 \mod 2.$\\
& & If $h_{2}=0$, $k_{2} \not = 0,1.$ \\
\hline
&&$k_{1}=0$, $h_{1}> 1$\\
$(2)$&$1    \sqcup  \langle 1 \rangle   \sqcup   N(h_1,1): 1    \sqcup   N(k_2,1)   \sqcup   N(h_2,0)$&$h_{2}\not = 1$, $h_{2} \not = k_{2}+1$\\
 && If $h_{2}=0$, $k_{2}\equiv 1 \mod 2.$\\
 \hline
&& $h_{1}\not = 1$,  $k_{1} \not = 0$\\
&& $h_{1} \not = k_{1}+1$ \\
$(3)$& $1   \sqcup   N(k_1,1)   \sqcup N(h_1,0): N(k_2,1   \sqcup   \langle 1 \rangle)   \sqcup   N(h_2,1)$ & If $h_{1}=0$, $k_{1} \equiv 1 \mod 2$. \\
&& If $h_{2}=0,1$, $k_{2}\not = 0.$\\
 && If $h_{2}=0$, $k_{2}\not = 1$.\\
 \hline
&& $h_{2} \not = 1$, $k_{2} \not = 0$ \\
$(4)$& $1   \sqcup   N(k_2, 1   \sqcup   \langle 1 \rangle)   \sqcup   N(h_2,0): N(k_1,1)   \sqcup   N(h_1,1)$&$h_{1}+ k_{1} \equiv 1 \mod 2$ \\
& & If $h_{2}=0,2$, $k_{2} \not = 1$. \\
&& If $h_{2}=2$, $k_{2} \equiv 1 \mod 2$.\\
 \hline
 && $k_1=0$, $h_1 \not = 1$\\
 $(5)$&$ 1   \sqcup   N(h_1,0)   \sqcup   N(h_2,1) :  N(k_2+1, 1   \sqcup   \langle 1 \rangle)$& $h_{2} \not = 0$, $h_{1} \not = h_{2}+1$ \\
 & & $ k_{2} \not = 1$  ($ \implies d \geq 6$) \\ 
  \hline
&&  $k_1=0$, $h_1 \not = 1$\\
&& $h_{2} \not = 0$ ,$h_{1} \not = h_{2}+1$\\
 $(6)$&$ 1   \sqcup   \langle 1 \rangle   \sqcup   N(h_1,0)   \sqcup   N(h_2,1) : N(k_2+1,1) $  &  $k_{2}\equiv 1 \mod 2$\\
&& $h_{1}\not = 2$\\
 & & $h_{1}\not = 1$\\
 \hline
 & &  $h_{1} \not = h_{2}+1$ \\
  $(7)$& $ N(k_2,1  \sqcup   \langle 1 \rangle)   \sqcup    N(h_2,1)    \sqcup   N(h_1,0): N(k_1+1,1)$& $k_{1}\equiv 1 \mod 2$ \\
&&  If $(h_{1}, h_{2})=(1,1), (2,0)$,  \\
 & & $ k_{2}\equiv 1 \mod 2$.\\
 \hline
&&$h_{1} \not = h_{2}+1$, $k_{2} \not = 1,2$ \\
  $(8)$& $ N(k_1+1,1)   \sqcup    N(h_2,1)   \sqcup   N(h_1,0):  N(k_2,1   \sqcup   \langle 1 \rangle) $&$h_{1} \not = k_{1}+ 2$ \\
 &&$h_{2} \not = k_{1} + 1$\\
 \hline
 \end{tabular}
       \begin{tablenotes}
      \item[(i)] Central column: the symbol $N(h_{3},0):N(h_{4},0)$ is omitted because common to every real scheme.
   \item[(ii)] Right column: sufficient conditions on the parameters $k_i, h_{j},d$ to have non-symmetric real schemes.
         \end{tablenotes}
       \end{threeparttable}
\end{table}
\begin{prop}[Non-symmetric and Class $d \geq 5$]
\label{prop: non-symm_real_scheme}
Let $d,$ $k_1,k_2,$ and $h_1,h_2,h_3,h_4$ be non-negative integers such that
\begin{itemize}
\item $d\geq 5$;
\item $k_1 + k_2=d-4$;
\item $\sum\limits_{i=1}^{4}h_i=d-1$;
\item $h_3 \equiv 1 \mod  2$, if $h_{3} \not = 0$;
\item $h_4 \equiv 1 \mod  2$, if $h_{4} \not = 0$.
\end{itemize}
Then each real scheme $\mathcal{S}_{1}:\mathcal{S}_{2}:N(h_3,0): N(h_4,0)$ in $\mathcal{S}_{DP2}(4,d)$ in Table \ref{tabella=non-symm}, is realizable in $X^{4}$ and in class $d$. Moreover,  the real schemes whose parameters respect the extra conditions in Table \ref{tabella=non-symm}, are non-symmetric in class $d$.
\end{prop}
\subsection{Obstructions and Welschinger-type invariants}
\label{subsec: obstr_DP2}
Welschinger invariants can be regarded as real analogues of genus zero Gromov-Witten invariants. They were introduced in \cite{Wels05} and count, with appropriate signs, the real rational curves which pass through a given real collection of points in a given real rational algebraic surface. In the case of $k$-sphere real del Pezzo surfaces of degree $2$, the Welschinger invariants, as well as their generalizations to higher genus (\cite{Shus15}), can be used to prove the existence of interpolating real curves of genus $0 \leq g\leq k-1$; see \cite{IteKhaShu15} and \cite{Shus15}. 
\begin{prop}\cite[Propositions 4 and 5]{Shus15}
\label{prop: shustin_welschinger}
Let $s$ be an integer greater than $1$ and $r_1$, $r_2$ be two non-negative odd integers such that $r_1+r_2=2s$. Let $\mathcal{P}$ be a generic configuration of $2s+j$ real points, with $j=2,1,0$, on a $k$-sphere real del Pezzo surface $X$ of degree $2$, where $k=2+j$, such that 
\begin{itemize}
\item $X_i$ contains $r_i$ points of $\mathcal{P}$, with $i=1,2$; 
\item $X_i$ contains one point of $\mathcal{P}$, if $i \not = 1,2$. 
\end{itemize}
Then, there exists a real algebraic curve $T$ of class $s$ and genus $j+1$  in $X$ passing through $\mathcal{P}$. Furthermore, the points of $\mathcal{P}$ belong to the one-dimensional connected components of $\mathbb{R}T$.
\end{prop}
We use the result of Proposition \ref{prop: shustin_welschinger} to prove the following proposition.
\begin{prop}
\label{prop: bezout_inv_welschinger}
Let $s$ be an integer strictly greater than $1$ and $r_1$, $r_2$ be two non-negative odd integers such that $r_1+r_2=2s$. Moreover, let $A$ be a non-singular real algebraic curve of class $d$ in a $k$-sphere real del Pezzo surface $X$ of degree $2$, with $k=4,3,2$. Let $t$ denote the number of connected components of $\mathbb{R}X$ which intersect $\mathbb{R}A$. Assume that $\mathbb{R}A$ has $r_i$ disjoint nests $N_{h}$ of depth $j_h$ on $X_i$, respectively with $1\leq h \leq r_1$ for $i=1$, and with $r_1+1\leq h \leq 2s$ for $i=2$.
\begin{enumerate}[label=(\arabic*)]
\item If $r_1,r_2>1$, then $\sum\limits_{h=1}^{2s}j_h \leq ds-(t-2)$;
\item If $r_1 =2s-1$ and $r_2=1$, then $\sum\limits_{h=1}^{2s-1}j_h \leq ds-(t-1)$. 
\end{enumerate}
\end{prop}
\begin{proof}
Assume that $r_1$ and $r_2$ are strictly greater than $1$. It follows that $\mathbb{R}A$ has at least $3$ disjoint nests on $X_i$, with $i=1,2$. In order to prove inequality $(1)$, let us choose a generic collection $\mathcal{P}$ of $2s+j$ real points, with $j=k-2$,  
in the following way. On each boundary of the $r_1$ (resp. $r_2$) disks in $X_1 \setminus \overset{r_1}{\underset{h=1}{\sqcup}} N_{h}$ (resp. $X_2 \setminus \overset{2s}{\underset{h=r_1+1}{\sqcup}} N_{h}$), pick a point. Moreover, pick a point on every connected component $X_i$, with $i=3,4$, such that the point belongs to $\mathbb{R}A$ any time the real algebraic curve has at least one oval on $X_i$. Then, Proposition \ref{prop: shustin_welschinger} assures the existence of a real algebraic curve $T$ of class $s$ and genus $j+1$ on $X$ passing through $\mathcal{P}$. Furthermore, the points of $\mathcal{P}$ belong to the one-dimensional connected components of $\mathbb{R}T$. Thus, the number of real intersection points of $A$ with $T$ is at least $2(\sum\limits_{h=1}^{2s}j_h+(t-2))$. Inequality $(1)$ follows directly from the fact that the intersection number $A\circ T=2ds$ is greater or equal than the number of real intersection points of $A$ with $T$.\\
The proof of $(2)$ is similar to the previous one.
\end{proof}
The following statement gives more topological obstructions for real curves in $1$-sphere real del Pezzo surfaces of degree $2$.
\begin{lem}
\label{lem: welschinger_1_sfere_class_2}
Let $A$ be a real algebraic curve in class $d$ in a $1$-sphere real del Pezzo surface $X$ of degree $2$. Assume that $\mathbb{R}A$ has three nests $N_{h}$ of depth $j_h$ on $X_1$. Then, $$j_1+j_2+j_3 \leq 2d.$$
\end{lem}
\begin{proof}
The statement follows from 
\begin{itemize}
\item the existence of a real algebraic curve $T$ of class $2$ and genus $0$ on $X$ passing through a given real configuration of $3$ distinct points on $X_1$ (see \cite[Table $1$, Section $2.2$]{IteKhaShu15}) and
\item an argument similar to that used in the proof of Proposition \ref{prop: bezout_inv_welschinger}.
\end{itemize}
\end{proof}
We have one more Bézout-type restriction on the topology of real curves in $4$-sphere real del Pezzo surfaces of degree $2$.
\begin{lem}
\label{lem: caso_particolare_4_sfere_k=2}
Let $A$ be a real algebraic curve in class $d$ in a $4$-sphere real del Pezzo surface $X$ of degree $2$. Assume that $\mathbb{R}A$ has a nest $N_{1}$ of depth $j_1$ on $X_1$ and a nest $N_{2}$ of depth $j_2$ on $X_2$. Let $t$ denote the number of connected components of $\mathbb{R}X$ which intersect $\mathbb{R}A$. Then, $$j_1+j_2 \leq 2d- (t-2).$$
\end{lem}
\begin{proof}
Let $|L|$ be a linear system of curves on $X$. Passing through a given (real) point on $X$ defines a (real) hyperplane on $|L|$. Therefore, given a (real) configuration of $h$ points on $X$ such that $h$ is less or equal to the dimension of $|L|$, there exists a (real) curve in $|L|$ passing through the $h$ points. It follows that there exists a real algebraic curve $T$ of class $2$ and genus $3$ on $X$ passing through a given real configuration $\mathcal{P}$ of $6$ distinct points such that $X_i$ contains $2$ points of $\mathcal{P}$, with $i=1,2$, and $X_i$ contains one point of $\mathcal{P}$, if $i \not = 1,2$. The existence of such a curve $T$ and an argument similar to that used in the proof of Proposition \ref{prop: bezout_inv_welschinger}, prove the statement.
\end{proof}
A variant of the technique used in proof of Lemma \ref{lem: caso_particolare_4_sfere_k=2} leads to prohibit a particular real scheme in class $3$ in $4$-sphere real del Pezzo surfaces of degree $2$.
\begin{lem}
\label{lem: caso_particolare_DP2_obstr}
There is no real algebraic curve of class $3$ in any $4$-sphere real del Pezzo surface $X$ of degree $2$ realizing the real scheme $$\mathcal{S}:=\langle1\rangle  \sqcup  \langle1\rangle   \sqcup  \langle1\rangle   \sqcup   \langle1\rangle :0:0:0.$$
\end{lem}
\begin{proof}
Assume that there exists a non-singular real algebraic curve $A$ of class $3$ realizing $\mathcal{S}$ in $X$. Let us choose a configuration $\mathcal{P}$ of $6$ real points as follows. On each boundary of the $4$ disks in $X_1 \setminus \mathbb{R}A$, pick a point. Moreover, pick a point on the connected components $X_2$ and $X_3$. Then, there exists a non-singular real algebraic curve $T$ of class $2$ passing through $\mathcal{P}$ and $T$ has at most two ovals on $X_1$ and one oval on both $X_2$ and $X_3$. 
Thus, the number of real intersection points of $A$ with $T$ is at least $14$. But the intersection number $A\circ T$ is $12$.
\end{proof}
In this article, apart form Proposition \ref{prop: non-symm_real_scheme}, we mainly focus on class $3$ real curves on $X^{k}$. To look further, in Proposition \ref{prop: petro} we present one possible application of Comessatti-Petrovsky inequality (\cite{Come28}, \cite{Petro33}, \cite{Petro38}), which gives a topological type restriction for real curves of even class $d$.
\begin{prop}
\label{prop: petro}
Let $A$ be a non-singular real algebraic maximal curve of even class $d \geq 8$ in a $k$-sphere real del Pezzo surface $X$ of degree $2$, with $1 \leq k \leq 4$. Then, the pair $(\mathbb{R}X,\mathbb{R}A)$ does not realize the real scheme $l:0:0:0$, with $l=d(d-1)+2$.
\end{prop}
\begin{proof}
Let $Y\stackrel{\pi}{\rightarrow} X$ be a double cover of $X$ ramified along $A$. Let $B_1$ and $B_2$ be two distinct disjoint unions of connected components of $\mathbb{R}X\setminus \mathbb{R}A$ such that each $B_i$ is bounded by $\mathbb{R}A$. There exist two lifts $\sigma_{1}, \sigma_{2}$ to $Y$ of the real structure of $X$ via the double cover $\pi$ and the real part of $Y$ is the double of one of the $B_i$'s. Thanks to Comessatti-Petrovsky inequality one has
 \begin{equation}
 \label{eqn: petro}
 -14-\frac{d(3d-2)}{2} \leq \chi (\mathbb{R}Y) \leq  16+\frac{d(3d-2)}{2},
 \end{equation}
 where $\chi (\mathbb{R}Y)= 2 \chi (B_{i})$.\\
For each possible choice of $B_{1}$ and $B_{2}$ (Remark \ref{rem: petro_tricky}), a direct application of (\ref{eqn: petro}) shows that at least one $\chi(B_{i})$ does not respect the inequalities.
\end{proof}
In order to find other Comessatti-Petrovsky restrictions for real schemes in class $d$, the following should be observed. 
\begin{rem}
\label{rem: petro_tricky}
In the proof of Proposition \ref{prop: petro}, the choices of $B_{1}$ and $B_{2}$ are not independent; in fact, the choice on $X_{i}$ of
two disjoint unions of connected components of $X_{i}\setminus \mathbb{R}A$ bounded by $\mathbb{R}A$ imposes the choice on the other spheres. However, it is not obvious to know which are these imposed halves, unless you already know the double covering $Y$ which contains in particular the information on $B_{1}$ and $B_{2}$.
\end{rem}
\subsection{Class 1 and 2}
\label{subsec: class 1 and 2_DP2} 
Let us construct some symmetric real curves of class $1$ and $2$ on $k$-sphere real del Pezzo surfaces of degree $2$.
\begin{proof}[Proof of Proposition \ref{prop: d=1,d=2_DP2}.]
Fix a non-negative integer $k \leq 4$ and any real scheme $\mathcal{S}$ in $\mathcal{S}_{DP2}(k,d)$, with $d=1,2$. It is easy to construct a non-singular real plane quartic $\overline{Q}$, with real scheme $k$, and a line (resp. conic) $C$ in $\mathbb{C}P^2$ such that, via the double cover of $\mathbb{C}P^2$ ramified along $\overline{Q}$, the lifting of $C$ is a class $1$ (resp. $2$) real algebraic curve realizing $\mathcal{S}$ in a $k$-sphere real del Pezzo surface of degree $2$. See Example \ref{exa: class_1_2}. 
\end{proof} 
\begin{exa}
\label{exa: class_1_2}
Let us fix $k=4$. From the quartics and lines arranged in $\mathbb{R}P^2$ as in Fig. \ref{fig: d=1}, one construct $4$-sphere real del Pezzo surfaces $X$ of degree $2$ and class $1$ real algebraic curves in $X$ realizing all real schemes in class $1$. \\
Analogously, to realize all real schemes in $X^{4}$ and class $2$ with a maximal number of ovals, it is enough to construct real plane conics and quartics mutually arranged in $\mathbb{R}P^{2}$ as in Fig. \ref{fig: d=2}. 
The construction of such a pair of real plane curves, realizing the first three real schemes of Fig. \ref{fig: d=2}, follows from Hilbert's construction method (\cite{Hilb02}) which allows to construct a real quartic perturbing the union of two real conics; the remaining arrangements in Fig. \ref{fig: d=2} are realized fixing a real quartic which has a non-convex oval, taking a pair of lines which intersect the quartic transversally, and perturbing the lines to a conic; see \cite{Brus21} for details on perturbations.
\end{exa}
\begin{figure}[h!]
\begin{center}
\begin{picture}(100,50)
\put(-95,-15){\includegraphics[width=0.8\textwidth]{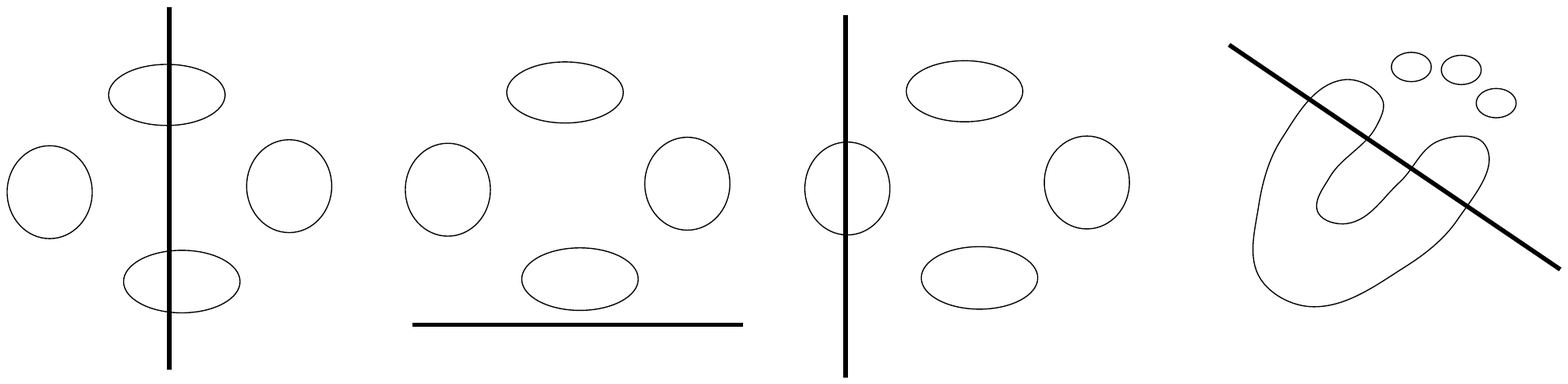}} 
\end{picture}
\end{center}
\caption{Arrangements of real lines (in thick black) and real maximal quartics (in gray) in $\mathbb{R}P^2$.}
\label{fig: d=1}
\end{figure}
\begin{figure}[h!]
\begin{center}
\begin{picture}(100,100)
\put(-95,-25){\includegraphics[width=0.8\textwidth]{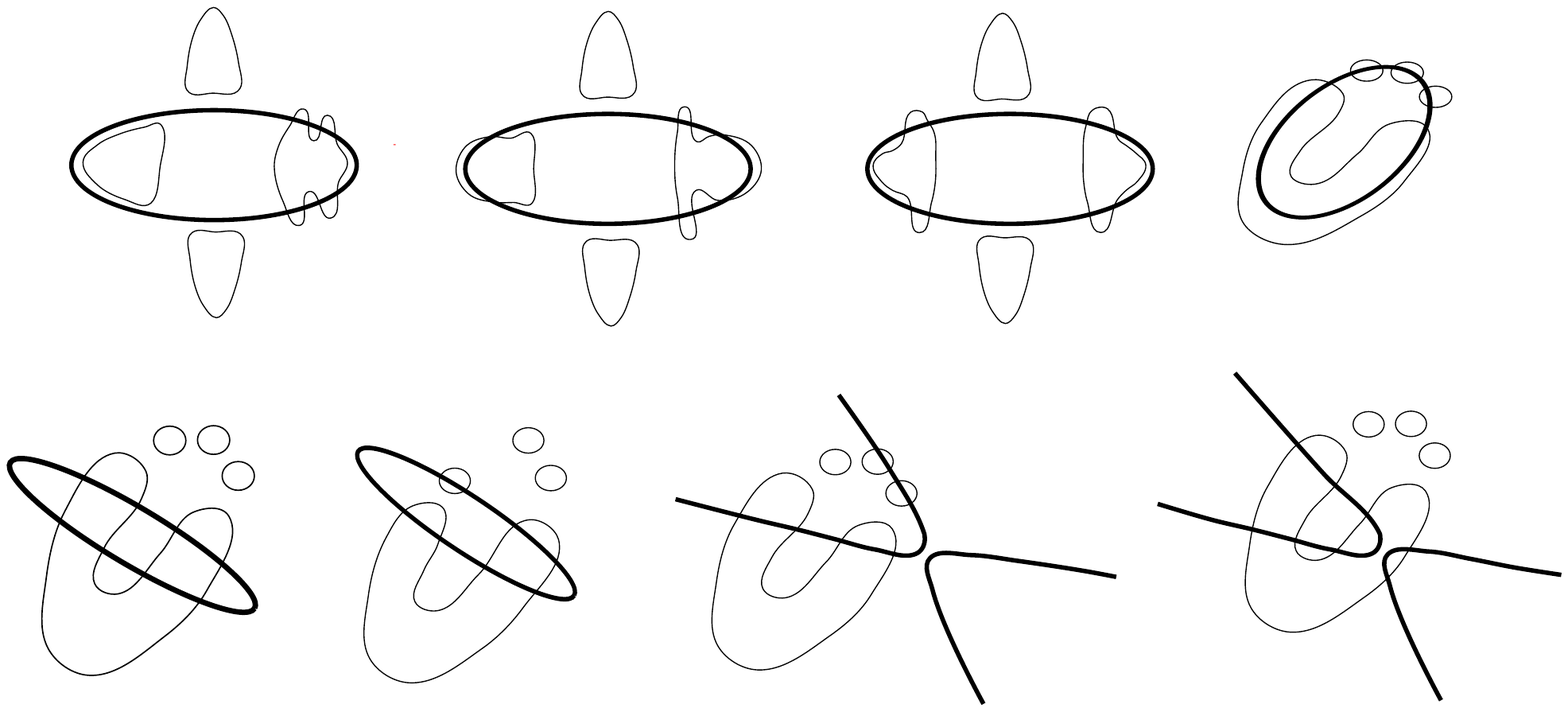}} 
\end{picture}
\end{center}
\caption{Arrangements of real conics (in thick black) and maximal quartics (in gray) in $\mathbb{R}P^2$.}
\label{fig: d=2}
\end{figure}
\subsection{Class 3}
\label{subsec: class_3_DP2} 
Let us prove a part of Theorem \ref{thm: principal_DP2} and Proposition \ref{prop: principal_DP2}, realizing some symmetric real schemes in $X^{k}$ and in class $3$.
\begin{prop}
\label{prop: circ_label_DP2}
Each real scheme $\mathcal{S}$ in $\mathcal{S}_{DP2}(k,3)$ labeled with $\circ$ and/or $ \circ^{*}$ in Table \ref{tabella=realized3}, is symmetric in class $3$.
\end{prop}
\begin{figure}[h!]
\begin{center}
\begin{picture}(100,110)
\put(-95,-10){\includegraphics[width=0.8\textwidth]{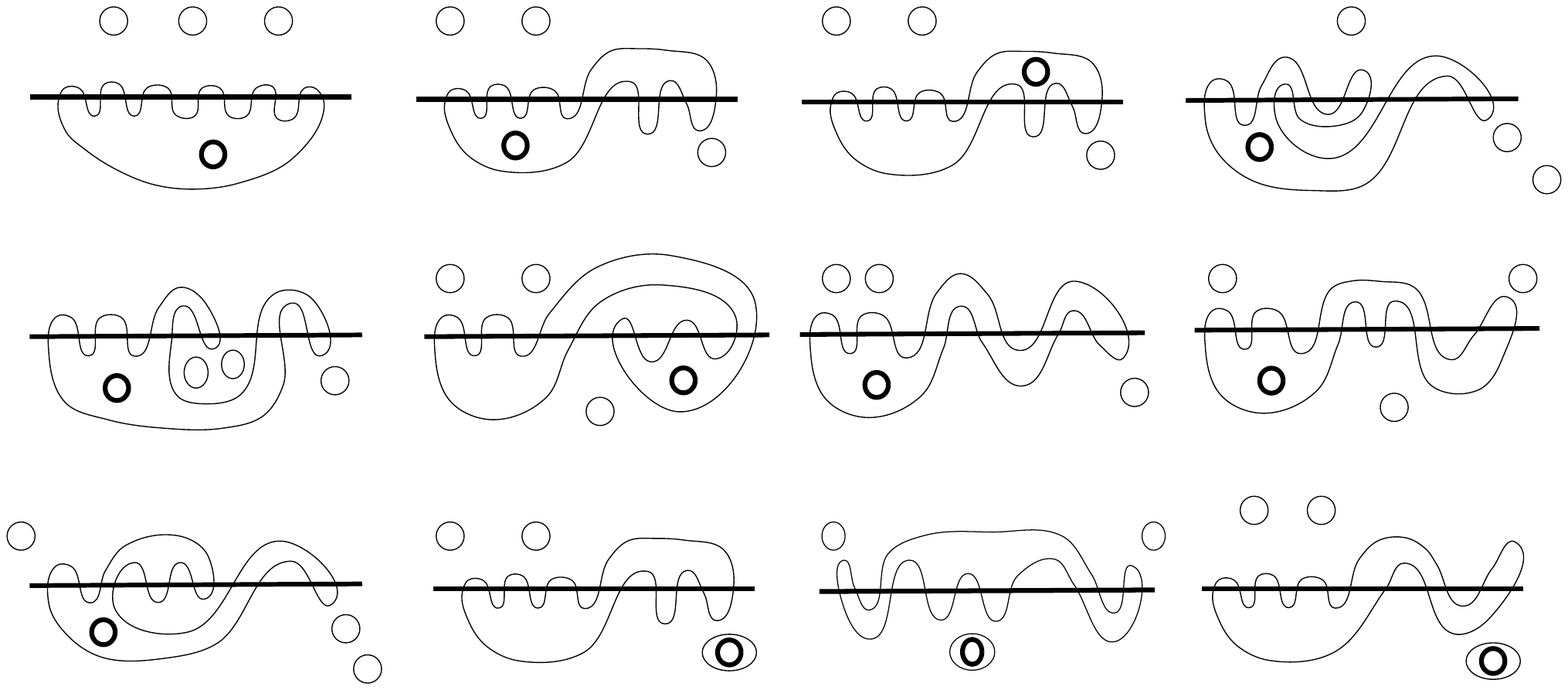}} 
\end{picture}
\end{center}
\caption{Mutual arrangements, up to isotopy, on $\mathbb{R}P^2$ of a real maximal cubic (in thick black) and a real maximal quartic (in gray).}
\label{fig: orevkov_cubic_quartic_DP2}
\end{figure}
\begin{figure}[h!]
\begin{center}
\begin{picture}(100,120)
\put(-100,0){\includegraphics[width=0.8\textwidth]{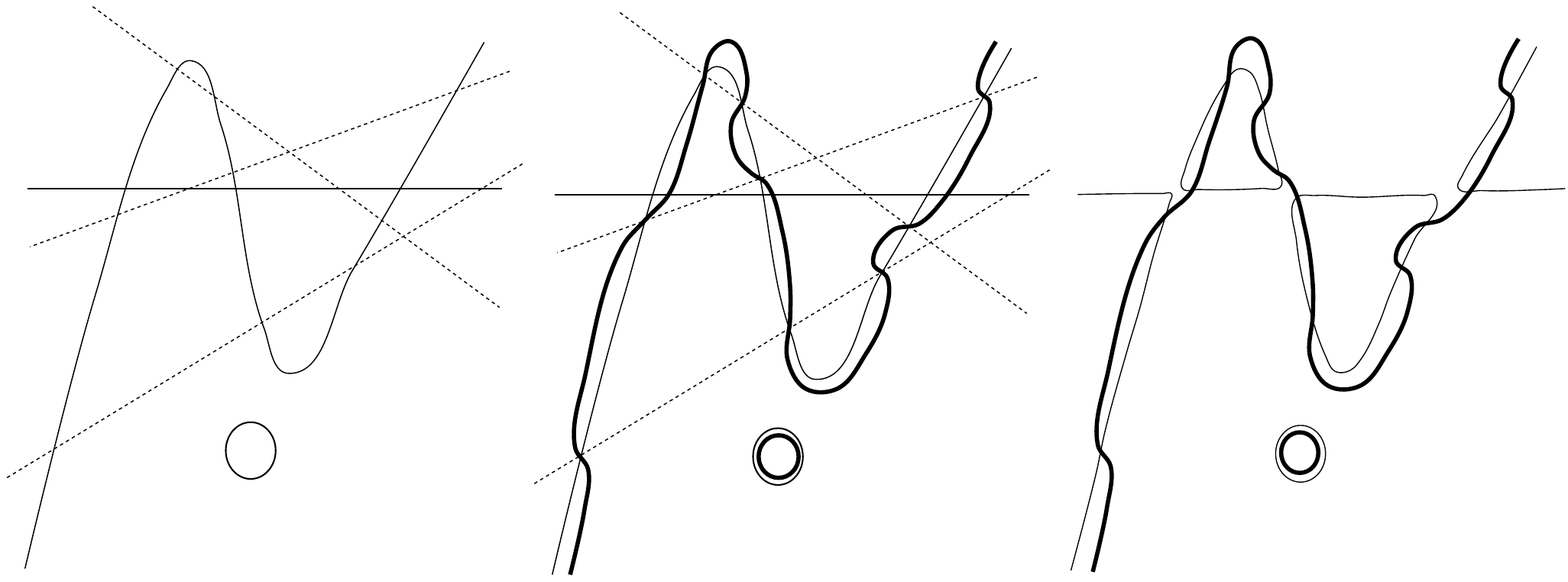}} 
\put(-70,-14){Fig. \ref{fig: 2.2.2.2_company_DP2}.1}
\put(30,-14){Fig. \ref{fig: 2.2.2.2_company_DP2}.2}
\put(127,-14){Fig. \ref{fig: 2.2.2.2_company_DP2}.3}
\end{picture}
\end{center}
\caption{$1-2$: Intermediate constructions on $\mathbb{R}P^2$. $3$: Mutual arrangement on $\mathbb{R}P^2$ of a real maximal cubic (in thick black) and a real maximal quartic (in gray).}
\label{fig: 2.2.2.2_company_DP2}
\end{figure}
\begin{figure}[h!]
\begin{center}
\begin{picture}(100,130)
\put(-150,0){\includegraphics[width=1.1\textwidth]{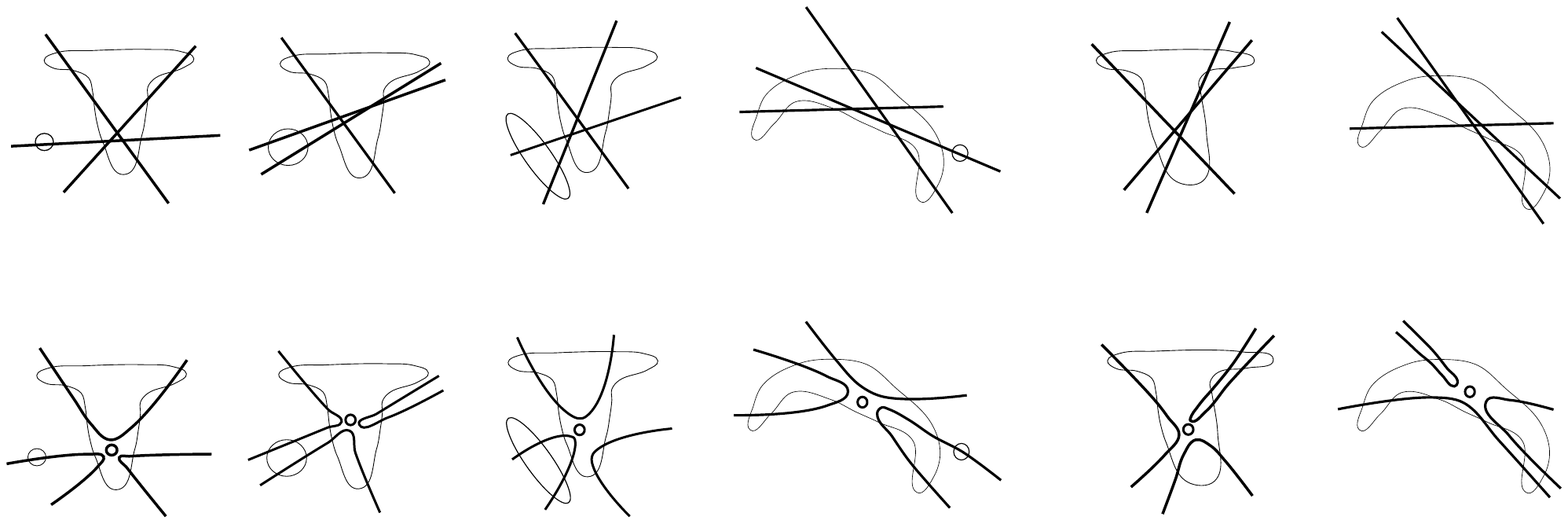}} 
\put(-135,67){Fig. \ref{fig: giuro_DP2}.1}
\put(-77,67){Fig. \ref{fig: giuro_DP2}.2}
\put(-20,67){Fig. \ref{fig: giuro_DP2}.3}
\put(48,67){Fig. \ref{fig: giuro_DP2}.4}
\put(133,67){Fig. \ref{fig: giuro_DP2}.5}
\put(198,67){Fig. \ref{fig: giuro_DP2}.6}
\put(-135,-14){Fig. \ref{fig: giuro_DP2}.i}
\put(-77,-14){Fig. \ref{fig: giuro_DP2}.ii}
\put(-20,-14){Fig. \ref{fig: giuro_DP2}.iii}
\put(48,-14){Fig. \ref{fig: giuro_DP2}.iv}
\put(133,-14){Fig. \ref{fig: giuro_DP2}.v}
\put(198,-14){Fig. \ref{fig: giuro_DP2}.vi}
\end{picture}
\end{center} 
\caption{$1-6$: Intermediate constructions on $\mathbb{R}P^2$. $i-vi$: Mutual arrangements on $\mathbb{R}P^2$ of real maximal cubics (in thick black) and real quartics $\overline{Q}$ (in gray).}
\label{fig: giuro_DP2}
\end{figure}
\begin{proof}
In \cite{Orev02}, Orevkov has constructed real maximal quartics $\overline{Q}$ and cubics $C$ arranged, up to isotopy, in $\mathbb{R}P^2$ as depicted in Fig. \ref{fig: orevkov_cubic_quartic_DP2}. To all such pairs correspond real algebraic curves of class $3$ in $4$-sphere real del Pezzo surfaces of degree $2$ realizing $12$ real schemes among those labeled with $\circ$ in Table \ref{tabella=realized3}. \\
Let us realize the real scheme $2:2:2:2$. There exist a real cubic $\tilde{C}$ and a real line $L$ in $\mathbb{C}P^2$ arranged in $\mathbb{R}P^2$ as represented in Fig. \ref{fig: 2.2.2.2_company_DP2}.1. Let $\tilde{p}(x,y,z)=0$ and $l(x,y,z)=0$ be real polynomial equations respectively defining $\tilde{C}$ and $L$. 
Pick three real lines $L_1$, $L_2$, $L_3$, as those depicted in dashed in Fig. \ref{fig: 2.2.2.2_company_DP2}.1. Take a small perturbation of $\tilde{C}$ replacing $\tilde{p}(x,y,z)$ with $p(x,y,z):=\tilde{p}(x,y,z)+ \varepsilon l_1(x,y,z)l_2(x,y,z)l_3(x,y,z)$, where $l_i(x,y,z)$ is a real polynomial defining $L_i$, with $i=1,2,3$, and $\varepsilon \not = 0$ is a sufficient small real number. Up to a choice of the sign of $\varepsilon$, the real curve $C$, defined by $p(x,y,z)=0$, is a real cubic arranged in $\mathbb{R}P^2$ as depicted (in thick black) in Fig. \ref{fig: 2.2.2.2_company_DP2}.2. Let $\bigcup_{i=1}^4 L_i$ be the union of four non-real lines pairwise complex conjugated and defined by a real polynomial $u(x,y,z)$. Take a small perturbation of $\tilde{C} \cup L$ replacing $\tilde{p}(x,y,z)l(x,y,z)$ with $\tilde{p}(x,y,z)l(x,y,z)+ \delta u(x,y,z)=0$, where $\delta\not = 0$ is a sufficient small real number. Up to the choice of the sign of $\delta$, such equation defines a non-singular real plane maximal quartic $\overline{Q}$ such that $Q\cup C$ is arranged in $\mathbb{R}P^2$ as pictured in Fig. \ref{fig: 2.2.2.2_company_DP2}.3. It follows that $2:2:2:2$ is realizable in $X^{4}$ and in class $3$.\\
Now, we end the proof realizing the real schemes in $\mathcal{S}_{DP2}(k,3)$ listed below. 
There exist a real quartic $\overline{Q}$ with real scheme $k$, where $2 \leq k \leq 4$ (resp. $1 \leq k \leq 4$), and three real lines arranged in $\mathbb{R}P^2$ as pictured in Fig. \ref{fig: giuro_DP2}.1 - \ref{fig: giuro_DP2}.4 (resp. Fig. \ref{fig: giuro_DP2}.5 and \ref{fig: giuro_DP2}.6); where we depict only the ovals of $\mathbb{R}\overline{Q}$ (in gray) intersecting the three lines (in thick black). 
Perturb the union of the three lines into a non-singular real cubic $C$ such that $C\cup \overline{Q}$ is arranged in $\mathbb{R}P^2$ respectively as depicted in Fig. \ref{fig: giuro_DP2}.i - \ref{fig: giuro_DP2}.vi; we depict only the ovals of $\mathbb{R}\overline{Q}$ intersecting the cubic (in thick black). From $C\cup \overline{Q}$, one realizes the following real schemes in $X^{k}$ and in class $3$:
\begin{itemize}
\item for $2 \leq k \leq 4$,
\begin{equation*}
2    \sqcup   \langle 4 \rangle:1:0:0, \quad
6:2:0:0, \quad
1   \sqcup   \langle 4 \rangle : 2:0:0, \quad
3    \sqcup   \langle1\rangle   \sqcup    \langle1\rangle:1:0:0 
\end{equation*}
\item for $1 \leq k \leq 4$,
\begin{equation*}
 1   \sqcup   \langle 6 \rangle :0:0:0, \quad \langle 1 \rangle   \sqcup   \langle 5 \rangle :0:0:0.
\end{equation*}
\end{itemize}
\end{proof}
\subsection{Symplectic curve on a $4$-sphere real symplectic degree $2$ del Pezzo surface}
\label{subsec: symplectic}
There exists a certain mutual arrangement in $\mathbb{R}P^2$ of a real symplectic cubic and a real symplectic quartic which is unrealizable algebraically; see \cite{Orev02}. Analogously to the algebraic case, one can construct from such arrangement in $\mathbb{R}P^2$ a real symplectic del Pezzo surface of degree $2$ and
 a real symplectic curve of class $3$ on it with topology prescribed from the arrangement on the real projective plane. 
\begin{proof}[Proof of Proposition \ref{prop: symplectic_curve_surface}]
Let us consider $(\mathbb{C}P^2, \omega_{std}, conj)$, where $\omega_{std}$ is the symplectic Fubini-Study $2$-form on $\mathbb{C}P^2$ and $conj:\mathbb{C}P^2 \rightarrow \mathbb{C}P^2$ is the standard real structure on $\mathbb{C}P^2$. Let $conj^*: H^2(\mathbb{C}P^2;\mathbb{Z}) \rightarrow H^2(\mathbb{C}P^2;\mathbb{Z})$ be the group homomorphism map induced by $conj$. It follows that $conj^*\omega_{std}=-\omega_{std}$. Due to \cite{Orev02}, there exist a non-singular real symplectic maximal quartic $\overline{Q}$ and a non-singular real symplectic maximal cubic $\overline{C}$ which are mutually arranged in $\mathbb{R}P^2$ as depicted in Fig. \ref{fig: pseudoholo_DP2}. The double cover $\overline{\phi}: \overline{X}\rightarrow \mathbb{C}P^2$ ramified along $\overline{Q}$ 
carries a natural symplectic structure $\omega$ such that $\omega=\overline{\phi}^*\omega_{std}$ (\cite{Gromo13},\cite{Aurou00}). 
Let $\sigma$ be one of the two lifts of $conj$ via the double ramified cover. Since $\overline{\phi}\circ \sigma = conj \circ \overline{\phi}$, we have that $\sigma^*\omega=-\omega$; namely $\sigma: \overline{X} \rightarrow \overline{X}$ is a real structure of $\overline{X}$. Then, up to choose $\sigma$, the surface $(\overline{X},\omega,\sigma)$ is real diffeomorphich to a $4$-sphere real del Pezzo surface of degree $2$ and, from $\overline{C}$, we construct a real symplectic curve of class $3$ in $\overline{X}$ realizing $2    \sqcup  \langle1\rangle    \sqcup  \langle1\rangle    \sqcup  \langle1\rangle:0:0:0$.
\end{proof}
\begin{figure}[h!]
\begin{center}
\begin{picture}(100,58)
\put(-60,-25){\includegraphics[width=0.6\textwidth]{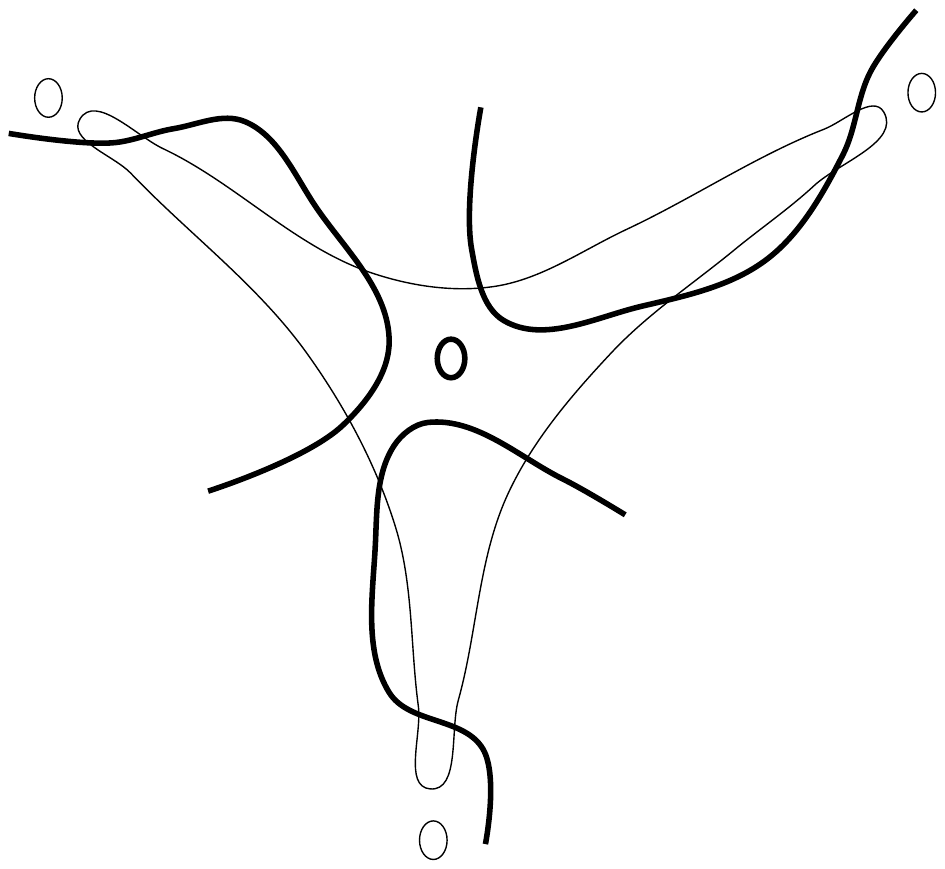}} 
\end{picture}
\end{center}
\caption{ }
\label{fig: pseudoholo_DP2}
\end{figure}
\subsection{Constructions via patchworking}
\label{subsec: difficult_constr_DP2}
In this section, we present a construction method that allows to construct (non-symmetric) real algebraic curves with prescribed topology in $k$-sphere real del Pezzo surfaces of degree $2$. First of all, let us give some definitions. \\
Let $Bl_{p_1,..,p_7} : S \rightarrow \mathbb{C}P^2$ be the blow-up of $\mathbb{C}P^2$ at a collection of $7$ points $p_1,..,p_7$ subject to the condition that all of them do not belong to a conic, $6$ of them belong to a conic, no $3$ of them belong to a line. 
Then, the strict transform of the conic passing through $6$ points of the collection 
is a smooth rational curve $E_S\subset S$ of self-intersection $(-2)$ in $S$. \\
Assume from now on, that $S$ contains a unique smooth rational curve of self intersection $(-2)$. The pair $(S,E_S)$ is called a $1$\textit{-nodal degree} $2$ \textit{del Pezzo pair}. \\
The anti-canonical system $\phi'$ of $S$ decomposes into a regular map $S \rightarrow S'$ of degree $1$ which contracts the $(-2)$-curve of $S$, and a double cover $S' \rightarrow \mathbb{C}P^2$ ramified along a quartic $\tilde{Q}$ with a double point as only singularity. Let us call the surface $S'$ a $1$-nodal del Pezzo surface of degree $2$. Conversely, the minimal resolution of the double cover of $\mathbb{C}P^2$ ramified along a quartic with a double point as only singularity is a $1$-nodal degree $2$ del Pezzo pair.\\
\par Let us equip $S$ with a real structure $\sigma '$, then $E_S$ is real. Assume that $\mathbb{R}S$ is homeomorphic to $\bigsqcup_{i=1}^jS^2$. It follows that the quartic $\tilde{Q} \subset \mathbb{C}P^2$ is real, it has a real non-degenerate double point as only singularity and that $\mathbb{R}\tilde{Q}$ consists of $j$ connected components of dimension $1$. Conversely, given such a quartic, one can construct a $1$-nodal degree $2$ del Pezzo pair $(S,E_S)$ where $S$ is equipped with a real structure such that $\mathbb{R}S$ is homeomorphic to $\bigsqcup_{i=1}^j S^2 $; see \cite{DegIteKha00}. The homology group $H_2^-(S;\mathbb{Z})$ is generated by $c_1(S)$ and $E_S$ (\cite{Russ02}).\\
Let us give some more definitions. 
\begin{defn}
\label{defn: 3-sfere-DP2-pair}
\begin{itemize}
\item[]
\item Let $(S,E_S)$ be a $1$-nodal degree $2$ del Pezzo pair. Let $S$ be equipped with a real structure $\sigma '$. If $\mathbb{R}S$ is homeomorphic to $\bigsqcup_{i=1}^j S^2 $, we say that $(S, E_S)$ is a $j$-\textit{sphere real} $1$-\textit{nodal degree} $2$ \textit{del Pezzo pair}.
\item Let $A\subset S$ be a real algebraic curve realizing the class $dc_1(S)+k[E_S] \in H_2(S;\mathbb{Z})$. Then, we say that $A$ has \textit{bi-class} $(d,k)$.
\end{itemize}
\end{defn} 
\begin{note}
\label{note: bouquet}
Let $s$ be a non-negative integer greater or equal to $2$. We denote with
$\bigvee_{j=1}^s S^n$ a bouquet of $s$ $n$-dimensional spheres.
\end{note}
\paragraph*{Topological construction:} Let $X_0'$ be a real reducible surface given by the union of two real algebraic surfaces $S$ and $T$, where 
\begin{enumerate}[label=(\arabic*)] 
\item $T$ is a non-singular real quadric surface;
\item $S$ contains a unique smooth rational $(-2)$-curve $E_S\subset S$ such that $(S,E_S)$ is a $j$-sphere real $1$-nodal degree $2$ del Pezzo pair, with $1\leq j \leq 3$; 
\item $S$ and $T$ intersect transversely along a real curve $E$ which is a bidegree $(1,1)$ real curve in $T$ and $E_S$ in $S$.
\end{enumerate}
Let $C_S \subset S$ and $C_T \subset T$ be non-singular real algebraic curves respectively of bi-class $(d,k)$ and of class 
$kE$ in $H_2(T;\mathbb{Z})$. Both $C_{S}$ and $C_{T}$ intersect transversely $E$ in the same real configuration of $2k$ distinct points; i.e. 
$$ E \cap C_{S}= E \cap C_{T}.$$
\par If $T$ is the quadric ellipsoid (resp. $T$ is the real quadric surface with empty real part) and $\mathbb{R}E = \emptyset$, the topological type $\mathcal{S}$ realized by $(\mathbb{R}S\cup \mathbb{R}T, \mathbb{R}C_S\cup \mathbb{R}C_T)$ is an arrangement of ovals in $\bigsqcup_{j=1}^{j+1}S^2$ (resp. $\bigsqcup_{j=1}^jS^2$). \\
\par Otherwise, if $T$ is the quadric ellipsoid (resp. $T$ is the quadric hyperboloid) and $\mathbb{R}E \simeq S^1$, from the topological type realized by the pair $(\mathbb{R}S\cup \mathbb{R}T, \mathbb{R}C_S\cup \mathbb{R}C_T)$, we can realize an arrangement $\mathcal{S}$ of ovals in $\bigsqcup_{j=1}^{j+1}S^2$ (resp. $\bigsqcup_{j=1}^{j}S^2$) as follows. \\
Locally $\mathbb{R}T\cap \mathbb{R}S$ is given as the intersection of two real planes as depicted in Fig. \ref{fig: 1smoothing_family_DP2}.1, and $(\mathbb{R}S \cup \mathbb{R}T) \setminus \mathbb{R}E$ has $4$ connected components $W_1,$ $W_2 \subset \mathbb{R}S$ and $H_1,$ $H_2 \subset \mathbb{R}T$ (Fig. \ref{fig: 1smoothing_family_DP2}.2). We can glue $W_1$ either to $H_1$ or to $H_2$ along $\mathbb{R}E$. After making a choice for $W_1$, we glue $W_2$ to the remaining connected component along $\mathbb{R}E$ (Fig. \ref{fig: 1smoothing_family_DP2}.3). Either choices of gluing the four connected components give us the disjoint union of $j+1$ (resp. $j$) spheres, and from $\mathbb{R}C_S \cup \mathbb{R}C_T$ we get an arrangement $\mathcal{S}$ of ovals in $\bigsqcup_{j=1}^{j+1}S^2$ (resp. $\bigsqcup_{j=1}^{j}S^2$). Example in Fig. \ref{fig: 1smoothing_family_DP2}.4 and \ref{fig: 1smoothing_family_DP2}.5.
\begin{figure}[h!]
\begin{center}
\begin{picture}(100,110)
\put(-135,5){\includegraphics[width=1.0\textwidth]{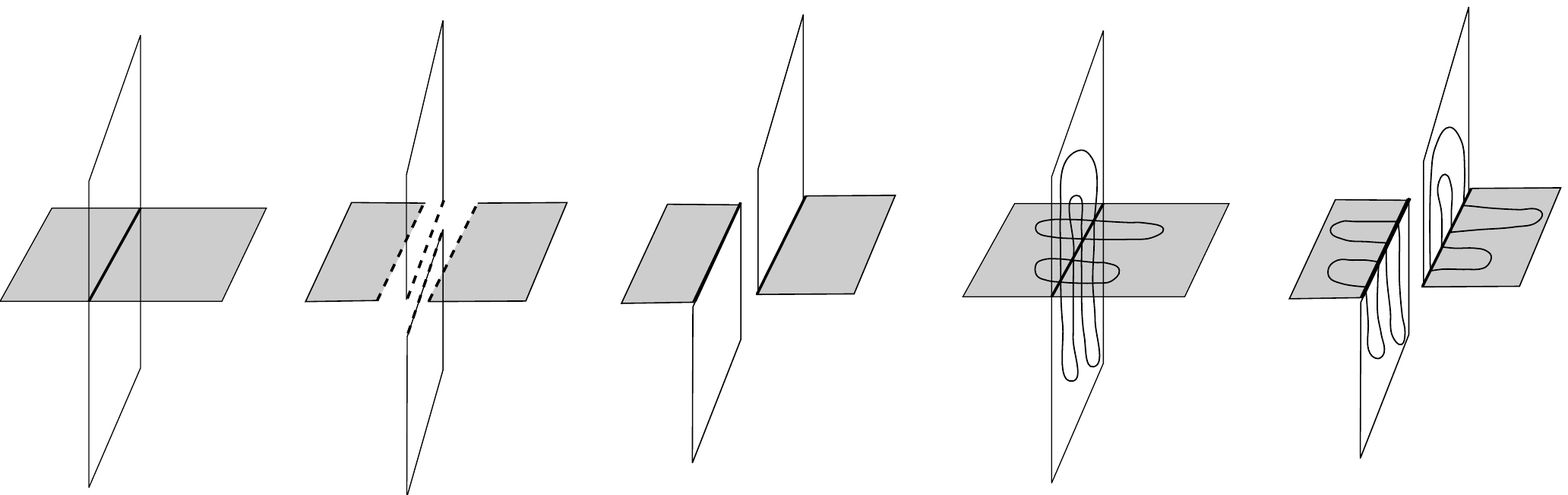}} 
\put(-123,-13){Fig. \ref{fig: 1smoothing_family_DP2}.1}
\put(-50,-13){Fig. \ref{fig: 1smoothing_family_DP2}.2}
\put(20,-13){Fig. \ref{fig: 1smoothing_family_DP2}.3}
\put(101,-13){Fig. \ref{fig: 1smoothing_family_DP2}.4}
\put(178,-13){Fig. \ref{fig: 1smoothing_family_DP2}.5}
\end{picture}
\end{center}
\caption{}
\label{fig: 1smoothing_family_DP2}
\end{figure}
\par By Theorem \ref{thm: weak_patch_DP2}, such topological construction is realizable algebraically. \\
The proof of Theorem \ref{thm: weak_patch_DP2} requires the existence of a real flat one-parameter family whose general fibers are ($j+1$)-sphere, respectively $j$-sphere real del Pezzo surfaces of degree $2$ and whose central fiber is $X_0'$. We prove the existence of such a family in Corollary \ref{cor: esistenza_family_DP2}. 
\begin{thm}
\label{thm: weak_patch_DP2}
The real scheme $\mathcal{S}$ is realizable in $X^{k}$ and in class $d$, with $k=j$, respectively $k=j+1$.
\end{thm}
\begin{proof}
Due to Corollary \ref{cor: esistenza_family_DP2}, we can put $X_0'$ in a real flat one-parameter family $\tilde{\pi}': \mathfrak{X}'\rightarrow D(0)$, where $\mathfrak{X}'$ is a $3$-dimensional real algebraic variety and $D(0)\subset \mathbb{C}$ is a real disk centered at $0$, such that the fibers $X_t'=:\tilde{\pi}'^{-1}(t)$ are ($k$-sphere real) non-singular degree $2$ del Pezzo surfaces (with $k=j$, resp. $k=j+1$), for $t\not= 0$ (and $t$ real), and the central fiber is $X_0'$. By Ramanujan's Vanishing theorem (\cite[Section 8]{Dolg12})
$$H^1(X_0'; \mathcal{O}_{X_0'}(C_0))=0;$$ 
then, \cite[Theorem 2.8]{ShuTyo06I} assures the existence of an open neighborhood $U(0)\subset D(0)$ and a deformation $C_t$ in $\tilde{\pi}'^{-1}(t)$ such that $C_t$ are non-singular (real) curves in $X_t'$ for $t$ (real) in $U(0)\setminus \{0\}$. Moreover, there exists a real $\tilde{t}\in U(0)\setminus \{0\}$ such that the pair $( \mathbb{R}X_{ \tilde{t} }',\mathbb{R}C_{\tilde{t}})$ realizes the real scheme $\mathcal{S}$. 
\end{proof}
To prove Corollary \ref{cor: esistenza_family_DP2}, we need the following proposition. In the proof we make use of a type of construction presented in \cite{Atiy58}, which found recent applications in real enumerative geometry; see \cite{BruPui13}. Similar constructions can be found in \cite{MoiTei88}, \cite{MoiTei90} and \cite{MoiTei94}. 
\begin{prop}
\label{prop: existence_DP2_family}
Let $\tilde{Q}$ be a real quartic in $\mathbb{C}P^2$ with a real non-degenerate double point $q$ as only singularity and such that $\mathbb{R}Q$ is homeomorphic to either $\bigsqcup_{i=1}^{j} S^1 \sqcup\{q \}$ or to $\bigsqcup_{i=1}^{j-1} S^1 \sqcup \bigvee_{j=1}^2 S^1$, where $1\leq j \leq 3$. Then, there exists a real flat one-parameter family of non-singular ($k$-sphere real) degree $2$ del Pezzo surfaces (with $k=j$, resp. with $k=j+1$) but the central fiber which is a real reducible surface $X_0'$ equal to the union of two real algebraic surfaces $S \cup T$, where 
\begin{enumerate}[label=(\arabic*)]
\item $T$ is either a real quadric hyperboloid or a real quadric surface with empty real part (resp. a real quadric ellipsoid);
\item $S$ is the minimal resolution of the double cover of $\mathbb{C}P^2$ ramified along $\tilde{Q}$ and it contains a unique smooth rational curve $E_S\subset S$ such that $(S,E_S)$ is a $j$-sphere $1$-nodal degree $2$ del Pezzo pair, with $1 \leq j \leq 3$; 
\item $S$ and $T$ intersect transversely along a curve $E$ which is a bidegree $(1,1)$ real curve in $T$ and the $(-2)$-curve $E_S$ in $S$.
\end{enumerate}
\end{prop}
\begin{proof}
Let $f(x,y,z)=0$ be a real polynomial equation defining the real quartic $\tilde{Q}$ in $\mathbb{C}P^2$. 
Up to multiply $f(x,y,z)$ by $-1$, we can always put $\tilde{Q}$ in a real flat one-parameter family $\pi: \mathfrak{Q} \rightarrow D(0)$, where 
\begin{itemize}
\item $D(0)\subset \mathbb{C}$ equipped with the standard real structure of $\mathbb{C}$, is a real disk centered at $0$;
\item $\mathfrak{Q} \subset \mathbb{C}P^2 \times D(0)$ is defined by $f(x,y,z) +\varepsilon z^4t^2=0,$ with $\varepsilon \in \{1,-1\}$ 
\end{itemize}
and such that 
\begin{itemize}
\item the fibers $\overline{Q}_t:=\pi^{-1}(t)$ are non-singular (real) quartics (with real part homeomorphic to either $\bigsqcup_{i=1}^{j} S^1 $ or to $\bigsqcup_{i=1}^{j+1} S^1$, and $\Pi_+$ is orientable) for $t\not = 0$ (and $t$ real);
\item $\overline{Q}_0=\tilde{Q}$. 
\end{itemize}
From the family of quartics, we can construct a real flat one-parameter family $\tilde{\pi}: \mathfrak{X} \rightarrow D(0)$ such that 
\begin{itemize}
\item $\mathfrak{X}$ is the double cover of $\mathbb{C}P^2 \times D(0)$ ramified along $\mathfrak{Q}$ and $\mathfrak{X}$ is isomorphic to the hypersurface in $\mathbb{C}P(1,1,1,2) \times D(0)$ defined by the polynomial equation $f(x,y,z)+\varepsilon z^4t^2= w^2;$ 
\item $X_0$ is the double cover of $\mathbb{C}P^2$ ramified along $\tilde{Q}$. Depending on the real scheme realized by the pair $(\mathbb{R}P^2,\mathbb{R}\tilde{Q})$, the real part of $X_0$ is homeomorphic either to $\bigsqcup_{i=1}^{j} S^2 \sqcup \{pt \}$ or to $\bigsqcup_{i=1}^{j-1} S^2 \sqcup \bigvee_{j=1}^2 S^2$, where $\{pt\}$ is a point.
\item the fibers $\tilde{\pi}^{-1}(t):=X_t$ are non-singular ($k$-sphere real) degree $2$ del Pezzo surfaces (with either $k=j$ or $k=j+1$ depending on $\mathbb{R}\overline{Q}_t$), for $t \not =0$ (and $t$ real).
\end{itemize}
Now, performing the blow up $Bl_p: \mathfrak{X}' \rightarrow \mathfrak{X}$ at the node $p$ of $\mathfrak{X}$, we obtain a real flat one-parameter family $\tilde{\pi}': \mathfrak{X}'\rightarrow D(0)$ such that 
$Bl_p^{-1}(p)=:T$ 
is a real quadric surface with real structure dependent on $f(x,y,z)$ and $\varepsilon$, 
the fibers $\tilde{\pi}'^{-1}(t):=X_t'$ are non-singular ($k$-sphere real) del Pezzo surfaces of degree $2$ (with either $k=j$ or $k=j+1$ depending on $\mathbb{R}X_t$), for $t \not =0$ (and $t$ real), and $X_0'$ is equal to the union of two real algebraic surfaces $S \cup T$, where $S$ and $T$ are as described in $(1)-(3)$. 
\end{proof}
\begin{cor}
\label{cor: esistenza_family_DP2}
Let $X_0'$ be a real reducible surface equal to the union of two real algebraic surfaces $S \cup T$, where $S$ and $T$ are as described in $(1)-(3)$ of Proposition \ref{prop: existence_DP2_family}. 
Then, there exists a real flat one-parameter family $\tilde{\pi}': \mathfrak{X}' \rightarrow D(0)$, where $D(0)\subset \mathbb{C}$ is a real disk centered in $0$, the fibers $\tilde{\pi}'^{-1}(t):=X_t'$ are non-singular ($k$-sphere real) del Pezzo surfaces of degree $2$ (with $k=j$, resp. $k=j+1$), for $t \not =0$ (and $t$ real), and the central fiber is $X_0'$.
\end{cor}
\begin{proof}
The anti-canonical system of $S$ is the minimal resolution of the double cover of $\mathbb{C}P^2$ ramified along a real algebraic quartic $\tilde{Q}$ with a real non-degenerate double point $q$ as only singularity and such that $\mathbb{R}\tilde{Q}$ is homeomorphic to either $\bigsqcup_{i=1}^{j} S^1 \sqcup\{q \}$ or to $\bigsqcup_{i=1}^{j-1} S^1 \sqcup \bigvee_{j=1}^2 S^1$, where $1\leq j \leq 3$. Applying the proof of Proposition \ref{prop: existence_DP2_family} to $\tilde{Q}$, we prove the statement. 
\end{proof}
\subsubsection{Intermediate constructions: Constructions on quadric surfaces and on $j$-sphere real $1$-nodal del Pezzo pairs of degree 2} 
\label{subsubsec: intermediate_construction_DP2} 
In Section \ref{subsec: final_constr_DP2}, we use Theorem \ref{thm: weak_patch_DP2} to end the proof of Theorem \ref{thm: principal_DP2}, Proposition \ref{prop: principal_DP2} and to prove Proposition \ref{prop: non-symm_real_scheme}. In order to do that we need some intermediate constructions. Therefore, the aim of this section is to
\begin{itemize}
\item construct real algebraic curves with prescribed topology and intersection with a given real curve in the quadric ellipsoid (Proposition \ref{prop: constr_quadric_DP2}), resp. in the quadric hyperboloid (Proposition \ref{prop: constr_quadric_DP2_3}); 
\item construct real algebraic curves with prescribed topology on $j$-sphere real $1$-nodal degree $2$ del Pezzo pairs (Proposition \ref{prop: S_pencil_quartic_DP2}).
\end{itemize} 
\begin{note}
\label{note: point_p}
As mentioned in Section \ref{subsec: cha2_codage_isotopie}, the interior and exterior of any oval $\mathcal{D}$ in $S^2$ is not well defined because it depends on a choice of a point in $S^2 \setminus \mathcal{D}$. \\
In the proof of Propositions \ref{prop: constr_quadric_DP2} and \ref{prop: constr_quadric_DP2_2}, whenever we talk about interior and exterior of an oval in $S^2$ is with respect to a chosen point $p \in S^2 \setminus \mathcal{D}$ and, to be clear on the choice, we depict $S^2$ projected from such $p$ on a plane in Fig. \ref{fig: quadric_ellips_DP2} $-$ \ref{fig: quadric_ellips_DP2_non-symm}.
Recall that every oval in $S^{2}$ bounds two disks. Then, we say that the disk containing $p$ is the exterior, and the other one is the interior.
\end{note}
In the proofs of Propositions \ref{prop: constr_quadric_DP2}, \ref{prop: constr_quadric_DP2_2}, \ref{prop: constr_quadric_DP2_3}, we use variants of Harnack's construction method (\cite{Harn76}).
\begin{prop}
\label{prop: constr_quadric_DP2}
Let $T$ be the quadric ellipsoid and let $E_T$ be a non-singular real algebraic curve of bidegree $(1,1)$ in $T$. Then, for any real configuration $\mathcal{P}_{2k}$ of $2k$ distinct points in $E_T$ fixed as follows, there exists a non-singular real algebraic curve $C_T$ of bidegree $(k,k)$ on $T$, intersecting transversely $E_T$ in the $2k$ points and such that the triplet $(\mathbb{R}T, \mathbb{R}E_T, \mathbb{R}C_T)$ is arranged respectively as depicted:
\begin{enumerate}[label=(\arabic*)]
\item in Fig. \ref{fig: quadric_ellips_DP2}.1  for $k=2$ and $4$ fixed real points;
\item in Fig. \ref{fig: quadric_ellips_DP2}.2 for $k=3$ and no fixed real points;
\item in Fig. \ref{fig: quadric_ellips_DP2}.3  for $k=3$ and $2$ fixed real points;
\end{enumerate}
\end{prop}
\begin{figure}[h!]
\begin{center}
\begin{picture}(100,165)
\put(-107,92){Fig. \ref{fig: quadric_ellips_DP2}.1}
\put(79,92){Fig. \ref{fig: quadric_ellips_DP2}.2}
\put(-140,-15){\includegraphics[width=1.0\textwidth]{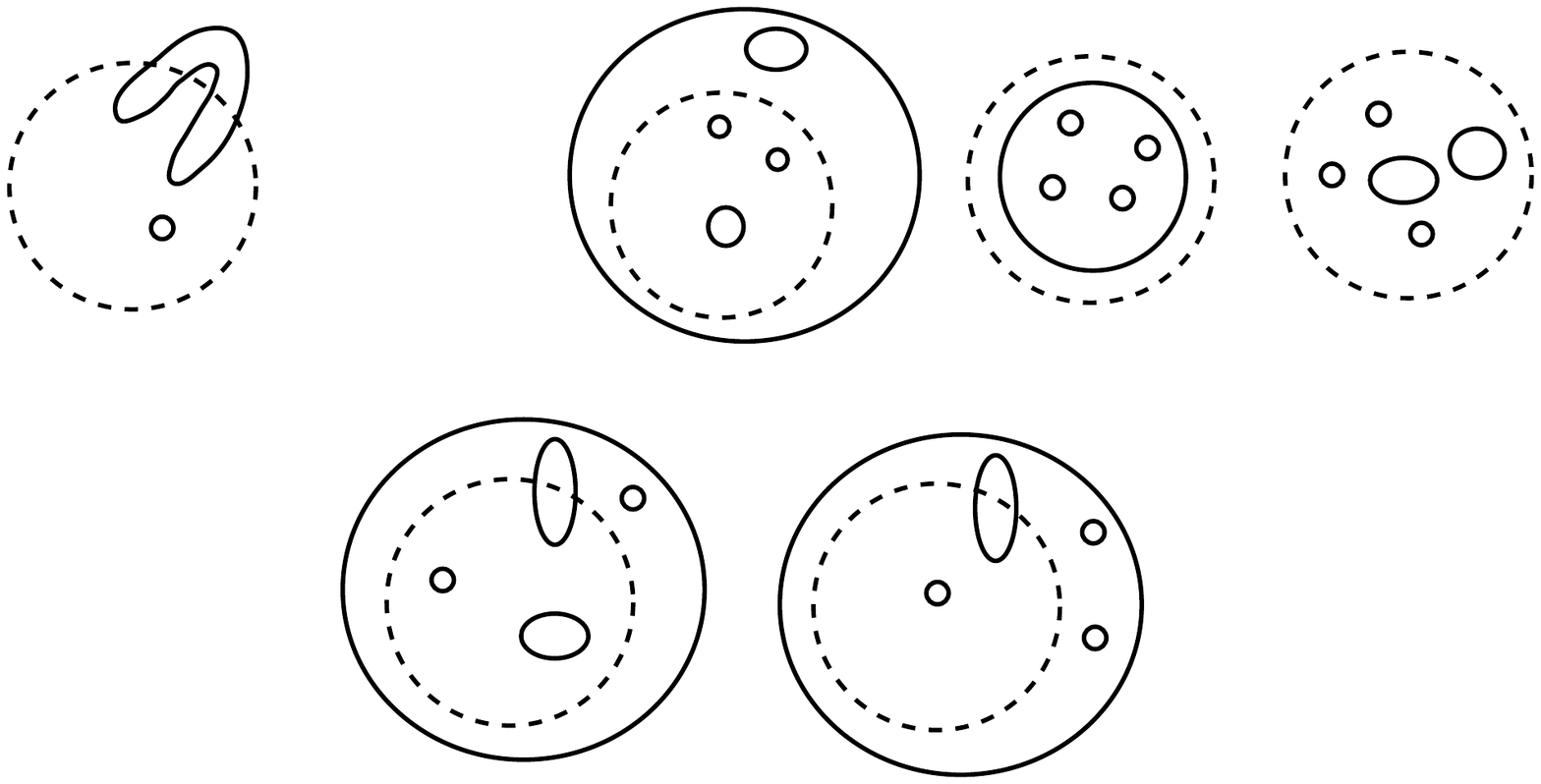}} 
\put(25,-12){Fig. \ref{fig: quadric_ellips_DP2}.3}
\end{picture}
\end{center}
\caption{$\mathbb{R}E_T$ in dashed.}
\label{fig: quadric_ellips_DP2}
\end{figure}
\begin{proof}
For any real configuration $\mathcal{P}_4\subset \mathbb{R}E_T$ let us construct a real curve $\tilde{H}$ of bidegree $(2,2)$ passing through $\mathcal{P}_4$ and such that the arrangement of $\mathbb{R}\tilde{H} \cup \mathbb{R}E_T$ is as depicted in Fig. \ref{fig: quadric_ellips_DP2}.1. First of all, remind that for any $2$ fixed distinct points on $E_T$, there exists a bidegree $(1,1)$ real algebraic curve passing through them.\\ 
Let $P_0(x,y)P_1(x,y)=0$ be a real polynomial equation defining the union of $E_T$ and a bidegree $(1,1)$ real curve $H$ such that the points of $\mathcal{P}_4$ belong to one connected component $\mathcal{E}$ of $\mathbb{R}E_T\setminus \mathbb{R}H$.  
Let $H_1$ and $H_2$ be two bidegree $(1,1)$ real curves such that $H_1 \cup H_2$ contains $\mathcal{P}_4$. Replace the left side of the equation $P_0(x,y)P_1(x,y)=0$ with $P_0(x,y)P_1(x,y)+ \varepsilon f_1(x,y)f_2(x,y)$, where $f_i(x,y)=0$ is an equation for $H_i$ and $\varepsilon$ is a sufficient small real number. Up to a choice of the sign of $\varepsilon$, one constructs a small perturbation $\tilde{H}$ of $E_T \cup H$, where $\tilde{H}$ is a bidegree $(2,2)$ non-singular real curve such that $\bigcup_{i=1}^2H_i \cap E_T = \tilde{H} \cap E_T $ and the triplet $(\mathbb{R}T, \mathbb{R}E_T, \mathbb{R}\tilde{H})$ is arranged as depicted in Fig. \ref{fig: quadric_ellips_DP2}.1. \\
Now, for any real configuration $\mathcal{P}_6 \subset E_T \setminus \mathbb{R}E_T$ with no real points, respectively 
with exactly $2$ real points, we want to construct real curves $C_T$ of bidegree $(3,3)$ passing through $\mathcal{P}_6$
and such that the arrangement of $\mathbb{R}C_T \cup \mathbb{R}E_T$ is as depicted in Fig. \ref{fig: quadric_ellips_DP2}.2 on the left, respectively in Fig. \ref{fig: quadric_ellips_DP2}.3. One can construct $C_T$ applying a small perturbation to $E_T \cup \tilde{H}$, respectively to $E_T \cup \tilde{H}_2$, where $\tilde{H}_2$ is a real curve of bidegree $(2,2)$ such that:
\begin{itemize}
\item the triplet $(\mathbb{R}T, \mathbb{R}E_T, \mathbb{R}\tilde{H}_2)$ is arranged as depicted in Fig. \ref{fig: quadric_ellips_DP2}.1;
\item the real points of $\mathcal{P}_6$ belong to an interior connected component, respectively an exterior connected component, of $\mathbb{R}E_T \setminus \mathbb{R}\tilde{H}_2$; 
see Notation \ref{note: point_p}.
\end{itemize}
Let us end the proof constructing real curves $C_T$ of bidegree $(3,3)$ 
\begin{itemize}
\item containing a given real configuration $\mathcal{P}'=\{p_1,\overline{p}_1,p_2,\overline{p}_2,p_3,\overline{p}_3\}$ of points on $E_T \setminus \mathbb{R}E_T$, where $p_i$ and 
$\overline{p}_i$ are complex conjugated points;
\item such that the arrangement of $\mathbb{R}C_T \cup \mathbb{R}E_T$ is as depicted in Fig. \ref{fig: quadric_ellips_DP2}.2 respectively in the center and on the right. 
\end{itemize}
Let $\Pi_i$ be a pencil of hyperplanes with base points $p_i$ and $\overline{p}_i$, with $i=1,2,3$. We want to show that one can always construct a non-singular real algebraic curve $C_T$ of bidegree $(3,3)$ as perturbation of the union of three hyperplanes respectively of $\Pi_1$, $\Pi_2$ and $\Pi_3$ such that the arrangement of the triplet $(\mathbb{R}T, \mathbb{R}E_T, \mathbb{R}C_T)$ is respectively as depicted in Fig. \ref{fig: particular_constructions_puttane_elissoide_DP2}.2. Namely, we prove that one can always find three hyperplanes respectively of $\Pi_1$, $\Pi_2$ and $\Pi_3$ whose union and real arrangement with respect to $\mathbb{R}E_T$ is respectively as in Fig. \ref{fig: particular_constructions_puttane_elissoide_DP2}.1.
\begin{figure}[h!]
\begin{center}
\begin{picture}(100,70)
\put(-85,-5){\includegraphics[width=0.7\textwidth]{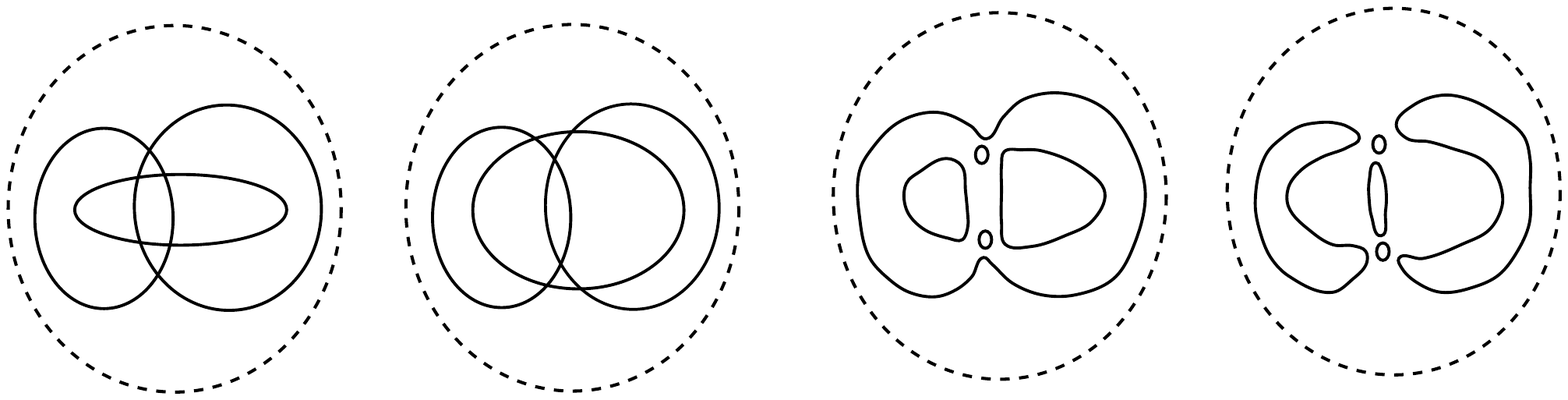}} 
\put(-35,-15){Fig. \ref{fig: particular_constructions_puttane_elissoide_DP2}.1}
\put(98,-15){Fig. \ref{fig: particular_constructions_puttane_elissoide_DP2}.2}
\end{picture}
\end{center}
\caption{$\mathbb{R}E_T$ in dashed.}
\label{fig: particular_constructions_puttane_elissoide_DP2}
\end{figure}
First of all, remark that on each of the two connected components of $\mathbb{R}T \setminus \mathbb{R} E_T$, the real part of the real hyperplanes of the pencil $\Pi_i$ vary from a real point $q_i$ to $\mathbb{R}E_T$, with $i=1,2,3$. Moreover, the real points $q_1$, $q_2$ and $q_3$ are distinct points.
\begin{figure}[h!]
\begin{center}
\begin{picture}(100,115)
\put(-85,5){\includegraphics[width=0.7\textwidth]{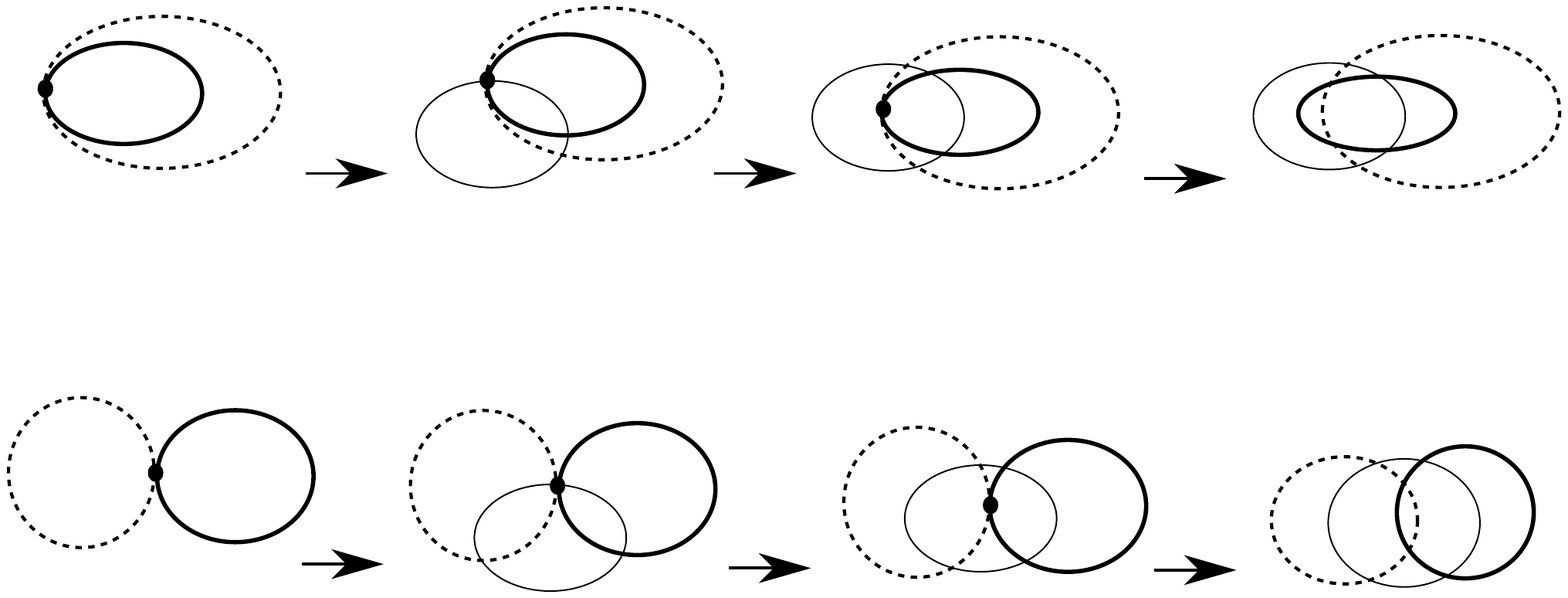}} 
\put(-80,60){Fig. \ref{fig: case3_puttane_elissoide_DP2}.1}
\put(-10,60){Fig. \ref{fig: case3_puttane_elissoide_DP2}.2}
\put(49,60){Fig. \ref{fig: case3_puttane_elissoide_DP2}.3}
\put(122,60){Fig. \ref{fig: case3_puttane_elissoide_DP2}.4}
\put(-80,-7){Fig. \ref{fig: case3_puttane_elissoide_DP2}.5}
\put(-10,-7){Fig. \ref{fig: case3_puttane_elissoide_DP2}.6}
\put(49,-7){Fig. \ref{fig: case3_puttane_elissoide_DP2}.7}
\put(122,-7){Fig. \ref{fig: case3_puttane_elissoide_DP2}.8}
\end{picture}
\end{center}
\caption{$H_{j}$ and $H_{k}$ in dashed and thick black, $H_{i}$ in gray.}
\label{fig: case3_puttane_elissoide_DP2}
\end{figure}
There exist two real hyperplanes $H_j \subset \Pi_j$ and $H_k \subset \Pi_k$ such that 
\begin{itemize}
\item $\mathbb{R}H_j$ is tangent to $\mathbb{R}H_k$ at a real point $s_{jk}$; 
\item the point $q_i$ do not belong to the interior of $\mathbb{R}H_j$ and $\mathbb{R}H_k$  
(see Notation \ref{note: point_p}); 
\item $\mathbb{R}H_j \cup \mathbb{R}H_k$ is as depicted in Fig. \ref{fig: case3_puttane_elissoide_DP2}.1 (resp. in Fig. \ref{fig: case3_puttane_elissoide_DP2}.5). 
\end{itemize}
Pick the real hyperplane $H_i \subset \Pi_i$  passing through $s_{jk}$. Then, the real part of $H_j\cup H_k \cup H_i$ is as depicted in Fig. \ref{fig: case3_puttane_elissoide_DP2}.2 (resp. in Fig. \ref{fig: case3_puttane_elissoide_DP2}.6). It follows that there exists a real hyperplane of the pencil $\Pi_i$ whose real arrangement with respect to $\mathbb{R}H_j \cup \mathbb{R} H_k$ is as depicted in Fig. \ref{fig: case3_puttane_elissoide_DP2}.3 (resp. in Fig. \ref{fig: case3_puttane_elissoide_DP2}.7). In conclusion, a small perturbation of the union of such three hyperplanes has real part as depicted in Fig. \ref{fig: case3_puttane_elissoide_DP2}.4 (resp. in Fig. \ref{fig: case3_puttane_elissoide_DP2}.8). 
\end{proof}
\begin{prop}
\label{prop: constr_quadric_DP2_2}
Let $T$ be the quadric ellipsoid and let $E_T$ be a non-singular real algebraic curve of bidegree $(1,1)$ in $T$. Then, 
\begin{itemize}
\item for any integer $k\geq 5$, 
\item for any real configuration $\mathcal{P}_{2k}$ of $2k$ distinct points, whose exactly $2$ are real, in $E_T$, 
\item for any given non-negative integers $k_1,$ $k_2$ such that $k_1+k_2=k-4$ 
\end{itemize}
there exists a non-singular real algebraic curve $C_T$ of bidegree $(k,k)$ on $T$, intersecting transversely $E_T$ in $\mathcal{P}_{2k}$ and such that the triplet $(\mathbb{R}T, \mathbb{R}E_T, \mathbb{R}C_T)$ is arranged as depicted in Fig. \ref{fig: non-symm_ellipsoid}.
\end{prop}
\begin{figure}[!h]
\begin{picture}(100,100)
\put(5,-19){\includegraphics[width=1.00\textwidth]{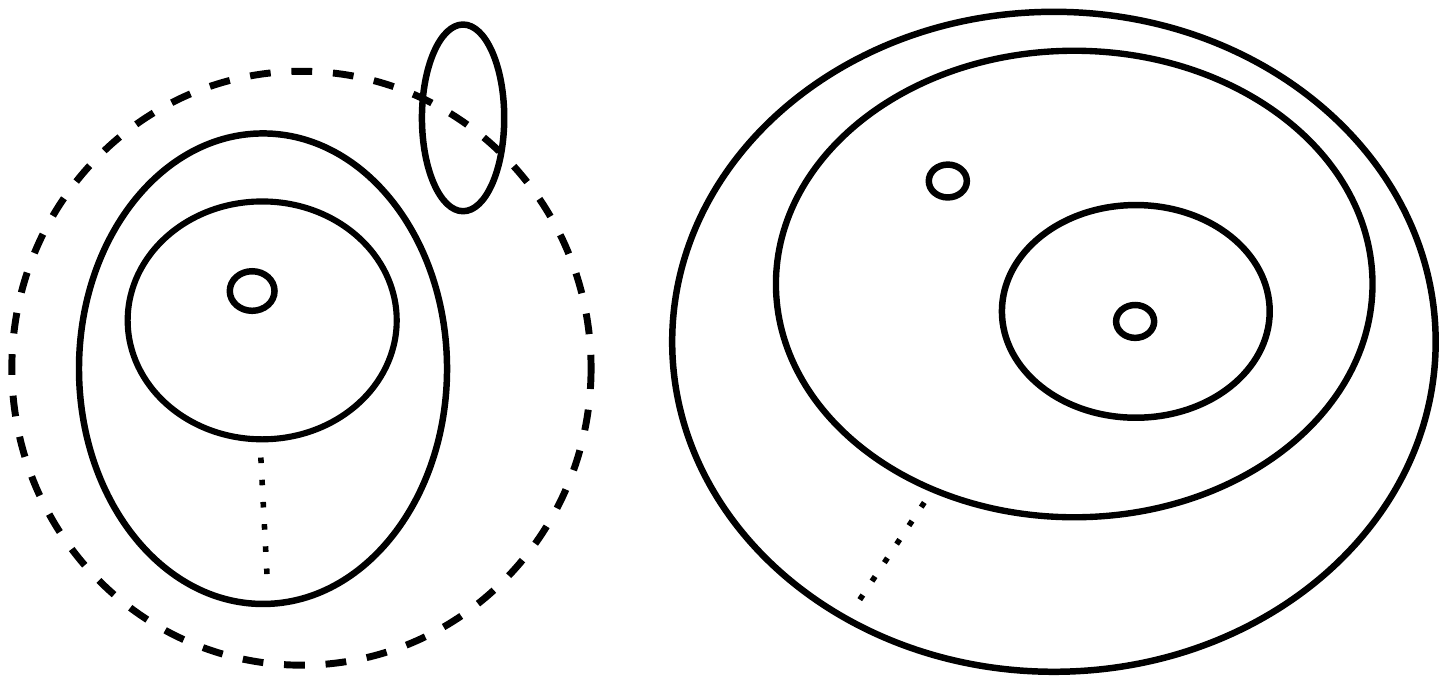}} 
\put(230,12){\rotatebox{-31}{$\textcolor{black}{\Big \} k_2}$}}
\put(115,22){$\textcolor{black}{\Big \} k_1}$}
\end{picture}
\caption{$\mathbb{R}E_T$ in dashed. The pair $(\mathbb{R}T, \mathbb{R}C_T)$ realizes $1   \sqcup   N(k_1,1)   \sqcup   N(k_2,1   \sqcup   \langle 1 \rangle)$.}
\label{fig: non-symm_ellipsoid}
\end{figure}
\begin{figure}[h!]
\begin{center}
\begin{picture}(100,70)
\put(-139,0){\includegraphics[width=1.0\textwidth]{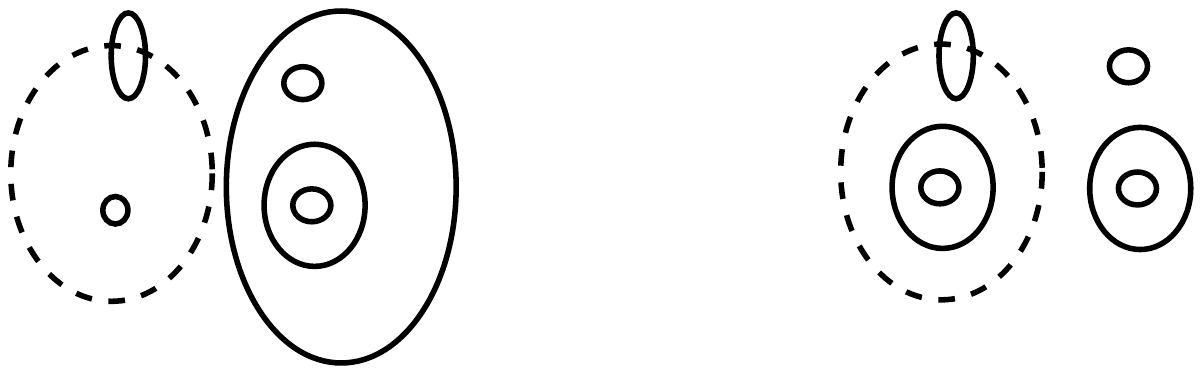}} 
\put(-80,-11){Fig. \ref{fig: k=5_quadric_ellips_DP2_non-symm}.1: $(\mathbb{R}T,\mathbb{R}E_T,\mathbb{R}\hat{H}_5)$}
\put(64,-11){Fig. \ref{fig: k=5_quadric_ellips_DP2_non-symm}.2: $(\mathbb{R}T,\mathbb{R}E_T,\mathbb{R}\hat{H}_5)$}

\end{picture}
\end{center}
\caption{$\mathbb{R}E_T$ in dashed.}
\label{fig: k=5_quadric_ellips_DP2_non-symm}
\end{figure}
\begin{figure}[h!]
\begin{center}
\begin{picture}(100,85)
\put(-121,-5){Fig. \ref{fig: quadric_ellips_DP2_non-symm}.1:}
\put(-134,-15){$(\mathbb{R}T,\mathbb{R}E_T,\mathbb{R}\hat{H}_2)$}
\put(-54,-5){Fig. \ref{fig: quadric_ellips_DP2_non-symm}.2}

\put(23,-5){Fig. \ref{fig: quadric_ellips_DP2_non-symm}.3: }
\put(10,-15){$(\mathbb{R}T,\mathbb{R}E_T,\mathbb{R}\hat{H}_3)$}
\put(93,-5){Fig. \ref{fig: quadric_ellips_DP2_non-symm}.4}

\put(177,-5){Fig. \ref{fig: quadric_ellips_DP2_non-symm}.5: }
\put(167,-15){$(\mathbb{R}T,\mathbb{R}E_T,\mathbb{R}\hat{H}_4)$}

\put(-125,0){\includegraphics[width=1.0\textwidth]{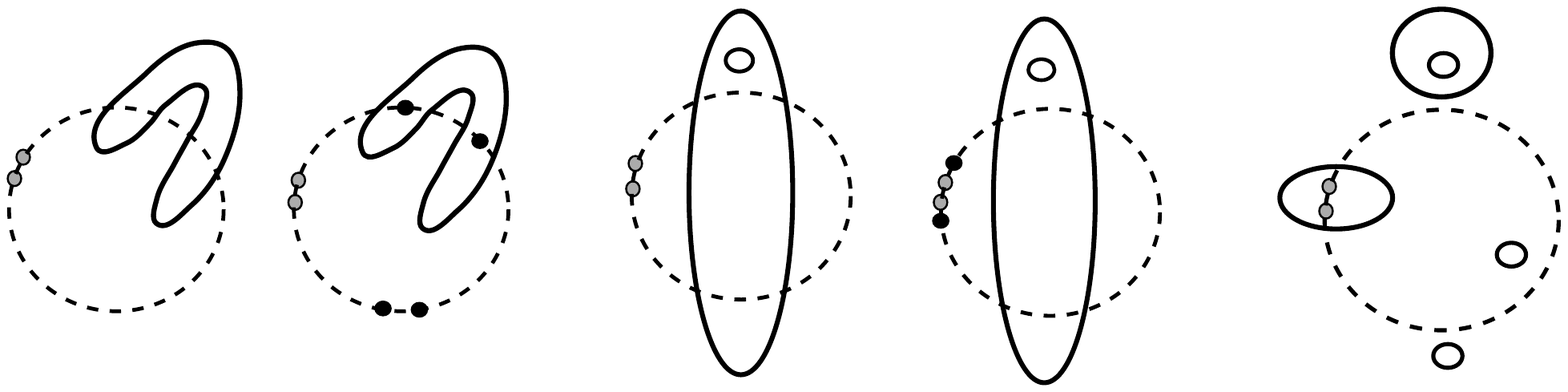}} 

\end{picture}
\end{center}
\caption{$\mathbb{R}E_T$ in dashed and $\{p_1,p_2\}=\{\textcolor{gray}{\bullet},\textcolor{gray}{\bullet}\}$.} 
\label{fig: quadric_ellips_DP2_non-symm}
\end{figure}

\begin{proof} 
The first step is to prove the statement for $k=5$. After we end the proof by induction on $k$. Let us introduce some notation. Let us call $\hat{H}_i$ any real curve of bidegree $(i,i)$ such that only one oval of $\mathbb{R}\hat{H}_i$, denoted with $\mathcal{D}_i$, intersects $\mathbb{R}E_T$. Moreover, we denote with $\mathcal{F}_1,\dots,\mathcal{F}_j$ and $\mathcal{\tilde{F}}_1,\dots, \mathcal{\tilde{F}}_j$ the connected components of $\mathbb{R}E_T \setminus \mathcal{D}_i$ respectively in the interior and in the exterior of $\mathcal{D}_i$; 
see Notation \ref{note: point_p}.\\ 
Fix $\mathcal{P}_{10} \subset E_T$ and denote with $p_1,p_2$ the two real points of $\mathcal{P}_{10}$. We start with the construction of a real curve $\hat{H}_5$ of bidegree $(5,5)$ passing through $\mathcal{P}_{10}$ and such that the triplet $(\mathbb{R}T,\mathbb{R}E_T,\mathbb{R}\hat{H}_5)$ is arranged respectively as depicted in Fig. \ref{fig: k=5_quadric_ellips_DP2_non-symm}.1 and \ref{fig: k=5_quadric_ellips_DP2_non-symm}.2.
If one can construct a real curve $\hat{H}_4$ of bidegree $(4,4)$ such that 
\begin{enumerate}[label=(\arabic*)]
\item $\mathbb{R}\hat{H}_4 \cap \mathbb{R}E_T= \mathcal{D}_4 \cap \mathbb{R}E_T$ consists of $2$ points; 
\item $\{p_1,p_2\} \subset \mathcal{F}_1$ of $\mathcal{D}_4$;
\item the triplet $(\mathbb{R}T, \mathbb{R} E_T, \mathbb{R}\hat{H}_4)$ is arranged as depicted in Fig. \ref{fig: quadric_ellips_DP2_non-symm}.5
\end{enumerate}
then, the real curve $\hat{H}_5$ exists as small perturbation of $\hat{H}_4 \cup E_T$; see the proof of Proposition \ref{prop: constr_quadric_DP2} for details on small perturbation method.\\
Let us construct $\hat{H}_4$. Fix a configuration $\mathcal{P}_4$ of $4$ points on $\mathbb{R}E_T$ such that $\{p_1,p_2\}$ belong to the same connected components of $\mathbb{R} E_T \setminus \mathcal{P}_4$. Via small perturbation, we construct a real curve $\hat{H}_2$ of bidegree $(2,2)$ such that 
\begin{itemize}
\item the triplet $(\mathbb{R}T, \mathbb{R} E_T, \mathbb{R}\hat{H}_2)$ is arranged as depicted in Fig. \ref{fig: quadric_ellips_DP2_non-symm}.1;
\item $\mathbb{R}\hat{H}_2 \cap \mathbb{R}E_T= \mathcal{D}_2 \cap \mathbb{R}E_T=\mathcal{P}_4$; 
\item $\{p_1,p_2\} \subset \tilde{\mathcal{F}_1}$ of $\mathcal{D}_2$.
\end{itemize}
Now, fix a real configuration $\mathcal{P}_6$ of $6$ points on $E_T$ such that (see Fig. \ref{fig: quadric_ellips_DP2_non-symm}.2) 
\begin{itemize}
\item exactly $4$ points are real; 
\item one of the fixed real points belongs to $\mathcal{F}_i \subset \mathcal{D}_2$, for $i=1,2$;
\item two of the fixed real points are on $\mathcal{\tilde{F}}_1 \subset \mathcal{D}_2$ and $p_1,p_2$ belong to the same connected component of $\mathcal{\tilde{F}}_1 \setminus \mathcal{D}_2$.
\end{itemize} 
One can construct a small perturbation $\hat{H}_3$ of $E_T \cup \hat{H}_2$, where $\hat{H}_3$ is a bidegree $(3,3)$ non-singular real curve passing through $\mathcal{P}_6$ and such that the triplet $(\mathbb{R}T, \mathbb{R}E_T, \mathbb{R}\hat{H}_3)$ is arranged as depicted in Fig. \ref{fig: quadric_ellips_DP2_non-symm}.3. To end the construction of $\hat{H}_4$, fix $8$ points on $E_T$ such that exactly $2$ points are real and belong to two different connected components of $\tilde{\mathcal{F}_1} \setminus \{p_1,p_2\} \subset \mathcal{D}_3$; see Fig. \ref{fig: quadric_ellips_DP2_non-symm}.4. One obtains $\hat{H}_4$ as small perturbation of $\hat{H}_3 \cup E_T$.
\par Let us proceed by induction to end the proof. Assume that the statement hold for $k-1$. Now, for any given non-negative integers $\tilde{k}_1,$ $\tilde{k}_2$ such that $\tilde{k}_1+\tilde{k}_2=k-4$, 
let us construct a bidegree $(k,k)$ real algebraic curve $C_T$ passing through a given real configuration $\mathcal{P}_{2k} \subset E_T$, whose exactly $2$ points are real and such that the triplet $(\mathbb{R}T,\mathbb{R}E_T,\mathbb{R}C_T)$ is arranged as depicted in Fig. \ref{fig: non-symm_ellipsoid}.\\
By induction, for any choice of $\tilde{k}_1,$ $\tilde{k}_2$ and $\mathcal{P}_{2k}$, there exists a bidegree $(k-1,k-1)$ real algebraic curve $\hat{H}_{k-1}$ such that
\begin{itemize}
\item $\hat{H}_{k-1}$ passes through a given real configuration $\mathcal{P}_{2k-2} \subset E_T$; 
\item the triplet $(\mathbb{R}T,\mathbb{R}E_T,\mathbb{R}\hat{H}_{k-1})$ is arranged as depicted in Fig. \ref{fig: non-symm_ellipsoid}, where $k_i=\tilde{k}_i$ and $k_j=\tilde{k}_j-1$ with $\{i,j\}=\{1,2\}$ and $\tilde{k}_{j} \not = 0$; 
\item the connected component $\mathcal{F}_1 \subset \mathcal{D}_{k-1}$ contains the $2$ real points of $\mathcal{P}_{2k}$.  
\end{itemize}
Finally, let $P_0(x,y)P_1(x,y)=0$ be a real polynomial equation defining the union of $E_T$ and $\hat{H}_{k-1}$ in some local affine chart of $T$. Pick $k$ real curves $L_{i}$ of bidegree $(1,1)$ such that $\bigcup_{i=1}^{k} L_{i}$ passes trhough $\mathcal{P}_{2k}$. Replace the left side of the equation $P_0(x,y)P_1(x,y)=0$ with $P_0(x,y)P_1(x,y)+ \varepsilon f_1(x,y)\dots f_k(x,y)$, where $f_i(x,y)=0$ is an equation for $L_i$ and $\varepsilon$ is a sufficient small real number. Up to a choice of the sign of $\varepsilon$, one constructs a small perturbation $C_{T}$ of $E_T \cup \hat{H}_{k-1}$ which is the wanted real curve of bidegree $(k,k)$.
\end{proof}

\begin{prop}
\label{prop: constr_quadric_DP2_3}
Let $T$ be the quadric hyperboloid and let $E_T$ be a non-singular real algebraic curve of bidegree $(1,1)$ in $T$. Then, for any real configuration of $2k$ distinct points in $E_T$ fixed as follows, there exists a non-singular real algebraic curve $C_T$ of bidegree $(k,k)$ on $T$, intersecting transversely $E_T$ in the $2k$ points and such that the triplet $(\mathbb{R}T, \mathbb{R}E_T, \mathbb{R}C_T)$ is arranged respectively as depicted:
\begin{enumerate}[label=(\arabic*)]
\item in Fig. \ref{fig: quadric_hyperb_DP2}.1,  \ref{fig: quadric_hyperb_DP2}.2 and  \ref{fig: quadric_hyperb_DP2}.3 for $k=3$ and $2$ fixed real points;
\item in Fig. \ref{fig: quadric_hyperb_DP2}.4 for $k=3$ and no fixed real points;
\item in Fig. \ref{fig: quadric_hyperb_DP2}.5 and  \ref{fig: quadric_hyperb_DP2}.6 for $k=2$ and $4$ fixed real points.
\end{enumerate}
\end{prop}
\begin{figure}[h!]
\begin{center}
\begin{picture}(100,107)
\put(-96,56){Fig. \ref{fig: quadric_hyperb_DP2}.1}
\put(30,56){Fig. \ref{fig: quadric_hyperb_DP2}.2}
\put(155,56){Fig. \ref{fig: quadric_hyperb_DP2}.3}
\put(-140,-10){\includegraphics[width=1.0\textwidth]{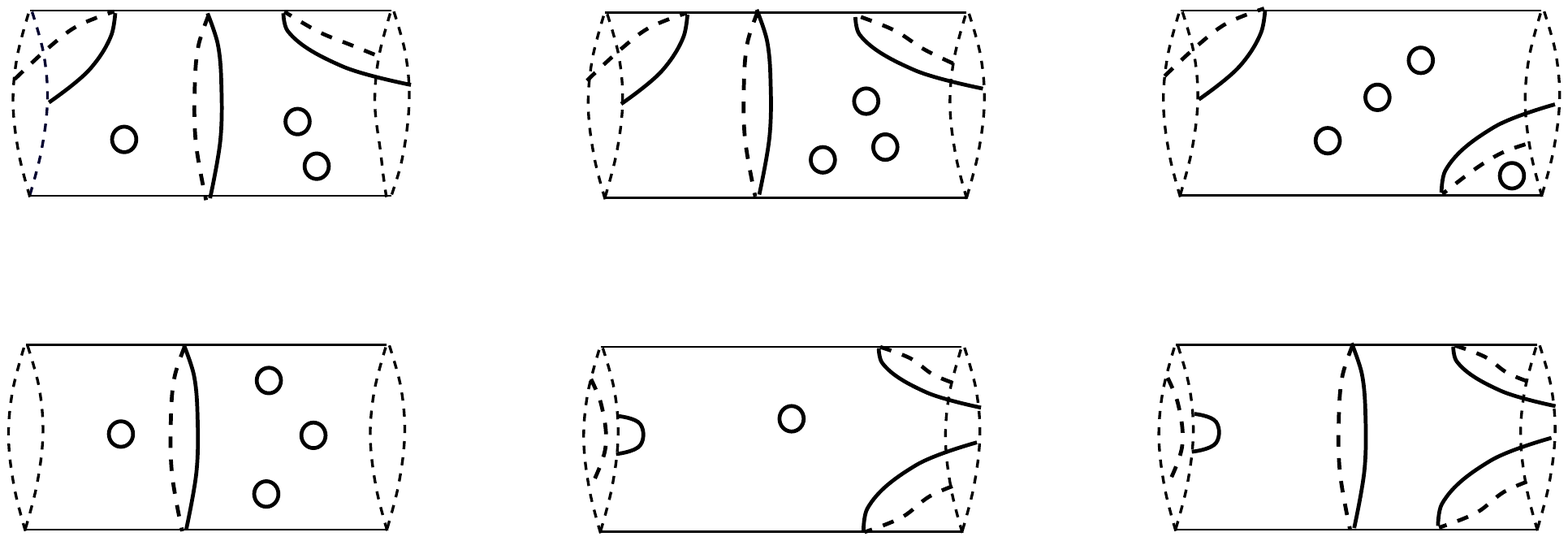}} 
\put(-96,-16){Fig. \ref{fig: quadric_hyperb_DP2}.4}
\put(30,-16){Fig. \ref{fig: quadric_hyperb_DP2}.5}
\put(158,-16){Fig. \ref{fig: quadric_hyperb_DP2}.6}
\end{picture}
\end{center}
\caption{$\mathbb{R}T \simeq S^1 \times S^1$ is depicted as a cylinder such that the $S^1$'s depicted in dashed on the sides of each cylinder are identified and represent $\mathbb{R}E_T$.}
\label{fig: quadric_hyperb_DP2}
\end{figure}
\begin{proof}
For any configuration $\mathcal{P}_4$ of $4$ fixed points on $\mathbb{R}E_T$ let us construct real curves $\tilde{H}$ of bidegree $(2,2)$ passing through $\mathcal{P}$ and such that the arrangements of $\mathbb{R}\tilde{H} \cup \mathbb{R}E_T$ are respectively as depicted in Fig. \ref{fig: quadric_hyperb_DP2}.5 and  \ref{fig: quadric_hyperb_DP2}.6. Let $P_0(x,y)P_1(x,y)=0$ be a real polynomial equation defining the union of $E_T$ and a bidegree $(1,1)$ real curve $H$ such that the points of $\mathcal{P}_4$ belong to one connected component $\mathcal{E}$ of $\mathbb{R}E_T\setminus \mathbb{R}H$. 
Let $H_1$ and $H_2$ be two bidegree $(1,1)$ real curves such that $H_1 \cup H_2$ contains $\mathcal{P}_4$. Replace the left side of the equation $P_0(x,y)P_1(x,y)=0$ with $P_0(x,y)P_1(x,y)+ \varepsilon f_1(x,y)f_2(x,y)$, where $f_i(x,y)=0$ is an equation for $H_i$ and $\varepsilon$ is a sufficient small real number. Up to a choice of the sign of $\varepsilon$, one constructs a small perturbation $\tilde{H}$ of $E_T \cup H$, where $\tilde{H}$ is a bidegree $(2,2)$ non-singular real curve such that $\bigcup_{i=1}^2H_i \cap E_T = \tilde{H} \cap E_T $ and the triplet $(\mathbb{R}T, \mathbb{R}E_T, \mathbb{R}\tilde{H})$ is arranged respectively as depicted in Fig. \ref{fig: quadric_hyperb_DP2}.5 and  \ref{fig: quadric_hyperb_DP2}.6. Analogously, via small perturbation method we can construct real algebraic curves of bidegree $(3,3)$ as described in $(2)-(3)$ and end the proof.  
\end{proof}
In order to accomplish some particular constructions in the proof of Proposition \ref{prop: S_pencil_quartic_DP2}, we need the following lemma.
\begin{lem}
\label{lem: curves_sigma1_DP2}
There exist real algebraic curves $\tilde{Q}$ and $C$ respectively of degree $4$ and $3$ in $\mathbb{C}P^2$ with a unique real non-degenerate double singularity at a point $q$, such that the triplets $(\mathbb{R}P^2,\mathbb{R}\tilde{Q},\mathbb{R}C)$ realize the real scheme depicted in Fig. \ref{fig: particular_constructions_sigma1_DP2} and Fig. \ref{fig: particular_constructions_sigma1_DP2_meno_sfere}.
\end{lem} 
\begin{figure}[h!]
\begin{center}
\begin{picture}(100,183)
\put(-140,0){\includegraphics[width=1.0\textwidth]{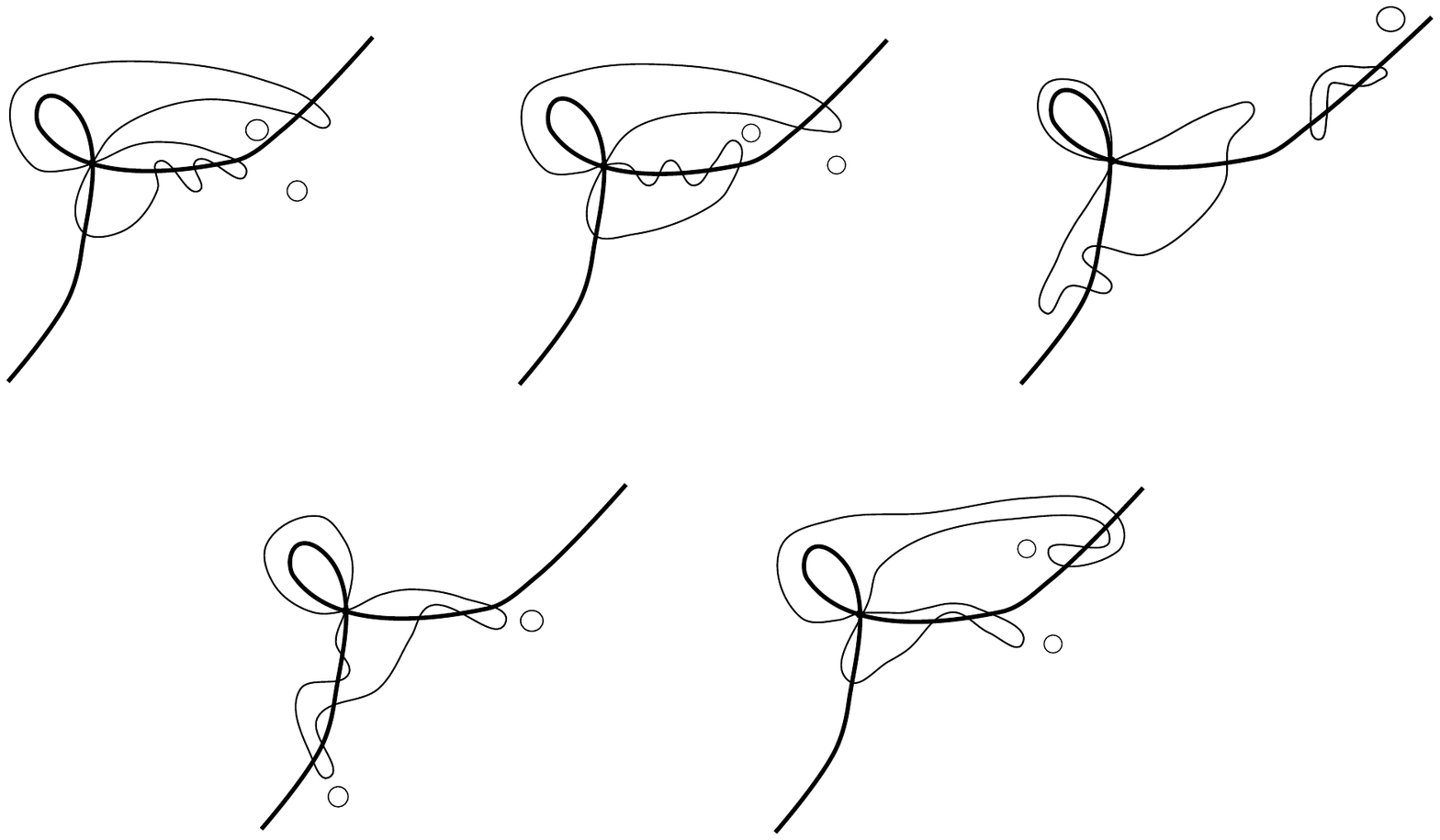}} 
\put(-91,97){Fig. \ref{fig: particular_constructions_sigma1_DP2}.1}
\put(29,97){Fig. \ref{fig: particular_constructions_sigma1_DP2}.2}
\put(146,97){Fig. \ref{fig: particular_constructions_sigma1_DP2}.3}

\put(-36,-12){Fig. \ref{fig: particular_constructions_sigma1_DP2}.4}
\put(90,-12){Fig. \ref{fig: particular_constructions_sigma1_DP2}.5}

\end{picture}
\end{center}
\caption{ }
\label{fig: particular_constructions_sigma1_DP2}
\end{figure}
\begin{figure}[h!]
\begin{center}
\begin{picture}(100,163)
\put(-140,0){\includegraphics[width=1.0\textwidth]{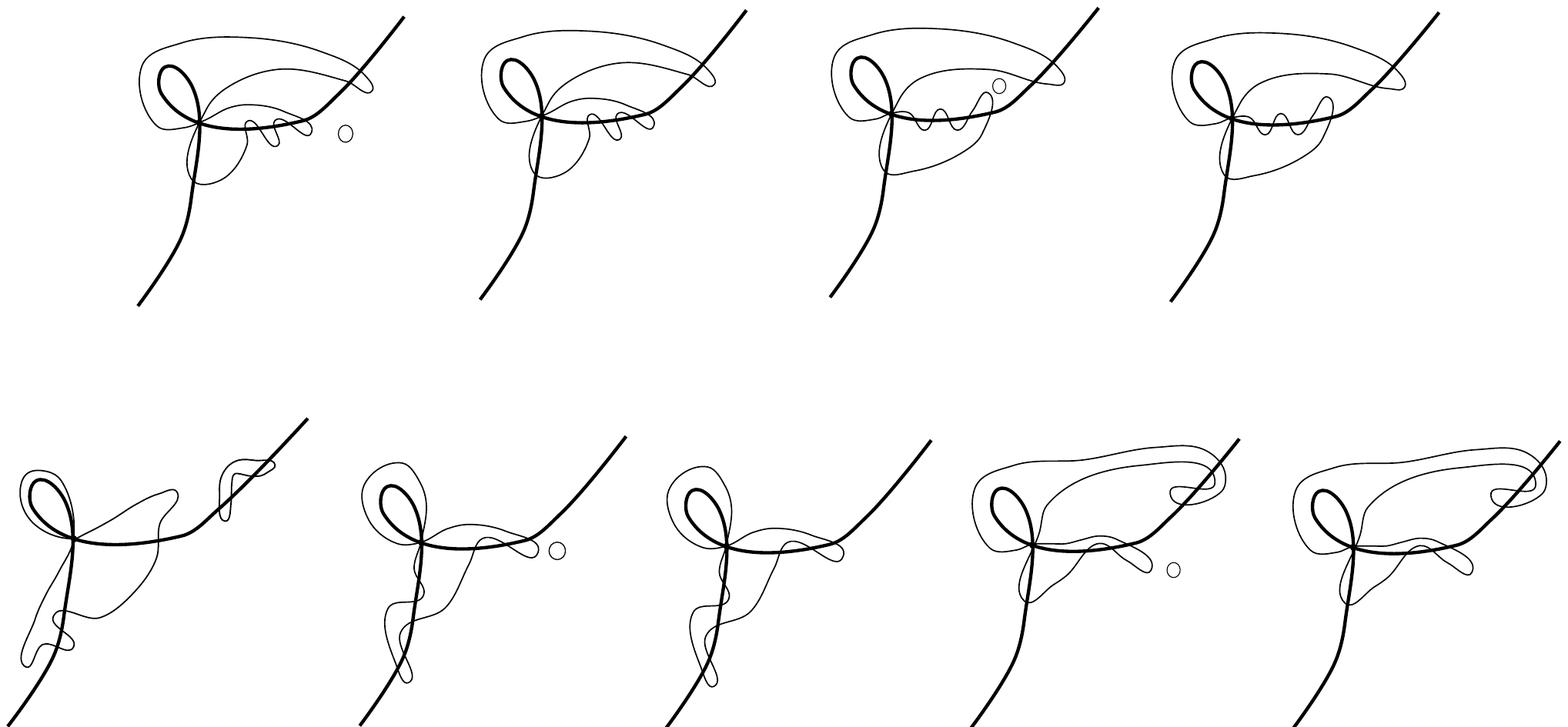}} 
\end{picture}
\end{center}
\caption{ }
\label{fig: particular_constructions_sigma1_DP2_meno_sfere}
\end{figure}
\begin{figure}[h!]
\begin{center}
\begin{picture}(100,133)
\put(-140,0){\includegraphics[width=1.0\textwidth]{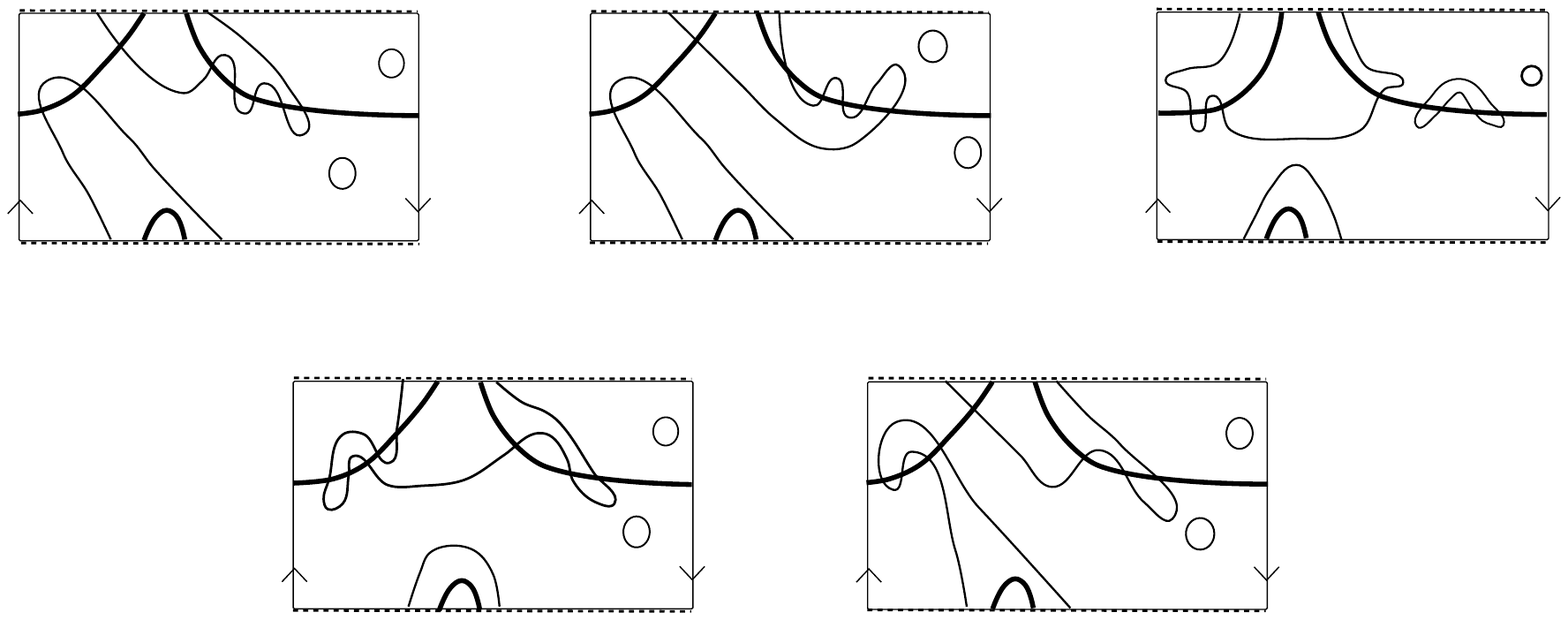}} 
\put(-81,62){Fig. \ref{fig: particular_constructions_sigma1_DP2_2}.1}
\put(37,62){Fig. \ref{fig: particular_constructions_sigma1_DP2_2}.2}
\put(150,62){Fig. \ref{fig: particular_constructions_sigma1_DP2_2}.3}

\put(-26,-12){Fig. \ref{fig: particular_constructions_sigma1_DP2_2}.4}
\put(91,-12){Fig. \ref{fig: particular_constructions_sigma1_DP2_2}.5}
\end{picture}
\end{center}
\caption{ }
\label{fig: particular_constructions_sigma1_DP2_2}
\end{figure}
\begin{proof}
Let us realize the real schemes in Fig. \ref{fig: particular_constructions_sigma1_DP2}. Then, thanks to an analogue construction method, one also realizes all real schemes in Fig. \ref{fig: particular_constructions_sigma1_DP2_meno_sfere}. In fact, remark that one gets the real schemes in Fig. \ref{fig: particular_constructions_sigma1_DP2_meno_sfere}, just deleting the empty ovals of the real quartic schemes in Fig. \ref{fig: particular_constructions_sigma1_DP2}.\\
The blow-up of $\mathbb{C}P^2$ at the point $q$ is the first Hirzebruch surface $\Sigma_1$ (Section \ref{subsec: cha2_Hirzebruch_surf}). Then, in order to prove the statement, it is sufficient to construct reducible real algebraic curves $K_i$, with $i=1,2,3,4,5$, of bidegree $(3,4)$ in $\Sigma_1$ as union of two non-singular real algebraic curves $\overline{Q}$ and $A$ respectively of bidegree $(2,2)$ and $(1,2)$ in $\Sigma_1$ such that the pairs $(\mathbb{R}\Sigma_1,\mathbb{R}\overline{Q} \cup\mathbb{R}A)$ realize the $\mathcal{L}$-schemes in Fig. \ref{fig: particular_constructions_sigma1_DP2_2}.\\
\begin{figure}[!h]
\begin{picture}(100,177)
\put(33,102){Fig. \ref{fig: intermediate_construction_dessin_enfant_DP2}.1}
\put(18,142){$p_1$}
\put(18,164){$p_3$}
\put(33,156){$p_4$}
\put(49,151){$p_2$}
\put(172,102){Fig. \ref{fig: intermediate_construction_dessin_enfant_DP2}.2}
\put(148,152){$p_1$}
\put(160,164){$p_3$}
\put(175,156){$p_4$}
\put(196,152){$p_2$}
\put(305,102){Fig. \ref{fig: intermediate_construction_dessin_enfant_DP2}.3}
\put(292,142){$p_1$}
\put(293,164){$p_3$}
\put(307,152){$p_4$}
\put(324,150){$p_2$}
\put(103,-2){Fig. \ref{fig: intermediate_construction_dessin_enfant_DP2}.4}
\put(103,44){$p_1$}
\put(112,56){$p_3$}
\put(122,67){$p_4$}
\put(133,54){$p_2$}
\put(245,-2){Fig. \ref{fig: intermediate_construction_dessin_enfant_DP2}.5}
\put(227,55){$p_1$}
\put(237,66){$p_3$}
\put(257,67){$p_4$}
\put(270,42){$p_2$}
\put(-15,7){\includegraphics[width=1.1\textwidth]{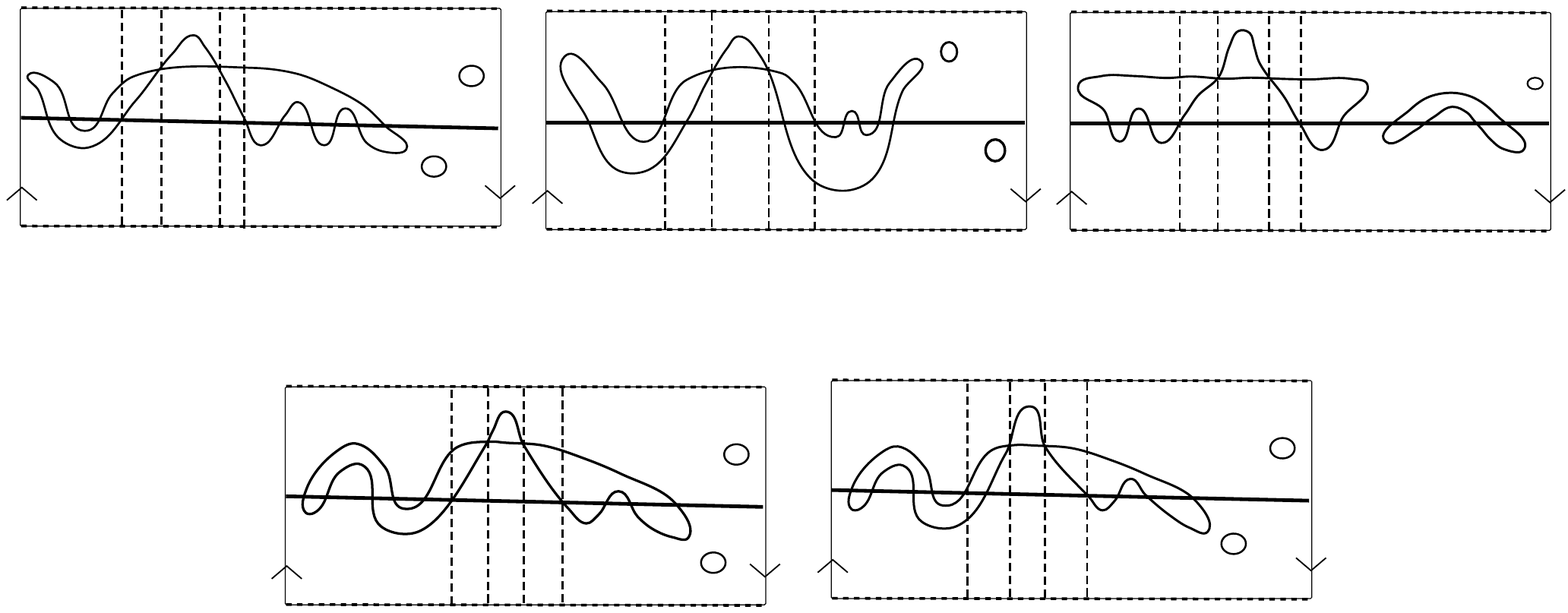}} 
\end{picture}
\caption{Intermediate constructions.}
\label{fig: intermediate_construction_dessin_enfant_DP2}
\end{figure}
\begin{figure}[!h]
\begin{picture}(100,90)
\put(10,-2){Fig. \ref{fig: intermediate_construction_dessin_enfant_DP2_2}.1}
\put(120,-2){Fig. \ref{fig: intermediate_construction_dessin_enfant_DP2_2}.2}
\put(225,-2){Fig. \ref{fig: intermediate_construction_dessin_enfant_DP2_2}.3}
\put(327,-2){Fig. \ref{fig: intermediate_construction_dessin_enfant_DP2_2}.4}
\put(-15,7){\includegraphics[width=1.1\textwidth]{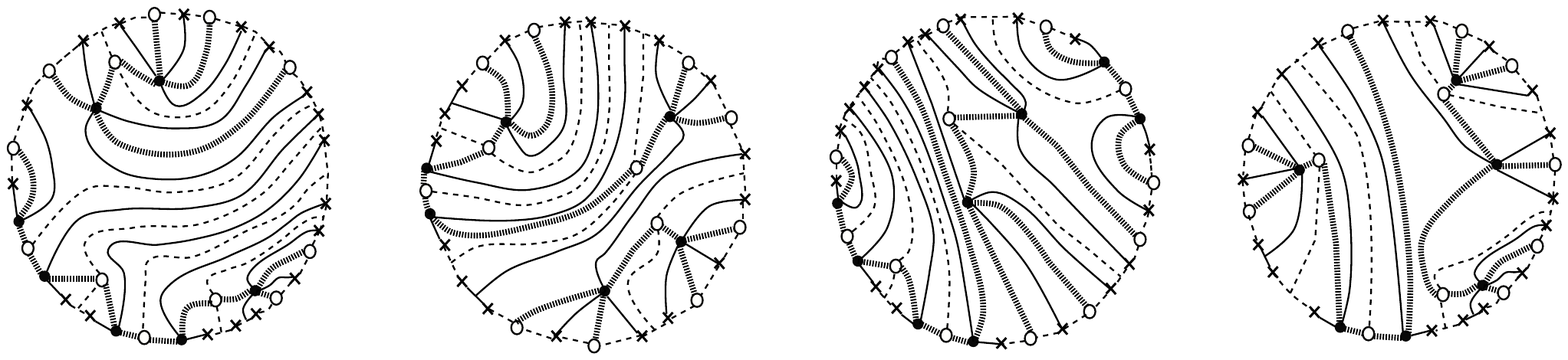}} 
\end{picture}
\caption{Intermediate constructions.}
\label{fig: intermediate_construction_dessin_enfant_DP2_2}
\end{figure}
Let us denote with $\tilde{\eta}_i$ 
the trigonal $\mathcal{L}$-schemes in $\mathbb{R}\Sigma_5$ respectively depicted in Fig. \ref{fig: intermediate_construction_dessin_enfant_DP2}, for $i=1,2,3,4,5$. Due to Theorem \ref{thm: existence_trigonal}, if the real graph associated to each $\tilde{\eta}_i$ is completable in degree $5$ to a real trigonal graph, then there exists a real algebraic trigonal curve $\tilde{K}_i$ realizing $\tilde{\eta}_i$, for all $i \in \{1,2,3,4,5\}$. Therefore, the completion $\Gamma_i$ of the real graph, associated to each $\tilde{\eta}_i$, respectively depicted in Fig. \ref{fig: intermediate_construction_dessin_enfant_DP2_2}.1 - \ref{fig: intermediate_construction_dessin_enfant_DP2_2}.3 and Fig. \ref{fig: intermediate_construction_dessin_enfant_DP2_2}.4 for $i=1$, $2$, $3$ and $i=4,5$, proves the existence of such $\tilde{K}_i$'s. \\
Each $\tilde{K}_i$ is reducible because it has $12$ non-degenerate double points and its normalization has $4$ real connected components. 
In particular, every $\tilde{K}_i$ has to be the union of a real curve of bidegree $(2,0)$ and a real curve of bidegree $(1,0)$.\\
Let us consider the birational transformation $$\Xi:= \beta^{-1}_{p_{1}}\beta^{-1}_{p_{2}}\beta^{-1}_{p_{3}}\beta^{-1}_{p_{4}}: (\Sigma_5,\tilde{K}_i)\dashrightarrow (\Sigma_1, K_i),$$ 
defined as in Section \ref{subsec: cha2_Hirzebruch_surf}; where the points $p_j$'s, with $j=1,2,3,4$, are the real double points of $\tilde{K}_i$ respectively as depicted in Fig. \ref{fig: intermediate_construction_dessin_enfant_DP2}, and the dashed real fibers are those intersecting the $p_j$'s. The image via $\Xi$ of the reducible real trigonal curve $\tilde{K}_i$ is a reducible curve $K_i$ of bidegree $(3,4)$ which is the union of two non-singular real curves $\overline{Q}$ and $A$, respectively of bidegree $(2,2)$ and $(1,2)$ in $\Sigma_1$. Moreover, the $\mathcal{L}$-scheme of each $K_i$ is as respectively depicted in Fig. \ref{fig: particular_constructions_sigma1_DP2_2}.
\end{proof}
We end this section giving the following intermediate constructions.
\begin{prop}
\label{prop: S_pencil_quartic_DP2}
\begin{enumerate}[label=(\roman*)]
\item[]
\item For every arrangement of ovals in $\bigsqcup_{i=1}^sS^2$ respectively depicted in Fig. \ref{fig: S_particular3_erwan_DP2}, Fig. \ref{fig: S_particular1_DP2} and Fig. \ref{fig: S_particular2_DP2}, for every $s \leq j \leq 3$, there exists a $j$-sphere real $1$-nodal del Pezzo pair $(S, E_S)$, and real algebraic 
curve $C_S \subset S$ of bi-class $(d,k)$ such that $\mathbb{R}C_S \cup \mathbb{R}E_S$ is arranged in $\mathbb{R}S$ as depicted 
\begin{enumerate}[label=(\arabic*)]
\item in Fig. \ref{fig: S_particular3_erwan_DP2}, for $d=3$ and $k=2$;
\item in Fig. \ref{fig: S_particular1_DP2} and in Fig. \ref{fig: S_particular2_DP2}, for $d=3$ and $k=3$.
\end{enumerate}
\item Let $d,$ $k_1,k_2$ and $h_1,h_2,h_3,h_4$  be non-negative integers such that
\begin{itemize}
\item $d\geq 5$;
\item $k_1 + k_2=d-4$;
\item $\sum\limits_{i=1}^{4}h_i=d-1$;
\item $h_3 \equiv 1 \mod  2$, if $h_{3} \not = 0$;
\item $h_4 \equiv 1 \mod  2$, if $h_{4} \not = 0$.
\end{itemize}
Then, there exist a $3$-sphere real $1$-nodal del Pezzo pair $(S, E_S)$ and a bi-class $(d,d)$ real curve $C_S \subset S$ such that $\mathbb{R}C_S \cup \mathbb{R}E_S$ is arranged in $\mathbb{R}S$ as depicted in Fig. \ref{fig: non-symm_3spheres_DP2}.
\end{enumerate}
\end{prop}
\begin{figure}[h!]
\begin{center}
\begin{picture}(100,92)
\put(-105,5){\includegraphics[width=0.8\textwidth]{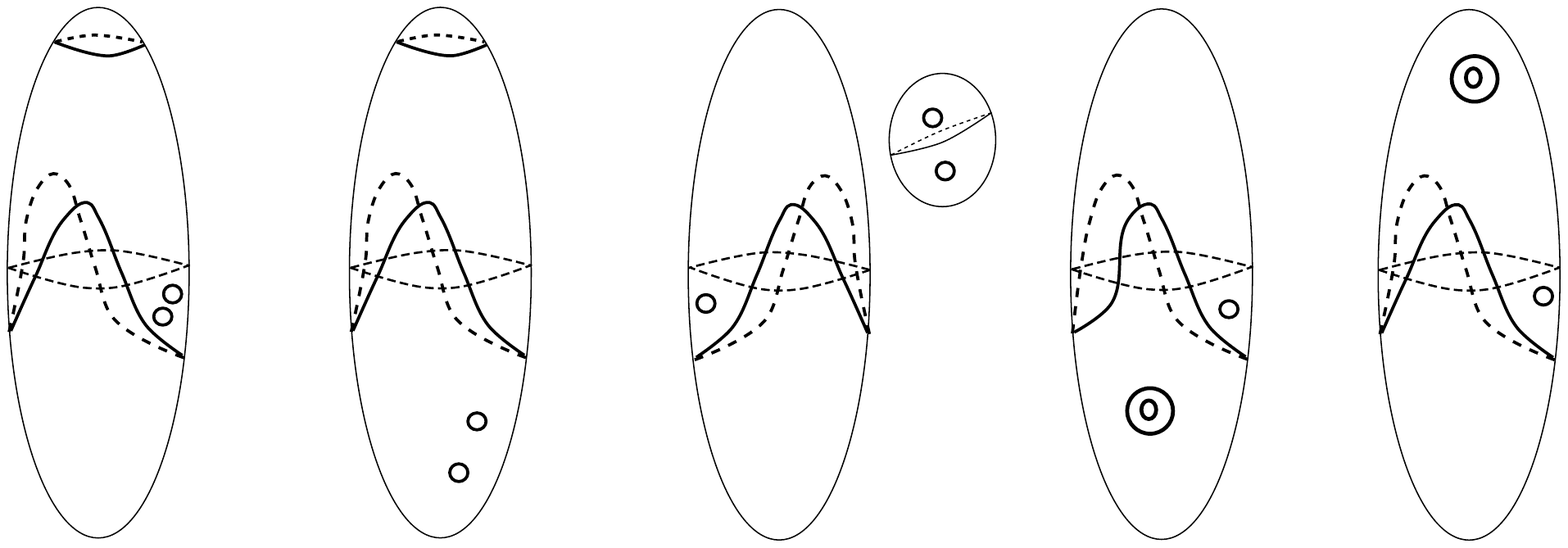}} 
\put(-102,-7){Fig. \ref{fig: S_particular3_erwan_DP2}.1:} 
\put(-92,-17){$s=1$}

\put(-40,-7){Fig. \ref{fig: S_particular3_erwan_DP2}.2: }
\put(-30,-17){$s=1$}

\put(24,-7){Fig. \ref{fig: S_particular3_erwan_DP2}.3: }
\put(34,-17){$s=2$}

\put(85,-7){Fig. \ref{fig: S_particular3_erwan_DP2}.4: }
\put(95,-17){$s=1$}

\put(142,-7){Fig. \ref{fig: S_particular3_erwan_DP2}.5:} 
\put(152,-17){$s=1$}
\end{picture}
\end{center}
\caption{$\mathbb{R}E_S\simeq S^1$ in dashed.}
\label{fig: S_particular3_erwan_DP2}
\end{figure}
\begin{figure}[h!]
\begin{center}
\begin{picture}(100,180)
\put(-105,-25){\includegraphics[width=0.8\textwidth]{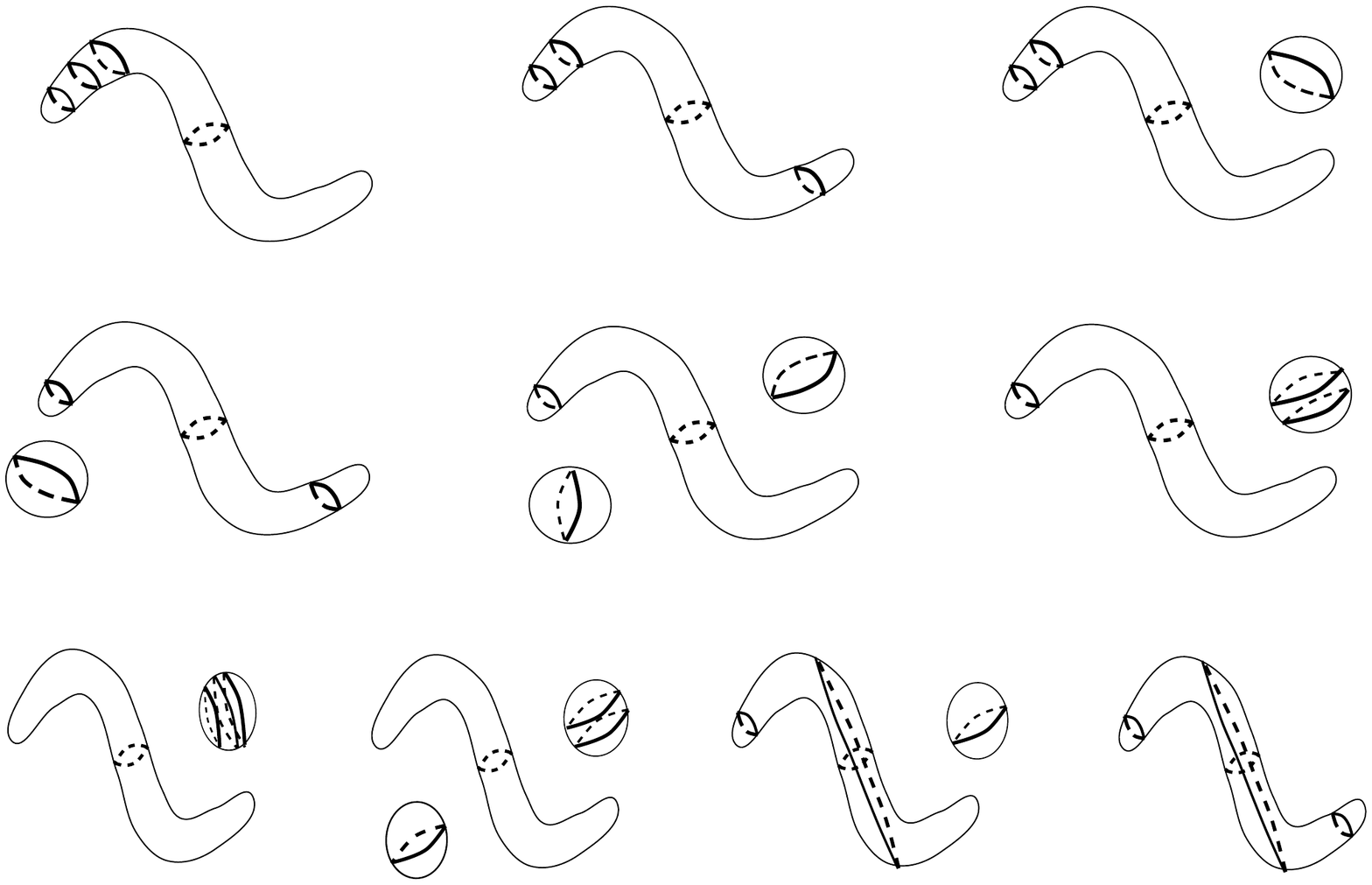}} 
\end{picture}
\end{center}
\caption{$\mathbb{R}E_S\simeq S^1$ in dashed.}
\label{fig: S_particular1_DP2}
\end{figure}
\begin{figure}[h!]
\begin{center}
\begin{picture}(100,70)
\put(-145,-10){\includegraphics[width=1.0\textwidth]{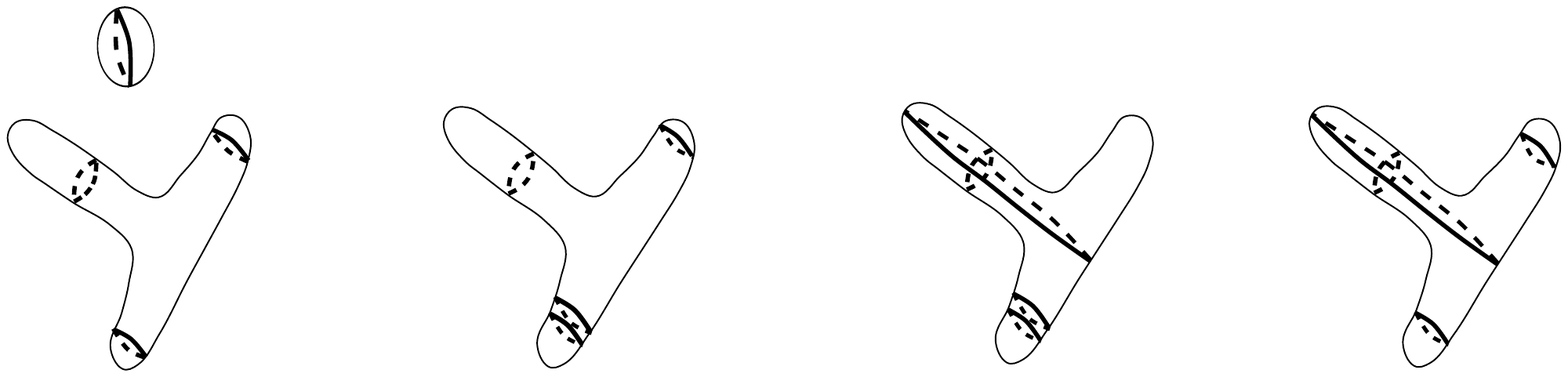}} 
\end{picture}
\end{center}
\caption{$\mathbb{R}E_S\simeq S^1$ in dashed.}
\label{fig: S_particular2_DP2}
\end{figure}
\begin{figure}[h!]
\begin{center}
\begin{picture}(100,160)
\put(-80,5){\includegraphics[width=0.7\textwidth]{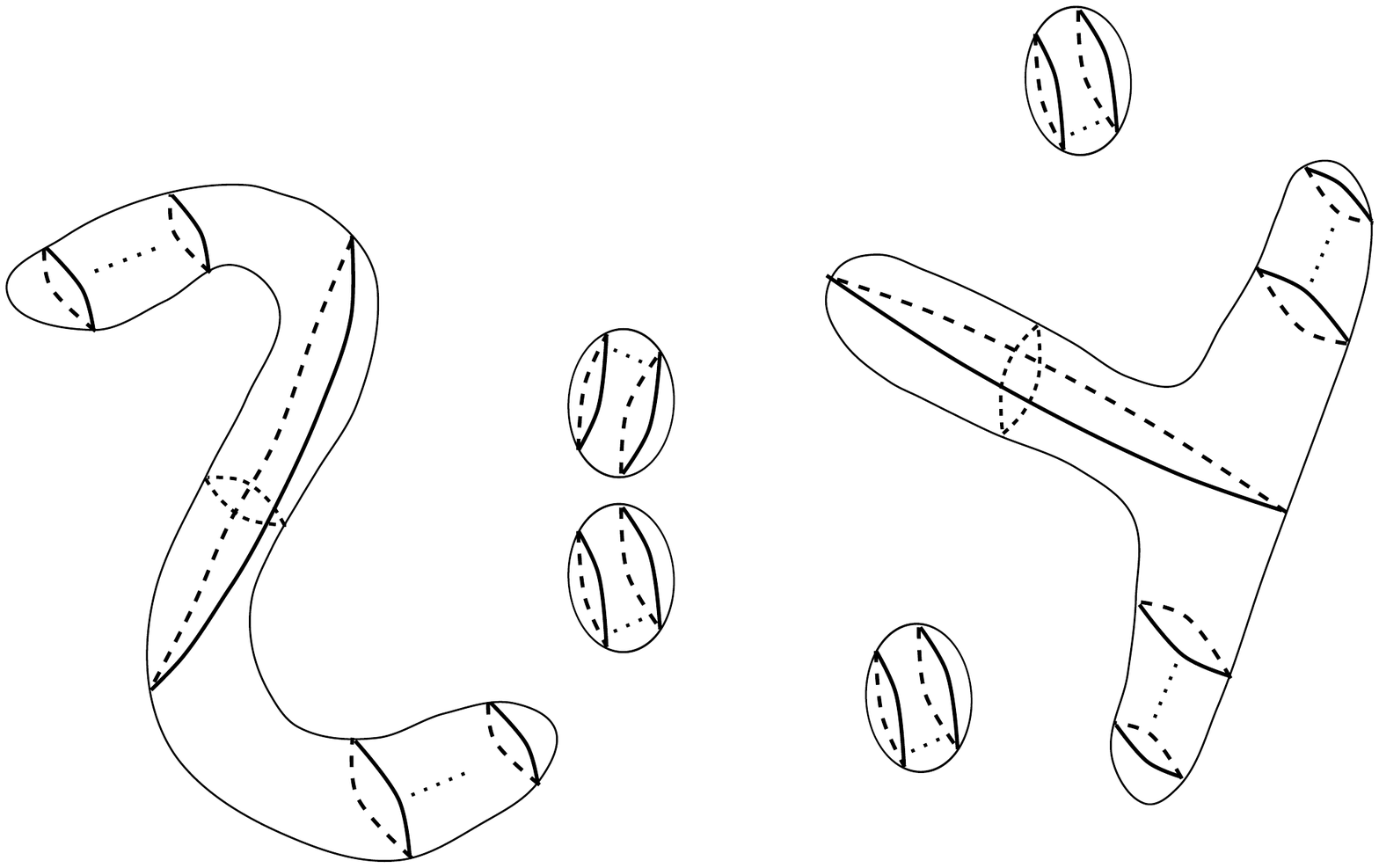}} 
\put(-35,-15){Fig. \ref{fig: non-symm_3spheres_DP2}.1}

\put(-76,132){\rotatebox{120}{$\textcolor{black}{\Big \}}$}}
\put(-71,142){$\textcolor{black}{h_1}$}
\put(4,4){\rotatebox{-61}{$\textcolor{black}{\Big \} h_2}$}}
\put(30,41){\rotatebox{-90}{$\textcolor{black}{\Big \}}$}}
\put(35,24){$\textcolor{black}{h_3}$}
\put(29,104){\rotatebox{90}{$\textcolor{black}{\Big \}}$}}
\put(34,114){$\textcolor{black}{h_4}$}

\put(108,-15){Fig. \ref{fig: non-symm_3spheres_DP2}.2}
\put(120,133){\rotatebox{-61}{$\textcolor{black}{\Big \} }$}}
\put(128,120){$\textcolor{black}{ h_3}$}
\put(92,16){\rotatebox{-61}{$\textcolor{black}{\Big \}}$}}
\put(100,3){$\textcolor{black}{ h_4}$}
\put(150,22){\rotatebox{-31}{$\textcolor{black}{\Big \} h_2}$}}
\put(177,107){\rotatebox{-20}{$\textcolor{black}{\Big \} h_1}$}}

\end{picture}
\end{center}
\caption{$\mathbb{R}E_S\simeq S^1$ in dashed. The pair $(\mathbb{R}S, \mathbb{R}C_S)$ realizes $ \langle   N(h_1,0)   \rangle   \sqcup   N(h_2,0) : N(h_3,0) : N(h_4,0)$.}
\label{fig: non-symm_3spheres_DP2}
\end{figure}
\begin{proof}
Let $\tilde{Q}$ be a real quartic with a real non-degenerate and non-isolated double point at $q$ as only singularity and 
let $C$ be a real curve of degree $d$ with one $k$-fold singularity at $q$. To a pair $(\tilde{Q},C)$ correspond a  
real $1$-nodal del Pezzo pair $(S,E_S)$ and a real algebraic curve $C_S\subset S$ of bi-class $(d,k)$ 
with topology described by the topological type realized by the triplet $(\mathbb{R}P^2, \mathbb{R}\tilde{Q}, \mathbb{R}C)$.\\
Proof of $(i)$: First of all, Lemma \ref{lem: curves_sigma1_DP2} immediately implies the existence of $j$-sphere real $1$-nodal degree $2$ del Pezzo pairs $(S,E_S)$ and real curves $C_S \subset S$ of bi-class $(3,2)$ 
such that the triplet $(\mathbb{R}S, \mathbb{R}E_S, \mathbb{R}C_S)$ is arranged as depicted in Fig. \ref{fig: S_particular3_erwan_DP2}, where we depict only the non-empty spheres of $\mathbb{R}S$.\\
\begin{figure}[h!]
\begin{center}
\begin{picture}(100,90)
\put(-135,5){\includegraphics[width=1.0\textwidth]{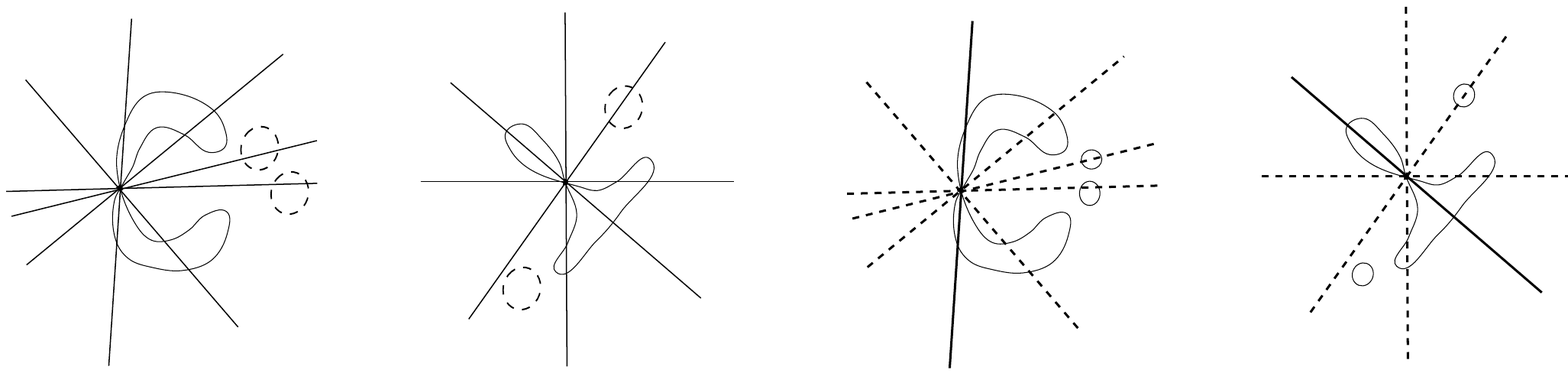}}
\put(-103,-12){Fig. \ref{fig: quartic_nodal_pencil_DP2_bis}.1: $(\mathbb{R}P^{2},\mathbb{R}L_{q_1}, \mathbb{R}Q)$}
\put(132,-12){Fig. \ref{fig: quartic_nodal_pencil_DP2_bis}.2}
\put(80,77){$\mathbb{R}L$}
\put(165,77){$\mathbb{R}L$}
\end{picture}
\end{center}
\caption{$1$: The set of points of $\mathbb{R}Q$ homeomorphic to $\bigsqcup_{i=1}^{j-1} S^1$ is depicted in dashed. $2$: The lines in dashed represent the $d-1$ lines of the pencil $L_{q_1}$.}
\label{fig: quartic_nodal_pencil_DP2_bis}
\end{figure}
Let us realize the arrangements in $(2)$. It is easy to see that there exist a real plane quartic $\tilde{Q}_1$ with a real non-degenerate double point $q_1$ as only singularity and real part homeomorphic to $\bigsqcup_{i=1}^{j-1} S^1 \sqcup \bigvee_{j=1}^2 S^1$ and a pencil of lines $L_{q_1} \subset \mathbb{C}P^2$, centered at $q_1$, such that $\mathbb{R}\tilde{Q}_1\cup \mathbb{R} L_{q_1}$ is arranged respectively as depicted in Fig. \ref{fig: quartic_nodal_pencil_DP2_bis}.1. The union of any three distinct lines of $L_{q_1}$ is a cubic $C_1$ with a triple point at $q_1$. From $\tilde{Q}_1$ and $C_1$ one can construct $j$-sphere real $1$-nodal degree $2$ del Pezzo pairs $(S,E_S)$ and curves $C_S \subset S$ of bi-class $(3,3)$ 
such that the triplet $(\mathbb{R}S, \mathbb{R}E_S, \mathbb{R}C_S)$ is arranged as depicted in Fig. \ref{fig: S_particular1_DP2} (resp. Fig. \ref{fig: S_particular2_DP2}) where we depict only the non-empty spheres of $\mathbb{R}S$.\\
Proof of $(ii)$: Assume that $\tilde{Q}_1$ has real part homeomorphic to $\bigsqcup_{i=1}^2 S^{1} \sqcup \bigvee_{j=1}^2 S^1$. The union of a line $L\subset L_{q_1}$ (in thick black) and other $d-1$ distinct lines (in dashed) of $L_{q_1}$ respectively as depicted in Fig. \ref{fig: quartic_nodal_pencil_DP2_bis}.2 is a degree $d$ real curve $C_d$ with a $d$-fold singularity at $q_1$. From $\tilde{Q}_1$ and $C_d$ one can construct $3$-sphere real $1$-nodal degree $2$ del Pezzo pairs $(S,E_S)$ and real curves $C_S \subset S$ of bi-class $(d,d)$ 
such that the triplet $(\mathbb{R}S, \mathbb{R}E_S, \mathbb{R}C_S)$ is arranged respectively as depicted in Fig. \ref{fig: non-symm_3spheres_DP2}.1 and \ref{fig: non-symm_3spheres_DP2}.2.
\end{proof}
\subsection{Final constructions}
\label{subsec: final_constr_DP2} 
We end the proof of Theorem \ref{thm: principal_DP2} and Proposition \ref{prop: principal_DP2}. Moreover, we prove Proposition \ref{prop: non-symm_real_scheme}. The proofs combine the results and constructions of Theorem \ref{thm: weak_patch_DP2} and Propositions \ref{prop: constr_quadric_DP2}, \ref{prop: constr_quadric_DP2_2}, \ref{prop: constr_quadric_DP2_3}, \ref{prop: S_pencil_quartic_DP2}.
\begin{prop}
\label{prop: star_label_DP2}
Every real scheme $\mathcal{S}$ in $\mathcal{S}_{DP2}(4,3)$ labeled with $\dagger$ in Table \ref{tabella=realized3}, is realizable in $X^{4}$ and in class $3$. Moreover, every $\mathcal{S}$ labeled with $\dagger^{*}$ is realizable in $X^{k}$ and in class $3$, with $1 \leq k \leq 3$.
\end{prop}
\begin{proof}
The realization of real schemes in class $3$ is done as follows.\\
\textbf{General construction}: Pick any $j$-sphere real $1$-nodal degree $2$ del Pezzo pair $(S,E_S)$ and any real algebraic curve $C_S \subset S$ of bi-class $(3,h)$  
constructed as in proof of Proposition \ref{prop: S_pencil_quartic_DP2}, with $h=2,3$. 
Due to Corollary \ref{cor: esistenza_family_DP2}, there exists a real algebraic surface $X_0'$ as union of $S$ and $T$, intersecting along a curve $E$, and there exists a real algebraic curve $C_0$ as union of $C_S$ and $C_T$, intersecting along $2h$ points of $E$; where $T$ is a quadric ellipsoid, respectively a quadric hyperboloid 
and $C_T \subset T$ is a real algebraic curve of bidegree $(h,h)$ constructed  as in proof of Proposition \ref{prop: constr_quadric_DP2}, respectively Proposition \ref{prop: constr_quadric_DP2_3}. 
Then, thanks to Theorem \ref{thm: weak_patch_DP2},  one realizes a real scheme in $X^{j+1}$, respectively in  $X^{j}$ and in class $3$.
\par Applying the above general construction in $4$ different ways, one realizes in $X^{k}$ and in class $3$ all the real schemes listed below. Let us divides such real schemes in $4$ groups:
\begin{enumerate}[label=(\arabic*)]
\item for $1 \leq k \leq 4$,
\begin{equation*}
\begin{aligned}[left]
&1   \sqcup   \langle 1 \rangle   \sqcup   \langle 1 \rangle :3:0:0,& \\
&2     \sqcup   \langle 2 \rangle :3:0:0, &
\end{aligned}
\begin{aligned}[left]
& 1   \sqcup   \langle 3 \rangle: \langle \langle 1 \rangle \rangle:0:0,&\\
&  1   \sqcup   \langle 2 \rangle:1   \sqcup   \langle 2 \rangle:0:0,&
\end{aligned}
\begin{aligned}[left]
&5:3:0:0,&  1     \sqcup   \langle 2 \rangle :2:2:0,\\
&  4:4:0:0
\end{aligned}
\end{equation*}
\item for $ 1 \leq k \leq 3$
\begin{equation*}
\begin{aligned}[left]
&1    \sqcup   \langle1\rangle    \sqcup  \langle2   \sqcup  \langle1\rangle\rangle  :0:0:0, &   \\
&1    \sqcup   \langle1\rangle    \sqcup    \langle1\rangle    \sqcup    \langle2\rangle:0:0:0, &  \\
&4    \sqcup   \langle1\rangle    \sqcup    \langle1\rangle  :0:0:0,&\\
\end{aligned}
\begin{aligned}[left]
&1    \sqcup   \langle2\rangle    \sqcup    \langle3\rangle  :0:0:0,&\\
&3    \sqcup    \langle4\rangle  :0:0:0,&\\
&8:0:0:0&
\end{aligned}
\end{equation*}

\item for $ 1 \leq k \leq 3$

\begin{equation*}
\begin{aligned}[left]
&2 \sqcup \langle 5 \rangle:0:0,&\\
&3 \sqcup  \langle  1 \rangle \sqcup  \langle 2 \rangle:0:0,&\\
&1  \sqcup  \langle 3 \rangle  \sqcup  \langle  \langle 1 \rangle \rangle:0:0,&\\
\end{aligned}
\begin{aligned}[left]
&1 \sqcup  \langle  1 \rangle \sqcup  \langle 4 \rangle:0:0, &\\
&1 \sqcup  \langle  \langle 4 \rangle \rangle:0:0, &\\
&2  \sqcup  \langle 2 \rangle  \sqcup  \langle  \langle 1 \rangle \rangle:0 :0, &
\end{aligned}
\begin{aligned}[left]
&2 \sqcup  \langle  1 \rangle \sqcup  \langle 3 \rangle:0:0,&\\
&1 \sqcup   \langle 1 \rangle \sqcup   \langle  \langle 3 \rangle \rangle:0:0,&\\
&3  \sqcup  \langle 1 \rangle  \sqcup  \langle  \langle 1 \rangle \rangle:0:0&
\end{aligned}
\end{equation*}
\item for $ 1 \leq k \leq 4$, all the remaining real schemes labeled with $\dagger$ and/or $\dagger^{*}$ in Table \ref{tabella=realized3}.
\end{enumerate}
Now let us apply the general construction to each case as follows.
\begin{enumerate}[label=(\arabic*)]
\item Take $j+1=k$ and $h=2$. Let $T$ be a quadric ellipsoid and $C_T \subset T$ a real curve of bidegree $(2,2)$ constructed  as in proof of Proposition \ref{prop: constr_quadric_DP2}.
\item Take $j=k$ and $h=2$. Let $T$ be a quadric hyperboloid and $C_T \subset T$  a real curve of bidegree $(2,2)$ constructed  as in proof of Proposition \ref{prop: constr_quadric_DP2_3}. Moreover, take $C_{S}\subset S$ such that the triplet $(\mathbb{R}S, \mathbb{R}E_S, \mathbb{R}C_S)$ is as depicted respectively in Fig. \ref{fig: S_particular3_erwan_DP2}.1, \ref{fig: S_particular3_erwan_DP2}.2, \ref{fig: S_particular3_erwan_DP2}.4 and \ref{fig: S_particular3_erwan_DP2}.5, where we depict only the non-empty spheres of $\mathbb{R}S$.
\item Take $j=k$, $h=3$: Let $T$ be a quadric hyperboloid and $C_T \subset T$ constructed  as in proof of Proposition \ref{prop: constr_quadric_DP2_3}. Moreover, take $C_{S}\subset S$ such that the triplet $(\mathbb{R}S, \mathbb{R}E_S, \mathbb{R}C_S)$ is as depicted in Fig. \ref{fig: S_particular_repetition}.
\item Take $j+1=k$, $h=3$. Let $T$ be a quadric ellipsoid and $C_T \subset T$ constructed as in proof of Proposition \ref{prop: constr_quadric_DP2}.
\end{enumerate}
\begin{figure}[h!]
\begin{center}
\begin{picture}(100,50)
\put(-132,-10){\includegraphics[width=1.0\textwidth]{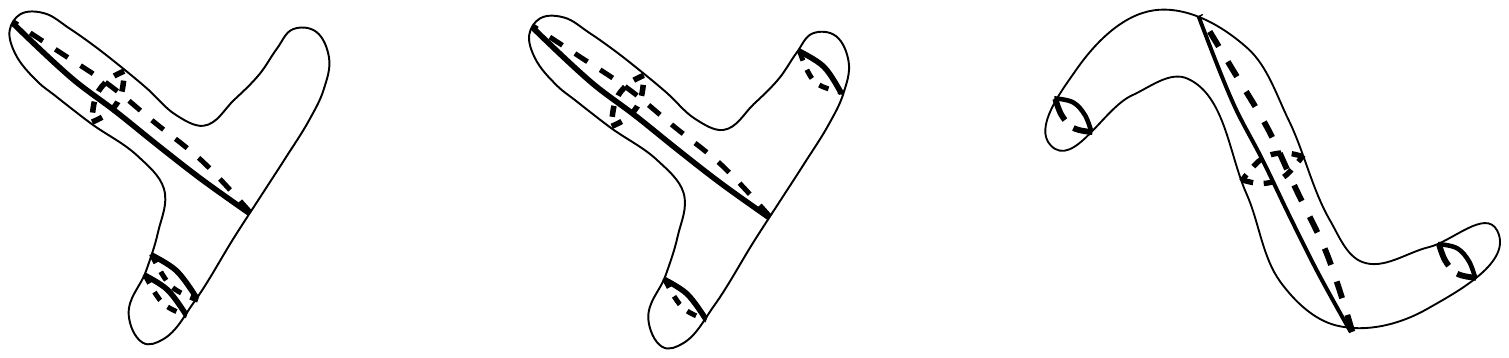}} 
\end{picture}
\end{center}
\caption{$\mathbb{R}E_S\simeq S^1$ in dashed.}
\label{fig: S_particular_repetition}
\end{figure}
\end{proof}
\begin{exa}
\label{exa: <1><2>:3_DP2}
We follow the steps of the proof of Proposition \ref{prop: star_label_DP2} to realize:
\begin{enumerate}[label=(\arabic*)]
\item the real scheme $\langle1\rangle    \sqcup   \langle2\rangle:3:0:0$ in $X^{4}$ and in class $3$.  
\begin{itemize}
\item Let $C_S\subset S$ be the real algebraic curve such that $\mathbb{R}E_S \cup \mathbb{R}C_S$ is arranged in $\mathbb{R}S$ as pictured in Fig. \ref{fig: <1><2>:3_ex_DP2}.1. 
\item Let $T$ be the quadric ellipsoid and let $C_T\subset T$ be the real algebraic curve of bidegree $(3,3)$ such that $\mathbb{R}E_T \cup \mathbb{R}C_T$ is arranged in $\mathbb{R}T$ as depicted in Fig. \ref{fig: <1><2>:3_ex_DP2}.2.
\begin{figure}[h!]
\begin{center}
\begin{picture}(100,110)
\put(-105,0){\includegraphics[width=0.9\textwidth]{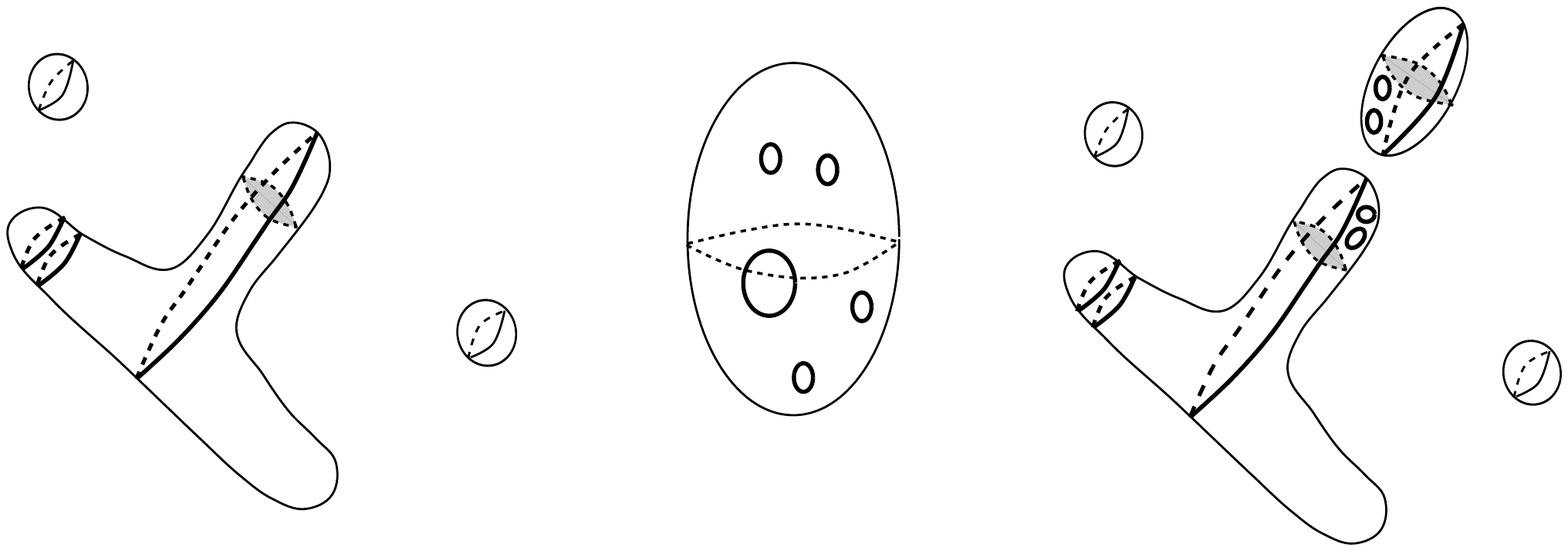}} 
\put(-74,-12){Fig. \ref{fig: <1><2>:3_ex_DP2}.1}
\put(52,-12){Fig. \ref{fig: <1><2>:3_ex_DP2}.2}
\put(140,-12){Fig. \ref{fig: <1><2>:3_ex_DP2}.3}
\end{picture}
\end{center}
\caption{$\mathbb{R}E\simeq S^1$ in thick dashed.}
\label{fig: <1><2>:3_ex_DP2}
\end{figure}
\item Thanks to Theorem \ref{thm: weak_patch_DP2} $\langle1\rangle    \sqcup   \langle2\rangle:3:0:0$ is realizable in $X^{4}$ and in class $3$ (Fig. \ref{fig: <1><2>:3_ex_DP2}.3);
\end{itemize}
\item the real scheme  $ b   \sqcup   \langle a+1 \rangle    \sqcup   \langle \langle 1 \rangle \rangle :0:0$ in $X^{k}$ and in class $3$, where $a,b$ denotes number of ovals and $a+b=3$ and $k=3,2,1$.
\begin{itemize}
\item Let $C_S\subset S$ be the real algebraic curve such that $\mathbb{R}E_S \cup \mathbb{R}C_S$ is arranged in $\mathbb{R}S$ as pictured in Fig. \ref{fig: meno_sfere_a_b:3_ex_DP2}.1, where we depict only the non-empty spheres of $\mathbb{R}S$. 
\item Let $T$ be the quadric hyperboloid and $C_T\subset T$ be the real algebraic curve of bidegree $(3,3)$ such that $\mathbb{R}E_T \cup \mathbb{R}C_T$ is arranged in $\mathbb{R}T$ as depicted in Fig. \ref{fig: meno_sfere_a_b:3_ex_DP2}.2.
\item Thanks to Theorem \ref{thm: weak_patch_DP2}, for any values of $a,b$, the real scheme $ b   \sqcup   \langle a+1 \rangle    \sqcup   \langle \langle 1 \rangle \rangle :0:0:0$ is realizable in $X^{k}$ and in class $3$. See Fig. \ref{fig: meno_sfere_a_b:3_ex_DP2}.3, where we depict only the non-empty spheres of $\mathbb{R}X$.
\end{itemize}
\end{enumerate}
\begin{figure}[h!]
\begin{center}
\begin{picture}(100,100)
\put(-105,0){\includegraphics[width=0.9\textwidth]{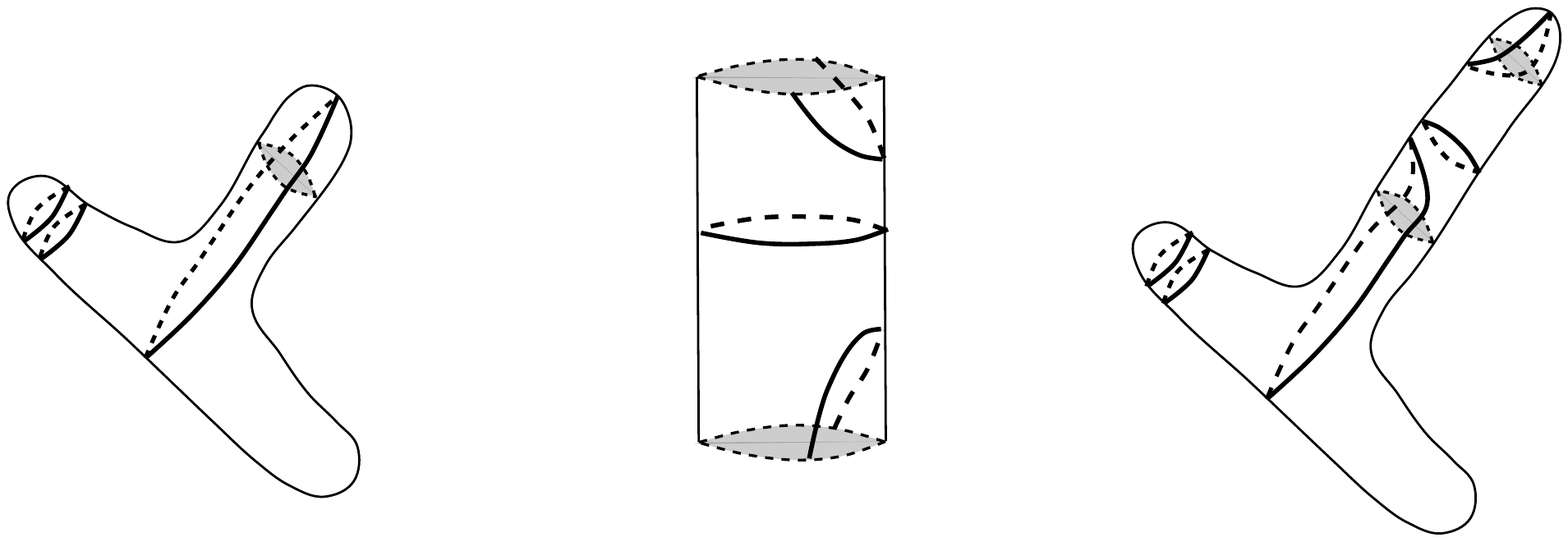}} 
\put(-80,-12){Fig. \ref{fig: meno_sfere_a_b:3_ex_DP2}.1}
\put(37,-12){Fig. \ref{fig: meno_sfere_a_b:3_ex_DP2}.2}
\put(40,72){$a$}
\put(40,42){$b$}
\put(140,-12){Fig. \ref{fig: meno_sfere_a_b:3_ex_DP2}.3}
\put(184,87){$a$}
\put(179,68){$b$}
\end{picture}
\end{center}
\caption{$\mathbb{R}E\simeq S^1$ in thick dashed.}
\label{fig: meno_sfere_a_b:3_ex_DP2}
\end{figure}
\end{exa}
The following definition is used as a (non-)symmetry detector in the proof of Proposition \ref{prop: non-symm_real_scheme}.
\begin{defn}
\label{defn: trivial_oval}
Let $\mathcal{S}$ be a topological type in $S^2$.
We say that $\mathcal{S}$ has a \textit{mirror} if there exist an element in the equivalence class of $\mathcal{S}$ (Definition \ref{defn: realizing_real_scheme}) of the form $\tilde{\mathcal{S}} \sqcup \tilde{\mathcal{S}}\sqcup \mathcal{T}$, with $\tilde{\mathcal{S}}$ different from $0$. 
Otherwise, we say that $\mathcal{S}$ has \textit{no mirrors}.
\end{defn} 
\begin{exa}
Let $\mathcal{S}$ be the topological type $1\sqcup \langle \langle 1 \rangle \rangle$ in $S^{2}$. There exists an element in the equivalence class of $\mathcal{S}$ of the form $\langle 1 \rangle \sqcup \langle 1 \rangle $, it follows that $\mathcal{S}$ has a mirror. An example of topological type in $S^{2}$ with no mirrors is $1 \sqcup \langle 1 \rangle$.
\end{exa}
\begin{proof}[Proof of Proposition \ref{prop: non-symm_real_scheme}]
Let $\mathcal{S}$ be any real scheme in class $d$ in Table \ref{tabella=non-symm} for some fixed integers $d,$ $k_1,k_2$ and $h_1,h_2,h_3,h_4$. To realize $\mathcal{S}$ in $X^{4}$ and in class $d$, let us proceed as in the proof of Proposition \ref{prop: star_label_DP2}.
\begin{itemize}
\item Let $C_S\subset S$ be the real algebraic curve of bi-class $(d,d)$ constructed as in proof of Proposition \ref{prop: S_pencil_quartic_DP2}. We have that $\mathbb{R}E_S \cup \mathbb{R}C_S$ is arranged in $\mathbb{R}S$ respectively as pictured in Fig. \ref{fig: non-symm_3spheres_DP2}.1 and \ref{fig: non-symm_3spheres_DP2}.2.
\item Let $T$ be the quadric ellipsoid and let $C_T\subset T$ be the real algebraic curve of bidegree $(d,d)$ constructed as in proof of Proposition \ref{prop: constr_quadric_DP2_2}. Then $\mathbb{R}E_T \cup \mathbb{R}C_T$ is arranged in $\mathbb{R}T$ as depicted in Fig. \ref{fig: non-symm_ellipsoid}.
\item Thanks to Theorem \ref{thm: weak_patch_DP2} $\mathcal{S}$ is realizable in $X^{4}$ and in class $d$.
\end{itemize}
Moreover, if $d,$ $k_1,k_2$ and $h_1,h_2,h_3,h_4$ respect the extra conditions in Table \ref{tabella=non-symm}, one can prove that $\mathcal{S}$ is non-symmetric in class $d$. Let us prove it by contradiction.\\ 
Assume that $\mathcal{S} $ is symmetric in class $d$. Then, there exist a real maximal quartic $\overline{Q}$ and a degree $d$ real algebraic curve $B \subset \mathbb{C}P^2$ such that the pair $(\mathbb{R}X, \mathbb{R}\phi^{-1}(B))$ realizes $\mathcal{S}$, where $\phi: X \rightarrow \mathbb{C}P^2$ is the double cover of $\mathbb{C}P^2$ ramified along $\overline{Q}$. Remark that the extra conditions of Table \ref{tabella=non-symm} constrain $\mathcal{S}$ to have no mirrors; see Definition \ref{defn: trivial_oval}. The absence of mirrors in $\mathcal{S}$ forces 
every oval of $\mathbb{R}B$ either to intersect $\mathbb{R}\overline{Q}$ or to form a non-injective pair with each oval of $\mathbb{R}\overline{Q}$. Since $\mathcal{S}$ has $2d+1$ ovals, the real curve $B$ has to intersect $\overline{Q}$ in at least $4d+2$ real points: that contradicts Bézout's theorem. It follows that $\mathcal{S}$ is non-symmetric in class $d$.
\end{proof}
\section{Real curves on $k$-sphere real del Pezzo surfaces of degree 1}
\label{sec: DP1}
\subsection{Definitions}
\label{subsec: intro_DP1}
Let $Y$ be $\mathbb{C}P^2$ blown up at eight points in generic position; then, the surface $Y$ is a del Pezzo surface of degree $1$ (see \cite{Russ02}, \cite[Chapter 8]{Dolg12}). The anti-bicanonical map $\psi: Y \rightarrow \mathbb{C}P^3$ is a double ramified cover of an irreducible singular quadric $Q$ in $\mathbb{C}P^3$; the branch locus of $\psi$ consists of the node $V$ of $Q$ and a non-singular cubic section $\tilde{S}$ on $Q$ disjoint from $V$. Conversely, any such double covering is a del Pezzo surface of degree $1$. 
By construction, the anti-canonical class $c_1(Y)$ is the pull back via $\psi$ of the class of a generatrix on $Q$ (\cite{DegIteKha00}). Let  $d$ be equal to $2s+ \varepsilon$, where $s$ in a non-negative integer and $\varepsilon \in \{0,1\}$. The lifting of any algebraic curve $C\subset Q$ of bidegree $(s,\varepsilon)$ via $\psi$ is a real algebraic curve $B$ of class $d=2s+ \varepsilon$ in $Y$ (Definition \ref{defn: curve_d}). \\
\par Let us consider the standard real structure of $\mathbb{C}P^3$. Form now on, assume that $Q$, $\tilde{S}$ are real and $Q$ is the quadratic cone of equation $X^2+Y^2-Z^2=0$ in $\mathbb{C}P^3$. The real part of $Q$ will be depicted as a quadrangle whose opposite sides are identified in a suitable way and the horizontal sides represent the node $V:=[0:0:0:1]$.\\
Let $Q_1$ and $Q_2$ be two distinct disjoint unions of connected components of $\mathbb{R}Q\setminus (\mathbb{R}\tilde{S}\cup \{V\})$ such that each $Q_i$ is bounded by $\mathbb{R}\tilde{S}\cup \{V\}$. There exist two lifts to $Y$ of the real structure of $Q$ via the double cover $\psi$ and the real part of $Y$ is the double of one of the $Q_i$'s. Let $\sigma$ be a lifting to $Y$ of the standard real structure on $Q$. Then  $Y$ is a $k$-sphere real del Pezzo surface of degree $1$ (Definition \ref{defn: ksphere} ) 
if and only if the pair $(\mathbb{R}Q,\mathbb{R}\tilde{S})$ realizes the real scheme depicted in Fig. \ref{fig: diverse_parti_reali_DP1}.k, with $0 \leq k \leq 4$. Moreover  $(Y,\sigma)$ is $\mathbb{R}$-minimal\textsuperscript{\ref{myfootnote}} if and only if 
$Y$ is a $4$-sphere real del Pezzo surface of degree $1$; see \cite{DegKha02}, \cite{DegIteKha00}.
\begin{figure}[!h]
\begin{picture}(100,40)
\put(0,-10){\includegraphics[width=1.00\textwidth]{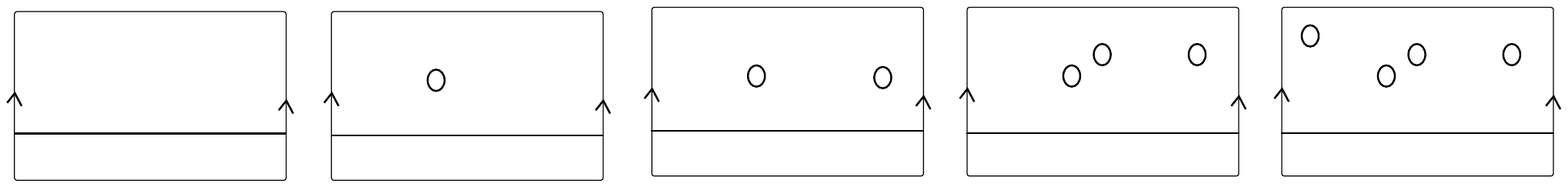}} 
\put(19,-8){Fig. \ref{fig: diverse_parti_reali_DP1}.0}
\put(93,-8){Fig. \ref{fig: diverse_parti_reali_DP1}.1}
\put(169,-8){Fig. \ref{fig: diverse_parti_reali_DP1}.2}
\put(243,-8){Fig. \ref{fig: diverse_parti_reali_DP1}.3}
\put(315,-8){Fig. \ref{fig: diverse_parti_reali_DP1}.4}
\end{picture}
\caption{$(\mathbb{R}Q, \mathbb{R}\tilde{S})$ respectively for $k=0,1,2,3,4$}
\label{fig: diverse_parti_reali_DP1}
\end{figure}
\begin{note}
Let $Y$ be a $k$-sphere real del Pezzo surface of degree $1$, with $0 \leq k \leq 4$. We denote the connected components of $\mathbb{R}Y$ with $Y_0, \dots,Y_k$, where $Y_0$ is homeomorphic to $\mathbb{R}P^2$ and $Y_j$, with $j \geq 1$, is homeomorphic to $S^2$. 
\end{note}
Let $Y$ be a $4$-sphere real del Pezzo surface of degree $1$, then $H_2^-(Y;\mathbb{Z})$ is generated by $c_1(Y)$ (\cite{Russ02}). 
\subsection{$\mathcal{J}$-Obstruction}
\label{subsec: obstructions_DP1}
The number of pseudo-lines of a real algebraic curve of class $d$ in a real minimal del Pezzo surface $Y$ of degree $1$, is determined by $d$.
\begin{prop}
\label{prop: pseudolines}
Let $B$ be a non-singular real algebraic curve of class $d$ in a real minimal del Pezzo surface $Y$ of degree $1$. Then, the real scheme realized by $\mathbb{R}B$ has one and only one pseudo-line if $d\equiv 1 \pmod{2}$ and no pseudo-lines otherwise.
\end{prop}
\begin{proof}
Let $d$ be odd (resp. even). Since the value modulo $2$ of the intersection form on $H_2^-(Y;\mathbb{Z})$ descends on $H_1(\mathbb{R}Y;\mathbb{Z}/2\mathbb{Z}) \simeq H_1(\mathbb{R}P^2;\mathbb{Z}/2\mathbb{Z})$, it follows that $\mathbb{R}B$ has an odd (resp. even) number of pseudo-lines. Moreover, the real part of $B$ has at most one pseudo-line since any two pseudo-lines meet in at least one point.
\end{proof}
The number of connected components of a real curve in a $k$-sphere real del Pezzo surface of degree $1$ is bounded as follows.
\begin{prop}
\label{prop: H-K_DP1}
Let $B$ be a non-singular real algebraic curve of class $d$ in a $k$-sphere real del Pezzo surface $Y$ of degree $1$, with $0 \leq k \leq 4$. Then, the number $l$ of connected components of $\mathbb{R}B$ is bounded as follows:
$$\varepsilon \leq l \leq \frac{d(d-1)}{2}+2,$$
where $\varepsilon \in \{0,1\}$ is such that $\varepsilon \equiv d \pmod{2}$.
\end{prop}
\begin{proof}
The right inequality follows from Harnack-Klein's inequality and the adjunction formula; while the left one follows from Proposition \ref{prop: pseudolines}.
\end{proof}
\subsection{Real schemes}
\label{subsec: real_schemes_DP1}
\begin{defn}
\label{defn: real_schemeDP1}
Let $Y$ be a $k$-sphere real del Pezzo surface $X$ of degree $1$, with $0 \leq k \leq 4$. Let us denote $\mathcal{S}_{DP1}(Y,k)$ the set of all real schemes in $Y$. \\
Notice that $\mathcal{S}_{DP1}(X,k)$ does not depend on the choice of $Y$. Therefore, from now on, we omit $Y$ and write $\mathcal{S}_{DP1}(k).$
\end{defn}
Let us enrich the notion of real scheme with some extra conditions deriving from Proposition \ref{prop: H-K_DP1}.

\begin{defn}
Let $\mathcal{S}$ be in $\mathcal{S}_{DP1}(k)$. We say that $\mathcal{S}$ is \textit{in class} $d$ or we write $\mathcal{S} \in \mathcal{S}_{DP1}(k,d)$ if the number $l$ of connected components of $\mathcal{S}$ is bounded as follows
$$\varepsilon \leq l \leq \frac{d(d-1)}{2}+2,$$
where $\varepsilon \in \{0,1\}$ and $\varepsilon \equiv d \mod{2}$.
\end{defn}

\begin{defn}
We say that $\mathcal{S} \in \mathcal{S}_{DP1}(k,d)$ is \textit{realizable in} $Y^{k}$ \textit{and in class} $d$, if there exist a $k$-sphere real del Pezzo surface $Y^{k}$ of degree $1$ and a real algebraic curve $B \subset Y^{k}$ of class $d$, such that the pair $(\mathbb{R}Y^{k}, \mathbb{R}B)$ realizes $\mathcal{S}$. 
\end{defn}
Let us lighten the real scheme notation introduced in Section \ref{subsec: cha2_codage_isotopie}.
\begin{note}
\label{note: simplify_notation_DP1}
Let $\mathcal{S}:=\mathcal{T}|\mathcal{S}_1:\dots: \mathcal{S}_4$ be a topological type in the disjoint union of a real projective plane and $4$ spheres. Let $Y^{k}$ be a $k$-sphere real del Pezzo surface of degree $1$, with $0 \leq k \leq 4$.  If at least $4-k$ entries $S_{j}$ are $0$, we say that $\mathcal{S}$ is a real scheme in $\mathbb{R}Y^{k}$.
\end{note}
\subsection{Positive and negative connected components}
\label{subsec: positive_negative_DP1} 
One can refine real scheme classifications in different ways. About $k$-sphere real del Pezzo surfaces of degree $1$, we are interested in a classification refinement which involves only $4$-sphere real del Pezzo surfaces of degree $1$. In this case, one can define a notion of positivity of the spheres inducing a refined classification.
\par Let $Y$ be a $4$-sphere real del Pezzo surface of degree $1$. The anti-bicanonical map $\psi$ of $Y$ is a double cover of $Q$ ramified along a real maximal cubic section $\tilde{S}$ and the vertex of $Q$. 
Independently from the choice of a complex orientation on $\mathbb{R}\tilde{S}$, we can distinguish the connected components of $\mathbb{R}\tilde{S}$ on $\mathbb{R}Q$ in the following way. \\
Let us consider $Q \setminus \{V\}$. There are four connected components, called \textit{ovals}, of $\mathbb{R}\tilde{S}$ realizing the trivial class in $H_1(\mathbb{R}Q \setminus \{V\}; \mathbb{Z}/2\mathbb{Z})$; while the connected component of $\mathbb{R}\tilde{S}$ realizing the non-trivial class in $H_1(\mathbb{R}Q \setminus \{V\}; \mathbb{Z}/2\mathbb{Z})$, is called \textit{long-component}. If the union of an oval and the long-component of $\mathbb{R}\tilde{S}$ in $\mathbb{R}Q \setminus \{V\}$ bounds an oriented surface, the oval is called \textit{positive}; otherwise \textit{negative}. 
\begin{figure}[!h]
\begin{picture}(100,107)
\put(-20,-25){\includegraphics[width=1.10\textwidth]{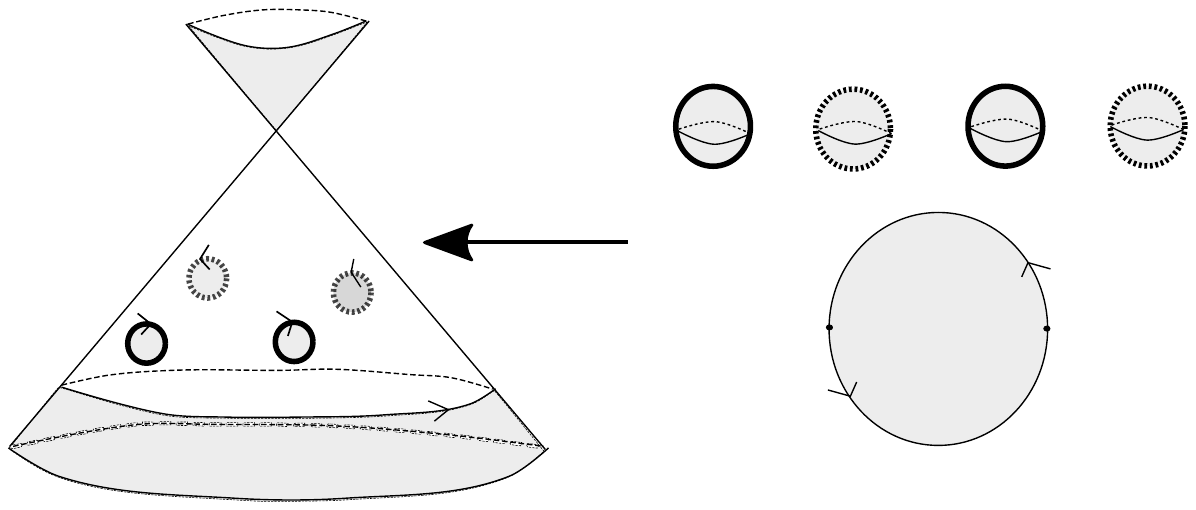}} 
\put(95,-5){$(\mathbb{R}Q,\mathbb{R}\tilde{S})$}
\put(165,72){$\psi$}
\put(241,-5){$\mathbb{R}Y$}
\end{picture}
\caption{Positive and negative connected components of $\mathbb{R}Y$.}
\label{fig: da sigma2_alla_del_pezzo}
\end{figure}
Each oval of $\mathbb{R}\tilde{S}$ is either positive or negative, independently from the choice of a complex orientation on $\mathbb{R}\tilde{S}$; moreover, two ovals are positive and two are negative. \\
Via $\psi$ we can label the connected components of $\mathbb{R}Y$ homeomorphic to $S^2$ as positive and negative; see Fig. \ref{fig: da sigma2_alla_del_pezzo}, where the preimage of negative and positive ovals, is depicted respectively in dashed and thick black. It follows that for any fixed non-negative integer $d$, we have a refined and a non-refined topological classification of real algebraic curves of class $d$ in $4$-sphere real del Pezzo surfaces of degree $1$.
\begin{defn}
\label{defn: coarse_real_scheme}
A real scheme in $\mathcal{S}_{DP1}(4)$ up to refined homeomorphism, is called a \textit{refined real scheme}. 
\end{defn}
\begin{note}
\label{note: coarse_real_scheme}
\begin{itemize}
\item[]
\item We use the convention that the connected components $Y_1$ and $Y_2$ are positive and $Y_3$ and $Y_4$ negative.
\item Let $ \mathcal{S}$ be a refined real scheme, we write $$ \mathcal{S}:=\mathcal{S}_0  |  \mathcal{S}_1:\mathcal{S}_2:\mathcal{S}_3:\mathcal{S}_{4}$$
if $\mathcal{S}_1 : \mathcal{S}_{2}$ encodes the oval arrangement on $Y_1 \sqcup Y_{2}$ and $\mathcal{S}_3 : \mathcal{S}_{4}$ that on $Y_3 \sqcup Y_{4}$.
\end{itemize}
\end{note}
\begin{defn}
\label{defn: refined_real_scheme_realizable}
We say that a \textit{refined real scheme} $\mathcal{S}:=\mathcal{S}_0  |  \mathcal{S}_1:\mathcal{S}_2:\mathcal{S}_3:\mathcal{S}_{4}$ is \textit{realizable in} $Y$ \textit{and in class} $d$ if there exist a $4$-sphere real del Pezzo surface $Y$ of degree $1$ and a class $d$ real algebraic curve $B\subset Y$ such that the pair 
\begin{itemize}
\item $(\mathbb{R}Y, \mathbb{R}B)$ realizes $\mathcal{S}$;
\item $(Y_{i} \sqcup Y_{j}, \mathbb{R}B)$ realizes $\mathcal{S}_{i} : \mathcal{S}_{j}$, for $(i,j) \in \{(1,2), (3,4)\}$.
\end{itemize}
\end{defn}

\subsection{Bézout-type obstrction, $d \geq 4$}
The $\mathcal{J}$-obstruction (Propositions \ref{prop: pseudolines}) is the only restriction for (refined) real schemes up to class $3$. The next statement provides an example of additional obstructions for real schemes in class $d \geq 4$ in $4$-sphere real del Pezzo surfaces of degree $1$. 
\begin{prop}
\label{prop: bezout_del_pezzo_deg_1}
Let $B$ be a non-singular real algebraic curve of class $d= 2s+ \varepsilon$ in a $4$-sphere real del Pezzo surface $Y$ of degree $1$, where $\varepsilon\in \{0,1\}$ and $s \in \mathbb{Z}_{\geq 1}$. Assume that all connected components of $\mathbb{R}B$ lie on $Y_0$ and on $t$ of the $Y_j$'s, for $j=1,2,3,4$. Assume that $N_{h}$, with $h=1,2,3$, are three nests of depth $i_h$ of $\mathbb{R}B$. Moreover, assume that $i_1 \leq i_2$ and $N_{1},$ $N_{2}$ form a disjoint pair of nests in $Y_0$; while $N_{3}$ lies on some $Y_j$, where $j\not = 0$. Then, we have the following restrictions on the depths of the nests:
\begin{equation}
\label{eq: obstr_1}
i_1+i_2 \leq 3s+ \varepsilon -t;
\end{equation}
\begin{equation}
\label{eq: obstr_2}
i_2+i_3 \leq 3s + \varepsilon - (t-1).
\end{equation}
\end{prop}
\begin{proof}
The argument of the proof is similar to that used in the proof of Lemma \ref{lem: caso_particolare_4_sfere_k=2}. In fact, it is sufficient to combine 
\begin{itemize}
\item the existence of a real algebraic curve $T$ of class $3$ and genus $4$ on $Y$ passing through a given real configuration $\mathcal{P}$ of $6$ distinct points;
\item the observation that $\mathbb{R}T$ has exactly one connected component on each connected component of $\mathbb{R}Y$, if each connected component of $\mathbb{R}Y$ contains at least one point of $\mathcal{P}$;
\item the strategy adopted in the proof of Proposition \ref{prop: bezout_inv_welschinger}.
\end{itemize}
\end{proof}
\subsection{Main results}
\label{subsec: main_DP1}
We classify real schemes up to class $3$ in $Y^{k}$, for all $0 \leq k \le 4$ (Proposition \ref{prop: delpezzo_1}). Furthermore, we show that the real scheme classification and the refined real scheme classification are the same up to class $3$ in $Y^{4}$ (Theorem \ref{thm: delpezzo_1}).
\begin{thm} 
\label{thm: delpezzo_1}
Any refined real scheme in class $d \in \{1,2,3\}$ which is not prohibited by Proposition \ref{prop: pseudolines}, is realizable in $Y$ and in class $d$.
\end{thm}
\begin{proof}
The statement follows from the proofs of Propositions \ref{prop: class_3_DP1}, \ref{prop: constr_small_pert}, \ref{prop: constr_harnack} and Corollary \ref{cor: final_del_pezzo_deg_1}. 
\end{proof}

\begin{prop}
\label{prop: delpezzo_1}
Any real scheme in $\mathcal{S}_{DP1}(k,d)$ which is not prohibited by Proposition \ref{prop: pseudolines}, with $d \in \{1,2,3\}$ and $0 \leq k \leq 4$, is realizable in $Y^{k}$ and in class $d$.
\end{prop}
\begin{proof}
If $k=4$, the statement follows from Theorem \ref{thm: delpezzo_1}. For $0 \leq k \leq 3$, it follows from the proofs of Propositions \ref{prop: class_3_DP1}, \ref{prop: constr_small_pert}, \ref{prop: constr_harnack} and Corollary \ref{cor: final_del_pezzo_deg_1}.
\end{proof}

\subsection{Constructions}
\label{subsec: construction_DP1}
This section is organized as follows. 
The (refined) real schemes in class $2$ which are in Table \ref{tabellalastchapter}, are realized in the proofs of Proposition \ref{prop: constr_harnack} and Corollary \ref{cor: final_del_pezzo_deg_1}. All the remaining (refined) real schemes in class $2$ and those in class $1$ are constructed in the proof of Proposition \ref{prop: constr_small_pert}. Finally, the (refined) real schemes in class $3$ are realized in Proposition \ref{prop: class_3_DP1} and their construction follows from that of (refined) real schemes in class $2$.   
\begin{table}[h!]
\centering
\begin{tabular}{ |l| l | l |}
\hline
&Real scheme in class $2$& Refined real scheme in class $2$\\
\hline
$(1)$&$0|\alpha:\beta:\gamma:0,$ &$0|\alpha:\beta:\gamma:0,$ \\
& with $ 0 \leq \alpha+ \beta + \gamma \leq 3$ &$0|0:\alpha:\beta:\gamma,$  with $ 0 \leq \alpha+ \beta + \gamma \leq 3$\\
\hline
$(2)$&$0|\langle \langle 1 \rangle \rangle:0:0:0$ &  $0|\langle \langle 1 \rangle \rangle:0:0:0$\\
&&$0|0:0:\langle \langle 1 \rangle \rangle:0$\\
\hline
\end{tabular}
\caption{\label{tabellalastchapter} The (refined) real schemes in $\mathcal{S}_{DP1}(k,2)$ realized in Proposition \ref{prop: constr_harnack} and Corollary \ref{cor: final_del_pezzo_deg_1}.}
\end{table}
\begin{defn}
\label{defn: curves_cp3}
The blow-up of $Q$ at the vertex $V$ is the second Hirzebruch surface $\Sigma_{2}$. 
Let $k$ and $l$ be two non-negative integers. We say that an algebraic curve $C$ on $Q$ has bidegree $(k,l)$ if the strict transform of $C$ in $\Sigma_{2}$ 
has bidegree $(k,l)$; see Section \ref{subsec: cha2_Hirzebruch_surf}.
\end{defn}
\begin{rem}
\label{rem: DP1}
In order to prove Theorem \ref{thm: delpezzo_1} and Proposition \ref{prop: delpezzo_1}, we reduce the construction of real algebraic curves of class $d$ in $Y^{k}$, with prescribed topology, to the construction of pairs of real algebraic curves in $Q$ with prescribed topology as follows.\\
To realize any fixed (refined) real scheme $\mathcal{S}$ in $\mathcal{S}_{DP1}(k,d)$, it is sufficient:
\begin{itemize}
\item to prove the existence of a non-singular real curve $\tilde{S}$ of bidegree $(3,0)$ on $Q$, realizing the real scheme in Fig. \ref{fig: diverse_parti_reali_DP1}.k;
\item to realize a real scheme $\eta$ by a bidegree $(s, \varepsilon)$ real curve $C \subset Q$,  
\end{itemize}
such that the lifting via $\psi$ of $C$ realizes $\mathcal{S}$ on $Y^{k}$, where $\psi: Y^{k} \rightarrow Q$ is the double cover of $Q$ ramified along $\tilde{S}$ and $V$. Moreover, for $k=4$ and $\mathcal{S}$ refined, one has to equip $\mathbb{R}\tilde{S}$ with one of the two complex orientations. 
\end{rem}
\subsubsection{Moving hyperplanes}
\label{subsubsec: small_pert_DP1}
Let us prove a part of Theorem \ref{thm: delpezzo_1} and Proposition \ref{prop: delpezzo_1}.
\begin{prop}
\label{prop: constr_small_pert}
All (refined) real schemes in class $d=1,2$ which are neither prohibited by Proposition \ref{prop: pseudolines} nor in Table \ref{tabellalastchapter}, are realizable in $Y^{k}$ (resp. in $Y$) and in class $d$, with $0 \leq k \leq 3$.
\end{prop}
\begin{proof}
Let us realize all refined real schemes in $\mathcal{S}_{DP1}(4,d)$ described in the statement, following the steps of Remark \ref{rem: DP1}.\\
There exists a real algebraic maximal curve $\tilde{S}$ of bidegree $(3,0)$ on $Q$ realizing the $\mathcal{L}$-scheme $\tilde{\eta}$ depicted in Fig. \ref{fig: cubiche_DP1}.1, respectively in Fig. \ref{fig: cubiche_DP1}.2 (the real graph associate to $\tilde{\eta}$ is completable in degree $2$. Therefore, Theorem \ref{thm: existence_trigonal} assures the existence of such $\tilde{S}$). Choose one of the two complex orientations on $\mathbb{R}\tilde{S}$.\\
\textbf{Case $d=1$}: Any real generatrix $F$ on $Q$ lifts to a real algebraic curve of class $1$ on a $4$-sphere real del Pezzo surface of degree $1$; in Fig. \ref{fig: cubiche_deg1_DP1} from the depicted real generatrices on $Q$, one recovers all refined real schemes in $\mathcal{S}_{DP1}(4,1)$.\\ 
\textbf{Case $d=2$}: For any two generatrices on $Q$, there exists a hyperplane section $H \subset \mathbb{C}P^{3}$ passing through $V$ such that $Q \cap H= F_{1} \cup F_{2}$. Moving slightly $H$, one construct a real algebraic curve $Z_{1}$ of bidegree $(1,0)$ on $Q$ such that $(\mathbb{R}Q \setminus \{V\}, \mathbb{R}\tilde{S}, \mathbb{R}(F_{1}\cup F_{2}))$ and $(\mathbb{R}Q \setminus \{V\}, \mathbb{R}\tilde{S}, \mathbb{R}Z_{1})$ realize the same topological type. Considering all the possible arrangements of $\tilde{S}$ and two generatrices, one can construct real algebraic curves $Z_{1}$ of bidegree $(1,0)$ which lifts to real algebraic curves of class $2$ on a $4$-sphere real del Pezzo surface of degree $1$ realizing all refined real schemes in $\mathcal{S}_{DP1}(4,2)$, but those in Table \ref{tabellalastchapter}.\\
One can realize all remaing real schemes in $\mathcal{S}_{DP1}(k,d)$, with $0 \leq k \leq 3$, applying analogous construction techniques.
\end{proof}
\begin{figure}[!h]
\begin{picture}(100,45)
\put(0,0){\includegraphics[width=1.00\textwidth]{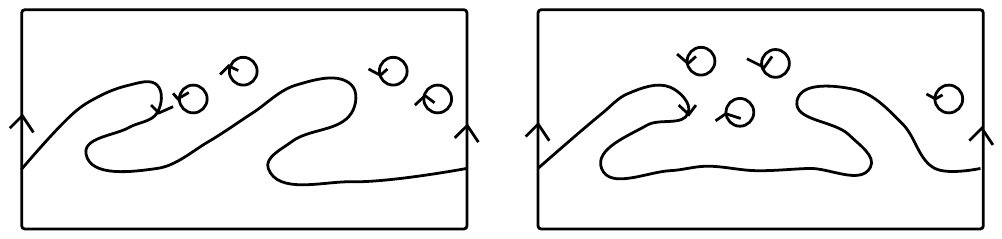}} 
\put(111,-7){Fig. \ref{fig: cubiche_DP1}.1}
\put(205,-7){Fig. \ref{fig: cubiche_DP1}.2}
\end{picture}
\caption{ }
\label{fig: cubiche_DP1}
\end{figure}
\begin{figure}[!h]
\begin{picture}(100,25)
\put(0,-37){\includegraphics[width=1.00\textwidth]{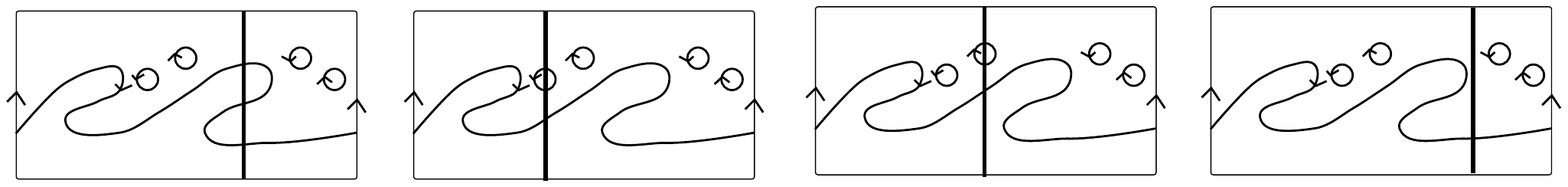}} 
\put(47,34){\textcolor{black}{$\mathbb{R}F$}}
\end{picture}
\caption{ }
\label{fig: cubiche_deg1_DP1}
\end{figure}
\begin{exa}[Application of the proof of Proposition \ref{prop: constr_small_pert}]
\label{exa: small_pert} 
In Fig. \ref{fig: spm1}.1 we fix an $\mathcal{L}$-scheme of $\tilde{S}$ and two generatrices $F_{1}$, $F_{2}$. In Fig. \ref{fig: spm1}.2, it is depicted the topological type realized by the triplet $(\mathbb{R}Q, \mathbb{R}\tilde{S},\mathbb{R}Z_{1})$. The existence of such a pair $(\tilde{S}, Z_{1})$ implies the realization of the refined real scheme $1 \sqcup \langle 1 \rangle|0:0:0:0$ in $\mathcal{S}_{DP1}(4,2)$.
\end{exa}
\begin{figure}[!h]
\begin{picture}(100,55)
\put(6,10){\includegraphics[width=1.0\textwidth]{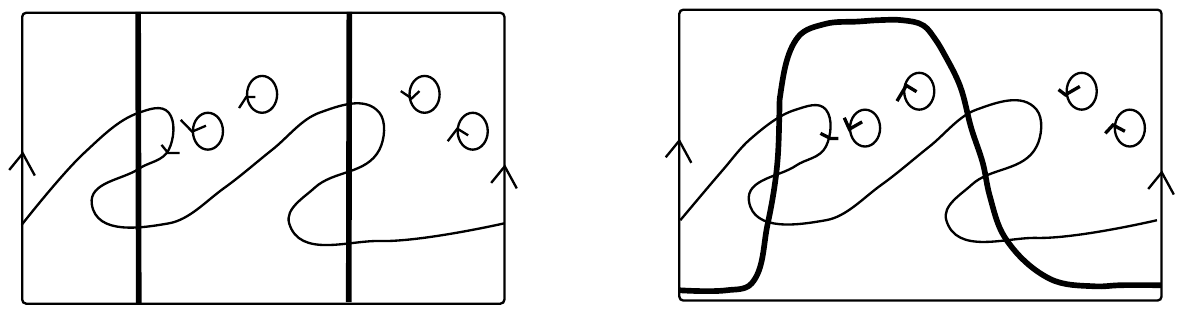} } 
\put(99,-5){Fig. \ref{fig: spm1}.1}
\put(220,-5){Fig. \ref{fig: spm1}.2}
\end{picture}
\caption{ }
\label{fig: spm1}
\end{figure}
\subsubsection{Harnack's construction method on Q}
\label{subsubsec: harnack_DP1}
In the proof of Propositions \ref{prop: constr_harnack}, we use a variant of Harnack's construction method. 
\begin{prop}
\label{prop: constr_harnack}
Any (refined) real scheme in class $2$ in $(1)$ of Table \ref{tabellalastchapter} is realizable in $Y^{k}$ (resp. $Y$) and in class $2$, with $0 \leq k \leq 3$.
\end{prop}
\begin{proof}
Fix a non-singular real algebraic curve $Z_1$ of bidegree $(1,0)$ on $Q$. Pick any other real algebraic curve $L_1$ of bidegree $(1,0)$ on $Q$ such that $Z_1 \cap L_1$ consists of two distinct real points. Let $P_0(x,y)P_1(x,y)=0$ be a polynomial equation defining the union of $Z_1$ and $L_1$ in some local affine chart of $Q$. Choose $4$ real generatrices $F_i$ on $Q$, with $i=1,2,3,4$, intersecting transversely $Z_1 \cup L_1$. Replace the left side of the equation $P_0(x,y)P_1(x,y)=0$ with $P_0(x,y)P_1(x,y)+ \varepsilon f_1(x,y)f_2(x,y)f_3(x,y)f_4(x,y)$, where $f_i(x,y)=0$ is an affine equation $F_i$ and $\varepsilon \not = 0$ is a sufficient small real number. Up to a choice of the sign of $\varepsilon$, one constructs a small perturbation $L_2$ of $Z_1 \cup L_1$,
where $L_2$ is a non-singular real curve of bidegree $(2,0)$ such that $\bigsqcup_{i=1}^4F_i \cap Z_1= L_2 \cap Z_1$. \\
Analogously, by a small perturbation of $Z_1 \cup L_2$, one can construct a non-singular real algebraic curve $\tilde{S}$ of bidigree $(3,0)$ on $Q$. \\
Following Remark \ref{rem: DP1}, one can construct a pair of real algebraic curves $(\tilde{S}, Z_{1})$ as explained above in order to realize any (refined) real schemes in $(1)$ of Table \ref{tabellalastchapter}  in $Y^{k}$ (resp. $Y$) and in class $2$. 
\end{proof}
\begin{exa}
The construction of a pair $(\tilde{S}, Z_{1})$ described in the proof of Proposition \ref{prop: constr_harnack}, is divided into two steps. An example of each of such steps is depicted from right to left in Fig. \ref{fig: harnack_DP1}.1 - \ref{fig: harnack_DP1}.2 and in Fig. \ref{fig: harnack_DP1}.2 - \ref{fig: harnack_DP1}.3. Each pair of figures represent a still image of \textit{before} and \textit{after} a small perturbation; the dashed segments are the generatrices $F_i$'s. The existence of a pair $(\tilde{S}, Z_{1})$ with real scheme as depicted in Fig. \ref{fig: harnack_DP1}.3, implies the realization of the refined real scheme $0|1:1:1:0$ in $\mathcal{S}_{DP1}(4,2)$.
\end{exa}
\begin{figure}[!h]
\begin{picture}(100,40)
\put(8,0){\includegraphics[width=1.00\textwidth]{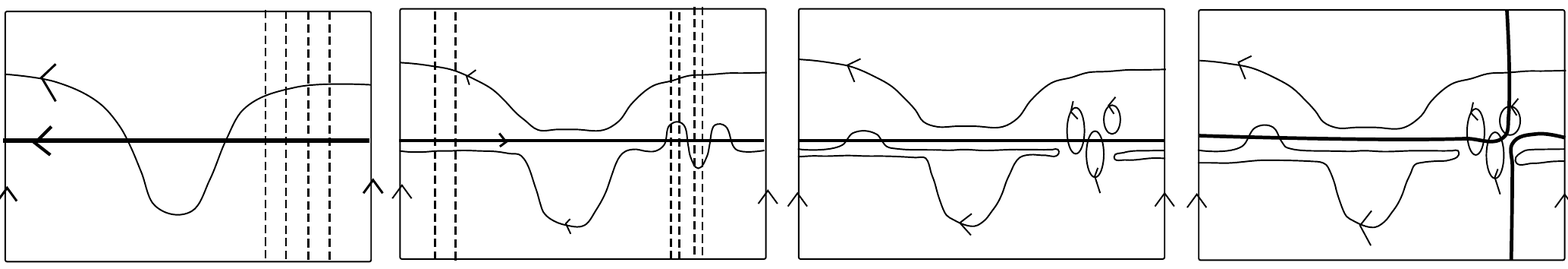}} 
\end{picture}
\put(-112,30){\textcolor{black}{$\mathbb{R}Z_1$}}
\put(-112,55){\textcolor{black}{$\mathbb{R}L_1$}}
\put(-37,72){\textcolor{black}{$\mathbb{R}F_i$'s}}
\put(-67,-5){Fig. \ref{fig: harnack_DP1}.1}
\put(248,72){\textcolor{black}{$\mathbb{R}Z_2$}}
\put(25,-5){Fig. \ref{fig: harnack_DP1}.2}
\put(120,-5){Fig. \ref{fig: harnack_DP1}.3}
\put(216,-5){Fig. \ref{fig: harnack_DP1}.4}
\caption{$1-3$: Example of Harnack's construction method. $4$: Application of the construction in Proposition \ref{prop: class_3_DP1}.}
\label{fig: harnack_DP1}
\end{figure}
\subsubsection{Five particular constructions and Class $3$}
\label{subsubsec: particular_bitchy_contruction_DP1}
Thanks to Corollary \ref{cor: final_del_pezzo_deg_1} and Proposition \ref{prop: class_3_DP1}, we end the proof of Theorem \ref{thm: delpezzo_1} and Proposition \ref{prop: delpezzo_1}. First of all, we give some intermediate constructions in Proposition \ref{prop: (2,2)} whose proofs rely on Viro's patchworking method (\cite{Viro84a}, \cite{Viro84b}, \cite{Viro89}, \cite{Viro83}) and constructions via \textit{dessins d'enfants} (Section \ref{subsec: cha2_dess_enf}). In Proposition \ref{prop: (2,2)}, we  talk about charts of real algebraic curves; this is a key notion of the Viro method and a definition can be found in \cite{Viro83}, \cite{IteShu03}.
\begin{prop}
\label{prop: (2,2)}
There exist real algebraic curves of bidegree $(3,0)$ on $Q$ with charts respectively as depicted in Fig. \ref{fig: chart_particular_2}.
\end{prop}
\begin{proof}
\begin{figure}[!h]
\begin{picture}(100,125)
\put(47,63){Fig. \ref{fig: chart_particular_2}.1}
\put(164,63){Fig. \ref{fig: chart_particular_2}.2}
\put(284,63){Fig. \ref{fig: chart_particular_2}.3}
\put(108,-3){Fig. \ref{fig: chart_particular_2}.4}
\put(228,-3){Fig. \ref{fig: chart_particular_2}.5}
\put(-19,7){\includegraphics[width=1.1\textwidth]{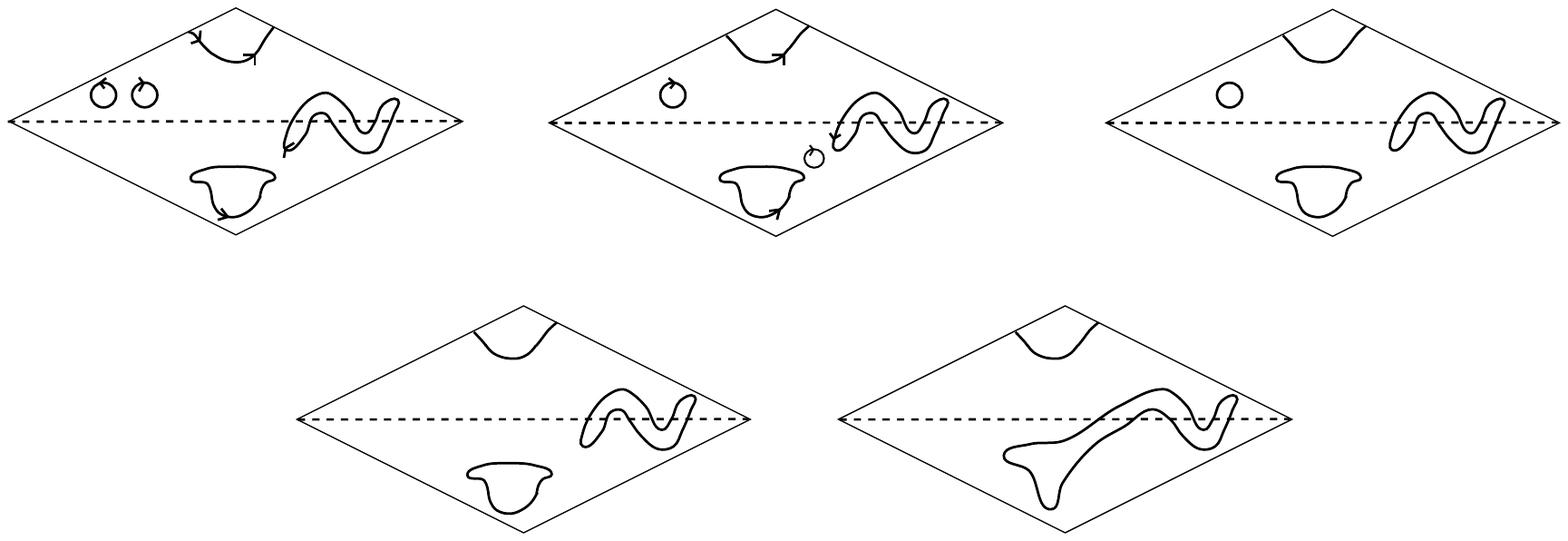}} 
\end{picture}
\caption{Charts of real algebraic curves.}
\label{fig: chart_particular_2}
\end{figure}
\begin{figure}[!h]
\begin{picture}(100,125)
\put(43,67){Fig. \ref{fig: chart_particular_3}.1}
\put(161,67){Fig. \ref{fig: chart_particular_3}.2}
\put(285,67){Fig. \ref{fig: chart_particular_3}.3}
\put(104,-3){Fig. \ref{fig: chart_particular_3}.4}
\put(222,-3){Fig. \ref{fig: chart_particular_3}.5}
\put(-15,7){\includegraphics[width=1.1\textwidth]{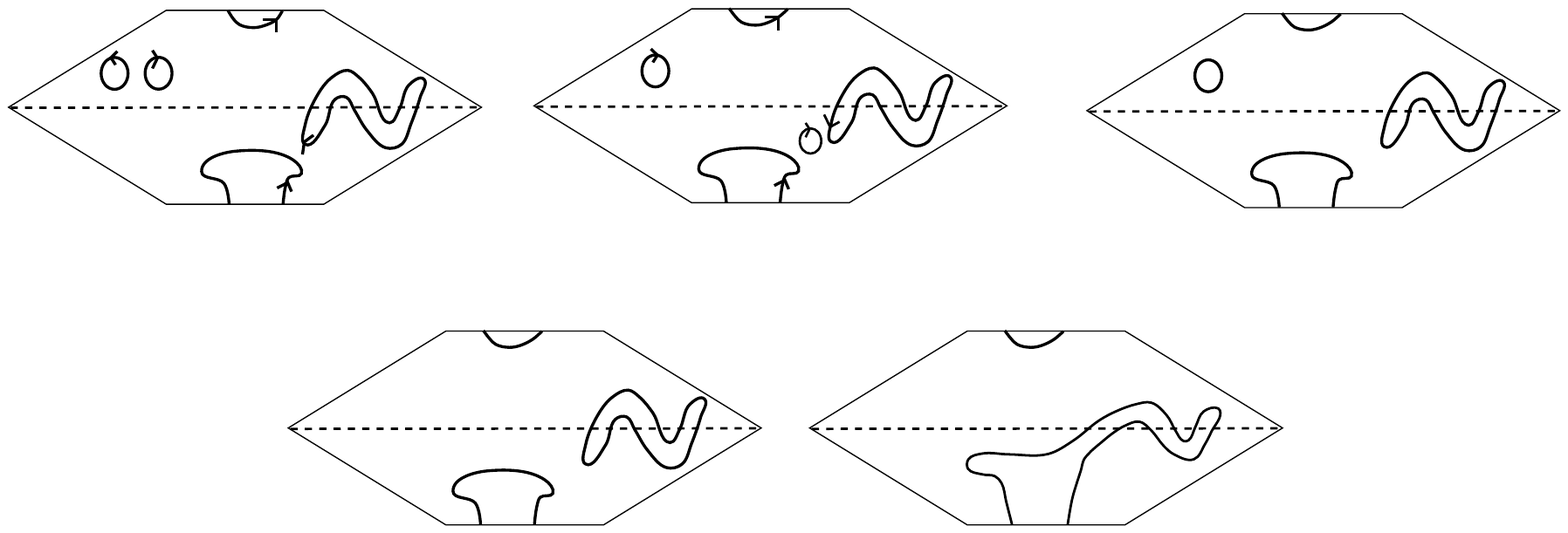}} 
\end{picture}
\caption{Charts of real algebraic curves.}
\label{fig: chart_particular_3}
\end{figure}

\begin{figure}[!h]
\begin{picture}(100,130)
\put(63,67){Fig. \ref{fig: intermediate_construction_2}.1}
\put(80,111){$p_6$}
\put(55,111){$p_5$}
\put(163,67){Fig. \ref{fig: intermediate_construction_2}.2}
\put(174,110){$p_6$}
\put(149,110){$p_5$}
\put(277,67){Fig. \ref{fig: intermediate_construction_2}.3}
\put(294,110){$p_6$}
\put(269,110){$p_5$}
\put(115,-3){Fig. \ref{fig: intermediate_construction_2}.4}
\put(135,42){$p_6$}
\put(110,42){$p_5$}
\put(224,-3){Fig. \ref{fig: intermediate_construction_2}.5}
\put(240,43){$p_6$}
\put(215,43){$p_5$}
\put(-15,3){\includegraphics[width=1.1\textwidth]{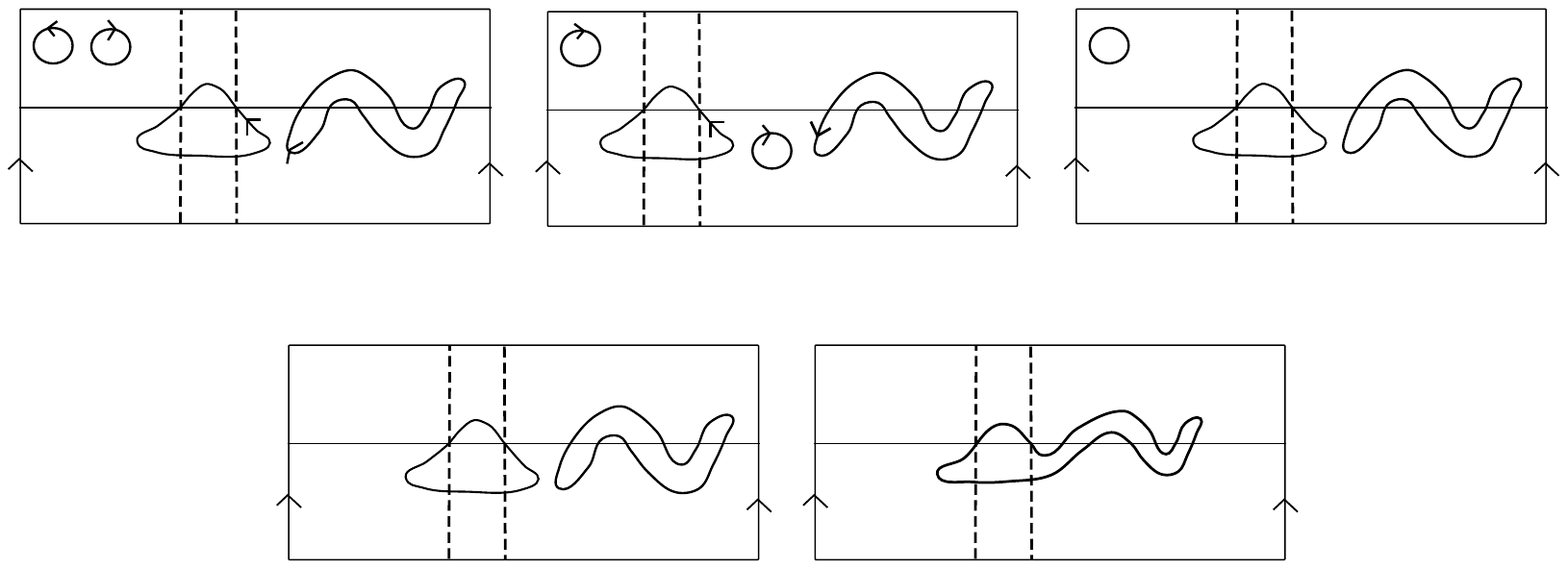}} 
\end{picture}
\caption{Intermediate constructions.}
\label{fig: intermediate_construction_2}
\end{figure}
\begin{figure}[!h]
\begin{picture}(100,230)
\put(25,112){Fig. \ref{fig: intermediate_construction_3}.1}
\put(174,112){Fig. \ref{fig: intermediate_construction_3}.2}
\put(317,112){Fig. \ref{fig: intermediate_construction_3}.3}
\put(95,-1){Fig. \ref{fig: intermediate_construction_3}.4}
\put(246,-1){Fig. \ref{fig: intermediate_construction_3}.5}
\put(-15,7){\includegraphics[width=1.1\textwidth]{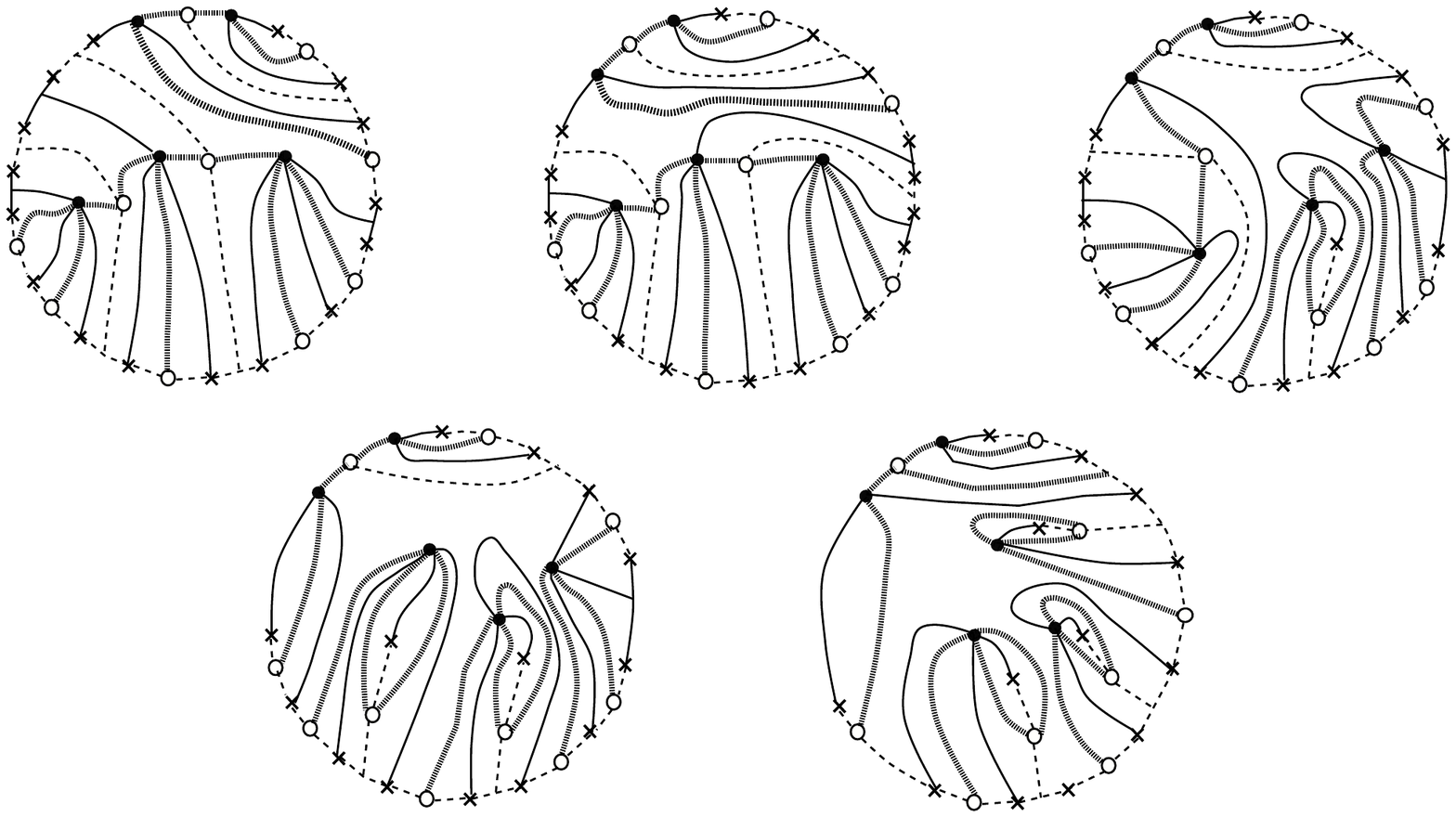}} 
\end{picture}
\caption{Intermediate constructions.}
\label{fig: intermediate_construction_3}
\end{figure}
As first step, let us construct real algebraic curves of bidegree $(2,2)$ in $\Sigma_2$ with charts respectively as in Fig. \ref{fig: chart_particular_3}. \\
Let $\tilde{\eta}_i$ 
be a trigonal $\mathcal{L}$-scheme on $\mathbb{R}\Sigma_4$ as in Fig. \ref{fig: intermediate_construction_2}, respectively with $i=1,2,3,4,5$. By Theorem \ref{thm: existence_trigonal} 
the completions, depicted in Fig. \ref{fig: intermediate_construction_3}, of the real graphs associated to the $\tilde{\eta}_i$'s prove the existence of real trigonal curves $\tilde{D}_i$'s in $\Sigma_4$ realizing the $\eta_i$'s. Moreover, the $\tilde{D}_i$'s are reducible because they have $8$ non-degenerate double points and their normalizations have $5$ real connected components. In addition, the $\tilde{D}_i$'s have to be the union of a real curve of bidegree $(2,0)$ and a real curve of bidegree $(1,0)$.\\
Let us consider, as defined in Section \ref{subsec: cha2_Hirzebruch_surf}, the birational transformation $$\Xi:=\beta^{-1}_{p_{5}}\beta^{-1}_{p_{6}}:(\Sigma_4,
\tilde{D}_i)\dashrightarrow (\Sigma_2, D_i\cup A),$$ where the points $p_5,p_6$, are the real double points of $\tilde{D}_i$ as depicted in Fig. \ref{fig: intermediate_construction_2}, where the dashed real fibers are those intersecting $p_5$ and $p_6$. The image via $\Xi$ of the reducible real trigonal curve $\tilde{D}_i$ is a reducible curve, in particular the union of two non-singular real curves $D_i$ and $A$, which are respectively of bidegree $(2,2)$ and $(1,0)$ in $\Sigma_2$. Moreover, the charts of the $D_i$'s are respectively as in Fig. \ref{fig: chart_particular_3}.\\
Finally, we can apply Viro's patchworking method to the polynomials and charts of the $D_i$'s and a non-singular real algebraic curve of bidegree $(1,0)$ in $\Sigma_2$, and we construct bidegree $(3,0)$ real algebraic maximal curves on $\mathbb{C}P^3$, whose charts and arrangements with respect to the coordinates axis $\{y=0\}$ are as in Fig. \ref{fig: chart_particular_2}.
\end{proof}
Proposition \ref{prop: (2,2)} directly implies the following statement.
\begin{cor}
\label{cor: final_del_pezzo_deg_1}
Any (refined) real scheme in $(2)$ of Table \ref{tabellalastchapter} is realizable in $Y^{k}$ (resp. $Y$) and in class $2$, with $0 \leq k \leq 3$.
\end{cor}
\begin{proof}
Proposition \ref{prop: (2,2)} implies the existence of pairs of non-singular real algebraic curves $\tilde{S},$ $Z_1 \subset Q$ respectively of bidegree $(3,0)$ and $(1,0)$, with topology prescribed by the charts in Fig. \ref{fig: chart_particular_2}. Therefore, thanks to Remark \ref{rem: DP1}, the (refined) real scheme in $(2)$ of Table \ref{tabellalastchapter} are realizable in $Y^{k}$ (resp. $Y$) and in class $2$, with $0 \leq k \leq 3$.
\end{proof}
\begin{prop}
\label{prop: class_3_DP1}
Any (refined) real scheme in class $3$ is realizable in $Y^{k}$ (resp. $Y$) and in class $3$, with $0 \leq k \leq 3$.
\end{prop}
\begin{proof}
Following the steps of Remark \ref{rem: DP1}, one realizes all (refined) real schemes in $\mathcal{S}_{DP1}(k,3)$ constructing pairs $(\tilde{S}, Z_{2})$ of real algebraic curves of bidegree respectively $(3,0)$ and $(1,1)$ using the pairs $(\tilde{S}, Z_{1})$ previously constructed in the proofs of Propositions \ref{prop: constr_small_pert}, \ref{prop: constr_harnack} and Corollary \ref{cor: final_del_pezzo_deg_1}, where $Z_{1}$ has bidegree $(1,0)$. We present the proof for a particular real scheme. The remaining constructions are similar to that performed.\\
\begin{figure}[!h]
\begin{picture}(100,85)
\put(-30,10){\includegraphics[width=1.19\textwidth]{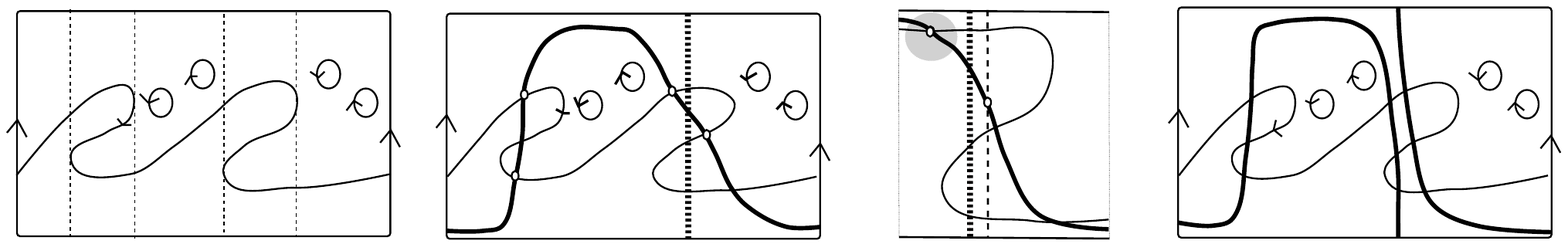} } 
\put(-5,75){$\overbrace{ }^{I_{1}}$}
\put(38,75){$\overbrace{ }^{I_{2}}$}
\put(10,-5){Fig. \ref{fig: spm2}.1: }
\put(127,-5){Fig. \ref{fig: spm2}.2}
\put(233,-5){Fig. \ref{fig: spm2}.3:}
\put(220,75){$\overbrace{\quad \quad \quad \quad \quad \quad}^{I_{2}}$}
\put(328,-5){Fig. \ref{fig: spm2}.4:}
\end{picture}
\caption{Intersection points pictured as $\circ$ and $\mathbb{D}(s_{2}, \varepsilon)$ in gray.}
\label{fig: spm2}
\end{figure}
Let us realize the (refined) real scheme $\mathcal{S}:=\mathcal{J}\sqcup \langle 1 \rangle\sqcup \langle 1 \rangle|0:0:0:0$ in $\mathcal{S}_{DP1}(4,3)$. \\
\textbf{Construction of $\tilde{S}$ and some notation:} Let $\tilde{S} \subset Q$ be a real algebraic maximal curve of bidegree $(3,0)$ with $\mathcal{L}$-scheme $\tilde{\eta}$ as depicted in Fig. \ref{fig: spm2}.1. Let $F_{1},F_{2},F_{3},F_{4}$ 
be the four generatrices of $Q$ in dashed in Fig. \ref{fig: spm2}.1. The $F_{i}$'s intersect $\mathbb{R}\tilde{S}$ in two real points, one of which with multiplicity $2$. Choose a point of $F_{i}$ with coordinates $(\varepsilon_{i}, 0)$ in some affine chart of $Q$, and take the interval
$$I_{i}=\{(x,y): y \in \mathbb{R}, x \in (\varepsilon_{i},\varepsilon_{i+1}) \subset \mathbb{R}\}, \text{ with } i=1,2.$$
Denote with $\mathcal{L}_{I_{i}}$ the set of the generatrices of $Q$ containing the points of $I_{i}$.\\
\textbf{Realization of $\mathcal{S}$}: Let $\eta$ be the real scheme in $\mathbb{R}Q$ depicted in thick black in Fig. \ref{fig: spm2}.4. As explained in Remark \ref{rem: DP1}, the choice of $\eta$ is determined by $\mathcal{S}$ and $\tilde{\eta}$. The realization of $\eta$ by a bidegree $(1,1)$ real curve, is consequent to the following observations and an application of small perturbation method:
\begin{enumerate}[label=(\roman*)]
\item There exists a real curve $Z_{1}$ of bidegree $(1,0)$ constructed as in proof of Proposition \ref{prop: constr_small_pert} (Example \ref{exa: small_pert}) such that:
\begin{itemize}
\item the triplet $(\mathbb{R}Q, \mathbb{R} \tilde{S}, \mathbb{R}Z_{1})$ is arranged as depicted in Fig. \ref{fig: spm2}.2;
\item $(\mathbb{R}\tilde{S} \cap \mathbb{R}Z_{1}) \subset (I_{1}\cup I_{2}).$
\end{itemize}
\item There exists a generatrix $F \subset \mathcal{L}_{I_{2}}$ (in bold in Fig. \ref{fig: spm2}.2) and two points $p,q$ of $\eta$ such that the triplets $(\mathbb{R}Q  \setminus \{p,q\}, \mathbb{R}\tilde{S}, \eta)$ and $(\mathbb{R}Q  \setminus ( \mathbb{R}(Z_{1}\cap F)), \mathbb{R}\tilde{S},\mathbb{R}(F \cup Z_{1}))$ realize the same topological type.
\end{enumerate}
Apply a small perturbation to $Z_{1}\cup F$ as follows.\\
Let $s_{i},w_{i}$ be two distinct points in $(Z_{1} \cap \tilde{S} \cap I_{i})$, belonging to two different connected components of $\tilde{S} \cap I_i$, as depicted in Fig. \ref{fig: spm2}.2, with $i=1,2$.\\
Let $\mathbb{D}(s_{2},\varepsilon)$ be a disk centered in $s_{2}$ with radius $\varepsilon \not =0$, where $\varepsilon$ can be chosen small enough so that $\forall s_{t} \in \mathbb{D}(s_{2},\varepsilon) \setminus \{s_{2} \}$ there exists a real curve $Z_{t}$ of bidegree $(1,0)$ passing through $s_{1}, w_{1}$ and $s_t$ and such that the pairs $(I_{i},\mathbb{R}Z_{t} \cap \mathbb{R}\tilde{S})$ and $(I_{i},\mathbb{R}Z_{1} \cap \mathbb{R}\tilde{S})$ realize the same $\mathcal{L}$-scheme, for $i=1,2$.\\
Let $\tilde{F} \subset \mathcal{L}_{I_{2}}$ be the generatrix in dashed in Fig. \ref{fig: spm2}.3, where it is pictured the topological type realized by the triplet$(I_{2}, \mathbb{R}\tilde{S}, F \cup \tilde{F} \cup Z_{1})$. \\
Let $p_1(x,y)f(x,y)=0$ be a polynomial equation defining the union of $Z_1$ and $F$ in some local affine chart of $Q$. Replace the left side of the equation $p_1(x,y)f(x,y)=0$ with $p_1(x,y)f(x,y)+ \delta p_t(x,y) \tilde{f}(x,y)$, where $p_t(x,y) \tilde{f}(x,y)=0$ defines the union of $Z_{t}$ with $\tilde{F}$ and $\delta \not = 0$ is a sufficient small real number. Up to a choice of $t$ and the sign of $\delta$, one constructs the wanted $Z_{2}$ as a small perturbation of $Z_1 \cup F$. The realization of $\eta$ by $Z_{2}$ implies the realizability of $\mathcal{S}$ in $Y$ and in class $3$.
\end{proof}
\section*{Acknowledgments}
I am very grateful to Erwan Brugallé for his constant support, many fruitful discussions and helpful insights on algebraic geometry and enumerative geometry. I want to thank Eugenii Shustin and Jean-Yves Welschinger for their time, very useful remarks and questions for future directions. I express my appreciation to Stepan Orevkov for interesting discussions. Special thanks to the referees whose comments and corrections have greatly improved the presentation and the content of the paper.
This research is supported by the TMS project ''Algebraic and topological cycles in complex and tropical geometry''.
\bibliographystyle{alpha}
\bibliography{biblio.bib}
Matilde Manzaroli, \textsc{University of Oslo, UiO, Norway}\par\nopagebreak
 \textit{E-mail address}: \texttt{ manzarom'at'math.uio.no}
\newpage
In Table \ref{tabella=realized3}, it is presented a list of all realized real schemes in class $3$ on $k$-sphere real del Pezzo surface of degree $2$, with $1 \leq k \leq 4$; the given labelling refers to the propositions where such real schemes are realized.\\
\begin{table}[h!]
 \begin{threeparttable}
 \caption{ \label{tabella=realized3} \textbf{(Symmetric) real schemes realized in $X^{k}$ and in class $3$}}
\begin{tabular}{ | l | l | l |}
\hline
\textbf{$k=1,2,3,4$}&\textbf{$k=2,3,4$} & \textbf{$k=3,4$} \\
\hline
$8:0:0:0   \quad \circ  \dagger^{*}$&$7:1 :0:0 \quad \dagger  \dagger^{*}$&$6:1:1 :0 \quad \dagger \dagger^{*}$\\
\hline
$1  \sqcup \langle6\rangle :0:0:0 \quad \circ    \circ^{*}$&$1    \sqcup    \langle5\rangle:1 :0:0 \quad \dagger  \dagger^{*}$&$1    \sqcup    \langle4\rangle:1:1 :0 \quad \dagger \dagger^{*}$\\
\hline
$2    \sqcup    \langle5\rangle :0:0:0  \quad \circ  \dagger^{*}$&$2    \sqcup    \langle4\rangle:1 :0:0 \quad \circ    \circ^{*}$&$5:2:1 :0 \quad \dagger \dagger^{*}$\\
\hline
$3    \sqcup    \langle4\rangle  :0:0:0 \quad \circ  \dagger^{*}$&$\langle1\rangle    \sqcup    \langle4\rangle:1 :0:0 \quad \dagger  \dagger^{*}$&$ 1    \sqcup    \langle3\rangle:2:1:0 \quad \dagger \dagger^{*}$\\
\hline
$\langle1\rangle    \sqcup    \langle5\rangle :0:0:0  \quad \circ    \circ^{*}$&$ 3    \sqcup   \langle1\rangle    \sqcup    \langle1\rangle:1 :0:0 \quad \circ    \circ^{*}$&$ 4:3:1 :0\quad \dagger \dagger^{*}$\\
\hline
$\langle2\rangle    \sqcup    \langle4\rangle  :0:0:0  \quad \circ$&$ 6:2 :0:0 \quad \circ   \circ^{*}$&$  1    \sqcup    \langle2\rangle:3:1 :0\quad \dagger  \dagger^{*}$\\
\hline
$1    \sqcup     \langle1\rangle    \sqcup    \langle4\rangle  :0:0:0 \quad \circ  \dagger^{*}$&$1    \sqcup    \langle4\rangle:2 :0:0 \quad \circ  \circ^{*}$&$4:2:2 :0\quad \dagger \dagger^{*} $\\
\hline
$1    \sqcup     \langle2\rangle    \sqcup    \langle3\rangle  :0:0:0 \quad \dagger^{*} $&$ 2    \sqcup    \langle3\rangle:2 :0:0 \quad \circ  \dagger^{*}$&$ 1    \sqcup    \langle2\rangle:2:2:0 \quad \dagger \dagger^{*}$\\
\hline
$2    \sqcup   \langle1\rangle    \sqcup    \langle3\rangle :0:0:0  \quad \dagger^{*}$&$\langle1\rangle    \sqcup    \langle3\rangle:2 :0:0 \quad \circ  \dagger^{*} $&$ 4:\langle\langle1\rangle\rangle:1:0 \quad \dagger \dagger^{*}$\\
\hline
$3    \sqcup   \langle1\rangle    \sqcup    \langle2\rangle :0:0:0  \quad \circ  \dagger^{*}$&$2    \sqcup   \langle1\rangle    \sqcup    \langle1\rangle:2 :0:0 \quad \circ  \dagger^{*} $&$3:3:2:0 \quad \dagger \dagger^{*} $\\
\hline
$4    \sqcup   \langle1\rangle    \sqcup    \langle1\rangle  :0:0:0 \quad \circ  \dagger^{*}$&$ 5:3 :0:0   \quad \dagger    \dagger^{*}$&$\langle\langle1\rangle\rangle:3:2:0 \quad \dagger \dagger^{*}$\\
\hline
$ 1    \sqcup   \langle1\rangle    \sqcup   \langle1\rangle    \sqcup    \langle2\rangle :0:0:0 \quad \dagger^{*} $&$1    \sqcup    \langle3\rangle:3 :0:0 \quad \dagger  \dagger^{*} $&$ $\\ 
\hline
$\langle1\rangle    \sqcup    \langle\langle4\rangle\rangle  :0:0:0 \quad \circ  \dagger^{*}$&$2    \sqcup    \langle2\rangle:3 :0:0 \quad \dagger  \dagger^{*} $& \textbf{$k=4$}\\
\hline
$1    \sqcup   \langle1\rangle    \sqcup    \langle\langle3\rangle\rangle  :0:0:0 \quad \dagger^{*}$&$\langle1\rangle    \sqcup    \langle2\rangle:3 :0:0 \quad \dagger  \dagger^{*}$&$5:1:1:1\quad \dagger  $\\
\hline
$1    \sqcup   \langle3\rangle    \sqcup    \langle\langle1\rangle\rangle  :0:0:0 \quad \dagger^{*} $&$1    \sqcup   \langle1\rangle    \sqcup    \langle1\rangle:3  :0:0 \quad \dagger  \dagger^{*}$&$ 4:2:1:1\quad \dagger$\\
\hline
$ 2    \sqcup   \langle2\rangle    \sqcup    \langle\langle1\rangle\rangle :0:0:0 \quad \dagger^{*} $& $5:\langle\langle1\rangle\rangle :0:0 \quad \dagger  \dagger^{*}$&$ 3:3:1:1\quad \dagger$\\
\hline
$3    \sqcup   \langle1\rangle    \sqcup    \langle\langle1\rangle\rangle  :0:0:0 \quad \circ  \dagger^{*} $&$1    \sqcup    \langle3\rangle:\langle\langle1\rangle\rangle :0:0 \quad \dagger  \dagger^{*}$&$  3:2:2:1\quad \dagger$\\

\hline
$ 1    \sqcup   \langle1\rangle    \sqcup  \langle2   \sqcup  \langle1\rangle\rangle  :0:0:0 \quad \dagger^{*}  $&$4:4 :0:0 \quad \dagger  \dagger^{*}$&$  2:2:2:2\quad \circ$\\
\hline
$ $&$1    \sqcup    \langle2\rangle:4 :0:0   \quad \dagger  \dagger^{*}$&$ $\\
\hline
$ $&$\langle\langle\langle1\rangle\rangle\rangle:4 :0:0 \quad \dagger  \dagger^{*}$&$ $\\
\hline
$ $& $1    \sqcup    \langle2\rangle:1    \sqcup    \langle2\rangle :0:0 \quad \dagger  \dagger^{*}$&$ $\\
\hline
\end{tabular}
        \begin{tablenotes}
        \item 
     \item \textbf{Legend of symbols}:\\
            \item[$\circ$] Symmetric real schemes realized  in $X^{4}$ and in class $3$; Proposition \ref{prop: circ_label_DP2}.            
            \item[$\circ^{*}$] Symmetric real schemes realized  in $X^{k}$ and in class $3$, with $1 \leq k \leq 3$; Proposition \ref{prop: circ_label_DP2}.
            \item[$\dagger$] Real schemes realized  in $X^{4}$ and in class $3$; Proposition \ref{prop: star_label_DP2}.
            \item[$\dagger^{*}$] Real schemes realized in $X^{k}$ and in class $3$, with $1 \leq k \leq 3$; Proposition \ref{prop: star_label_DP2}.

         \end{tablenotes}
     \end{threeparttable}
\end{table}
\par 
For $k \not = 4$, the obstructions in Section \ref{subsec: obstr_DP2} are less restrictive on real curves on $k$-sphere real del Pezzo surface $X^{k}$ of degree $2$. Therefore, in Table \ref{tabellaknot4}, it is given a raw list of real schemes in class $3$ which are still unrealized on the $X^{k}$'s; but, which can not be realized by class $3$ real curves on $4$-sphere real del Pezzo surface of degree $2$.
\begin{table}[h!]
\centering
\caption{\label{tabellaknot4} \textbf{Real schemes in class $3$ in $X^{k}$ for $k \not = 4.$}}
\begin{tabular}{|l|l|}
\hline
$k=3,2,1$&$k=3$\\
\hline
$ \langle 1 \rangle   \sqcup   \langle 1 \rangle   \sqcup   \langle1\rangle   \sqcup   \langle1\rangle :0:0$ & $\langle 1 \rangle   \sqcup   \langle1\rangle :\langle \langle 1  \rangle \rangle :1$\\
\hline
$k=3,2$&$\langle1\rangle   \sqcup   \langle1\rangle :2:2$\\
\hline
$\langle1\rangle   \sqcup   \langle1\rangle : \langle1\rangle   \sqcup   \langle1\rangle :0$&$\langle\langle1\rangle  \rangle :\langle \langle 1 \rangle \rangle :2$\\
\hline
$1   \sqcup   \langle2\rangle : \langle1\rangle   \sqcup   \langle1\rangle :0$&$1   \sqcup   \langle 2 \rangle : \langle\langle 1  \rangle \rangle :1$\\
\hline
$1   \sqcup   \langle1\rangle   \sqcup   \langle1\rangle : \langle \langle1\rangle \rangle :0$&$\langle1\rangle   \sqcup   \langle1\rangle :3:1$\\
\hline
$\langle1\rangle   \sqcup   \langle2\rangle : \langle \langle1\rangle \rangle :0$&$ $\\
\hline
\end{tabular}
\end{table}
\end{document}